\documentclass[reqno, 10pt]{amsart}

\usepackage[usenames,dvipsnames]{color}

\usepackage{amsthm,amsfonts,amssymb,amsmath,amsxtra}%,fdsymbol}
\usepackage[all]{xy}
\SelectTips{cm}{}
\usepackage{xr-hyper}
\usepackage[colorlinks=
   citecolor=Black,
   linkcolor=Red,
   urlcolor=Blue%, backref=page
]{hyperref}
\usepackage{verbatim}

\usepackage[margin=1.45in]{geometry}

\usepackage{mathrsfs}

\RequirePackage{xspace}
\RequirePackage{etoolbox}
\RequirePackage{varwidth}
\RequirePackage{enumitem}
\RequirePackage{tensor}
\RequirePackage{mathtools}
\RequirePackage{longtable}
\RequirePackage{multirow}

\newcommand\reallywidehat[1]{\arraycolsep=0pt\relax%
\begin{array}{c}
\stretchto{
  \scaleto{
    \scalerel*[\widthof{\ensuremath{#1}}]{\kern-.5pt\bigwedge\kern-.5pt}
    {\rule[-\textheight/2]{1ex}{\textheight}} %WIDTH-LIMITED BIG WEDGE
  }{\textheight} % 
}{0.8ex}\\           % THIS SQUEEZES THE WEDGE TO 0.8ex HEIGHT
#1\\                 % THIS STACKS THE WEDGE ATOP THE ARGUMENT
\rule{-1ex}{0ex}
\end{array}
}

\newcommand{\sB}{\ensuremath{\mathscr{B}}\xspace}

\newcommand{\sD}{\ensuremath{\mathscr{D}}\xspace}

\newcommand{\sG}{\ensuremath{\mathscr{G}}\xspace}

\newcommand{\sM}{\ensuremath{\mathscr{M}}\xspace}

\newcommand{\sP}{\ensuremath{\mathscr{P}}\xspace}

\newcommand{\sS}{\ensuremath{\mathscr{S}}\xspace}

\newcommand{\sZ}{\ensuremath{\mathscr{Z}}\xspace}

\newcommand{\fkm}{\ensuremath{\mathfrak{m}}\xspace}

\newcommand{\fkM}{\ensuremath{\mathfrak{M}}\xspace}

\newcommand{\et}{{\text{\rm \'et}}}

\newcommand{\diam}{{\Diamond}}

\newcommand{\BA}{\ensuremath{\mathbb {A}}\xspace}
\newcommand{\BB}{\ensuremath{\mathbb {B}}\xspace}

\newcommand{\BD}{\ensuremath{\mathbb {D}}\xspace}

\newcommand{\BF}{\ensuremath{\mathbb {F}}\xspace}
\newcommand{\BG}{\ensuremath{\mathbb {G}}\xspace}
\newcommand{\BH}{\ensuremath{\mathbb {H}}\xspace}

\newcommand{\BM}{\ensuremath{\mathbb {M}}\xspace}

\newcommand{\BP}{\ensuremath{\mathbb {P}}\xspace}
\newcommand{\BQ}{\ensuremath{\mathbb {Q}}\xspace}
\newcommand{\BR}{\ensuremath{\mathbb {R}}\xspace}

\newcommand{\BW}{\ensuremath{\mathbb {W}}\xspace}
\newcommand{\BX}{\ensuremath{\mathbb {X}}\xspace}

\newcommand{\BZ}{\ensuremath{\mathbb {Z}}\xspace}

\newcommand{\CD}{\ensuremath{\mathcal {D}}\xspace}
\newcommand{\CE}{\ensuremath{\mathcal {E}}\xspace}
\newcommand{\CF}{\ensuremath{\mathcal {F}}\xspace}
\newcommand{\CG}{\ensuremath{\mathcal {G}}\xspace}
\newcommand{\CH}{\ensuremath{\mathcal {H}}\xspace}

\newcommand{\CL}{\ensuremath{\mathcal {L}}\xspace}
\newcommand{\CM}{\ensuremath{\mathcal {M}}\xspace}

\newcommand{\CO}{\ensuremath{\mathcal {O}}\xspace}
\newcommand{\CP}{\ensuremath{\mathcal {P}}\xspace}
\newcommand{\CQ}{\ensuremath{\mathcal {Q}}\xspace}

\newcommand{\CT}{\ensuremath{\mathcal {T}}\xspace}

\newcommand{\CY}{\ensuremath{\mathcal {Y}}\xspace}
\newcommand{\CZ}{\ensuremath{\mathcal {Z}}\xspace}

\newcommand{\ad}{{\mathrm{ad}}}

\DeclareMathOperator{\Aut}{Aut}

\DeclareMathOperator{\diag}{diag}

\DeclareMathOperator{\Gal}{Gal}
\newcommand{\GL}{\mathrm{GL}}

\newcommand{\GO}{\mathrm{GO}}

\let\Im\relax
\DeclareMathOperator{\Im}{Im}

\newcommand{\loc}{\ensuremath{\mathrm{loc}}\xspace}

\newcommand{\red}{\ensuremath{\mathrm{red}}\xspace}

\DeclareMathOperator{\Res}{Res}

\newcommand{\SL}{{\mathrm{SL}}}
\DeclareMathOperator{\Spa}{Spa\,}
\DeclareMathOperator{\Spec}{Spec\,}
\DeclareMathOperator{\Spd}{Spd\,}
\DeclareMathOperator{\Spf}{Spf\,}

\newcommand{\Sp}{{\mathrm{Sp}}}

\newcommand{\wt}{\widetilde}
\newcommand{\wh}{\widehat}

\newcommand{\lps}{[\![}
\newcommand{\rps}{]\!]}

\newtheorem{theorem}{Theorem}
\newtheorem{proposition}[theorem]{Proposition}
\newtheorem{lemma}[theorem]{Lemma}
\newtheorem {conjecture}[theorem]{Conjecture}
\newtheorem{corollary}[theorem]{Corollary}

\theoremstyle{definition}

\newtheorem{definition}[theorem]{Definition}

\newtheorem{remark}[theorem]{Remark}

\newenvironment{altenumerate}
   {\begin{list}
      {\textup{(\theenumi)} }
      {\usecounter{enumi}
       \setlength{\labelwidth}{0pt}
       \setlength{\labelsep}{0pt}
       \setlength{\leftmargin}{0pt}
       \setlength{\itemsep}{\the\smallskipamount}
       \renewcommand{\theenumi}{\roman{enumi}}
      }}
   {\end{list}}
\newenvironment{altitemize}
   {\begin{list}
      {$\bullet$}
      {\setlength{\labelwidth}{0pt}
	   \setlength{\itemindent}{5pt}
       \setlength{\labelsep}{5pt}
       \setlength{\leftmargin}{0pt}
       \setlength{\itemsep}{\the\smallskipamount}
      }}
   {\end{list}}

\numberwithin{equation}{subsection}
\numberwithin{theorem}{subsection}

\setitemize[0]{leftmargin=0.3in,itemsep=\the\smallskipamount}
\setenumerate[0]{leftmargin=0.3in,itemsep=\the\smallskipamount}

\renewcommand{\to}{%
   \ifbool{@display}{\longrightarrow}{\rightarrow}%
   }
\let\shortmapsto\mapsto
\renewcommand{\mapsto}{%
   \ifbool{@display}{\longmapsto}{\shortmapsto}%
   }
\newlength{\olen}
\newlength{\ulen}
\newlength{\xlen}
\newcommand{\xra}[2][]{%
   \ifbool{@display}%
      {\settowidth{\olen}{$\overset{#2}{\longrightarrow}$}%
       \settowidth{\ulen}{$\underset{#1}{\longrightarrow}$}%
       \settowidth{\xlen}{$\xrightarrow[#1]{#2}$}%
       \ifdimgreater{\olen}{\xlen}%
          {\underset{#1}{\overset{#2}{\longrightarrow}}}%
          {\ifdimgreater{\ulen}{\xlen}%
             {\underset{#1}{\overset{#2}{\longrightarrow}}}
             {\xrightarrow[#1]{#2}}}}%
      {\xrightarrow[#1]{#2}}
   }
\makeatother
\newcommand{\xyra}[2][]{%
   \settowidth{\xlen}{$\xrightarrow[#1]{#2}$}%
   \ifbool{@display}%
      {\settowidth{\olen}{$\overset{#2}{\longrightarrow}$}%
       \settowidth{\ulen}{$\underset{#1}{\longrightarrow}$}%
       \ifdimgreater{\olen}{\xlen}%
          {\mathrel{\xymatrix@M=.12ex@C=3.2ex{\ar[r]^-{#2}_-{#1} &}}}%
          {\ifdimgreater{\ulen}{\xlen}%
             {\mathrel{\xymatrix@M=.12ex@C=3.2ex{\ar[r]^-{#2}_-{#1} &}}}
             {\mathrel{\xymatrix@M=.12ex@C=\the\xlen{\ar[r]^-{#2}_-{#1} &}}}}}%
      {\mathrel{\xymatrix@M=.12ex@C=\the\xlen{\ar[r]^-{#2}_-{#1} &}}}%
   }
\makeatletter
\newcommand{\xla}[2][]{%
   \ifbool{@display}%
      {\settowidth{\olen}{$\overset{#2}{\longleftarrow}$}%
       \settowidth{\ulen}{$\underset{#1}{\longleftarrow}$}%
       \settowidth{\xlen}{$\xleftarrow[#1]{#2}$}%
       \ifdimgreater{\olen}{\xlen}%
          {\underset{#1}{\overset{#2}{\longleftarrow}}}%
          {\ifdimgreater{\ulen}{\xlen}%
             {\underset{#1}{\overset{#2}{\longleftarrow}}}
             {\xleftarrow[#1]{#2}}}}%
      {\xleftarrow[#1]{#2}}
   }
\newcommand{\isoarrow}{%
   \ifbool{@display}{\overset{\sim}{\longrightarrow}}{\xrightarrow\sim}%
   }

\newcommand{\quash}[1]{}

 \usepackage{relsize} 
\usepackage[bbgreekl]{mathbbol} 
\usepackage{amsfonts} 
\DeclareSymbolFontAlphabet{\mathbb}{AMSb} %to ensure that the meaning of \mathbb does not change
\DeclareSymbolFontAlphabet{\mathbbl}{bbold}

\newcommand{\LG}{\wh{{L}^+_W\CG}}

 \newcommand{\Ms}{ {\mathcal {M}}^{\rm int}_{/x_0}}
\newcommand{\und}{\underline}

 \newcommand{\br}{\breve}

\newcommand{\Loc}{{\rm {Loc}}}
\newcommand{\Adm}{{\rm {Adm}}}

\newcommand{\into}{\hookrightarrow}
\newcommand{\hook}{\hookrightarrow}
\newcommand{\La}{\Lambda}

\newcommand{\iso}{\xrightarrow{\sim}}

\begin{document}
%------------------------------------------------------
 
\title{On integral local Shimura varieties }
\author[G. Pappas]{Georgios Pappas}
\address{Dept. of Mathematics, Michigan State University, E. Lansing, MI 48824, USA}
\email{pappasg@msu.edu}

\author[M. Rapoport]{Michael Rapoport}
\address{Mathematisches Institut der Universit\"at Bonn, Endenicher Allee 60, 53115 Bonn, Germany} 
\email{rapoport@math.uni-bonn.de}

\date{\today}

\begin{abstract}{We give a construction of \emph{integral local Shimura varieties} which are formal schemes that generalize the well-known  integral models of the Drinfeld $p$-adic upper half spaces. The construction applies to all classical groups, at least for odd $p$. These formal schemes also generalize the formal schemes defined by Rapoport-Zink via moduli of $p$-divisible groups, and are characterized purely in group-theoretic terms. 

More precisely, for a local $p$-adic Shimura datum $(G, b, \mu)$ and a quasi-parahoric group scheme $\CG$ for $G$, Scholze has defined a functor on perfectoid spaces which parametrizes $p$-adic shtukas. He  conjectured that this functor is representable by a normal formal scheme which is locally formally of finite type and flat over $O_{\br E}$. Scholze-Weinstein proved this conjecture  when $(G, b, \mu)$ is of (P)EL type  by using Rapoport-Zink formal schemes. 
 We prove this conjecture for any $(G, \mu)$  of abelian type when  $p\neq 2$, and when  $p=2$ and $G$ is of type $A$ or $C$.
We also relate the generic fiber of this formal scheme to the local Shimura variety, a rigid-analytic space attached by Scholze to $(G, b, \mu, \CG)$. 
}  
\end{abstract}

\maketitle

\tableofcontents

 \section{Introduction}
 
 Recall the  mechanism of Shimura varieties: to a Shimura datum $(G, X)$ (here the first entry is a reductive group over $\BQ$), there is associated a tower of algebraic varieties $( {\rm Sh}_K(G, X)\mid K\subset G(\BA_f))$ over the reflex field $E=E(G, X)$. Here, $K$ runs through the open compact subgroups of $G(\BA_f)$. Furthermore, at least if $(G, X)$ is of PEL-type, there exist integral $p$-adic models if $K$ is of the form $K=K^p K_p$, where $K_p$ satisfies some additional conditions. More precisely, fix a prime number $p$ and a $p$-adic place $v$ of $E$. Then if $K_p\subset G(\BQ_p)$ is the stabilizer of a point in the extended Bruhat-Tits building of $G_{\BQ_p}$, there exists in the PEL-type case an integral model $\sS_K(G, X)$ over $O_{E, v}$.
  The study of these integral models has been  the focus of interest for many years, with spectacular applications in arithmetic.   Furthermore, in recent years such models have  been constructed even for Shimura varieties of abelian type, comp. \cite{KP}, \cite{KZhou}, see \cite{PaICM}. In the present paper, we are concerned with a local $p$-adic analogue of this mechanism. 
  
 We fix a prime number $p$. There is the notion of a local Shimura datum $(G, b, \mu)$ \cite{RV}, a $p$-adic analogue of the notion of a (global) Shimura datum. Here $G$ is a reductive group over $\BQ_p$, and $b\in G(\br \BQ_p)$ and $\mu$ is a conjugacy class of minuscule cocharacters of $G$. Let $E=E(G, \mu)$ be the local reflex field. In \cite{RV}, it is postulated that there should be an associated tower of rigid-analytic varieties $( {\rm Sht}_K(G, b, \mu)\mid K\subset G(\BQ_p))$ over the completion $\br E$ of the maximal unramified extension of $E$. This idea was put into reality by P.~Scholze, see \cite{SchICM}. He defined functors on the category ${\rm Perfd}_k$ of perfectoid spaces over the residue field $k$ of $\br E$ and showed that they are representable by rigid-analytic spaces. Furthermore, if $\CG$ is a smooth model of $G$ over $\BZ_p$, Scholze defines a functor $\CM^{\rm int}_{\CG, b, \mu}$ on ${\rm Perfd}_k$ which he shows to be a $v$-sheaf.  Let $\CG$ be a quasi-parahoric group scheme, in the sense of  \cite[\S 21.5]{Schber}; see \S\ref{def:quasi}. This is a natural class of smooth group scheme models over $\BZ_p$ of $G$. In the analogous function field case, T. Richarz \cite{Ric} has proved that quasi-parahoric group schemes can be characterized as those smooth group scheme models of $G$ for which the corresponding affine Grassmannian is ind-proper. Scholze has conjectured that $\CM^{\rm int}_{\CG, b, \mu}$ is representable by a formal scheme $\sM_{\CG, b, \mu}$ which is normal and flat and locally formally of finite type over $O_{\br E}$. Assuming this conjecture, Scholze-Weinstein \cite{Schber}  show that when $\CG$ is a parahoric group scheme for $G$ and $K=\CG(\BZ_p)$, then $\sM_{\CG, b, \mu}$ is an integral model of ${\rm Sht}_K(G, b, \mu)$ where $K=\CG(\BZ_p)$, i.e., the rigid-analytic generic fiber $\sM_{\CG, b, \mu}^{\rm rig}$ of $\CM^{\rm int}_{\CG, b, \mu}$ can be identified    with ${\rm Sht}_K(G, b, \mu)$. In \cite{Schber} the representability conjecture is proved when $(G, b, \mu)$ is of EL-type, and  in many cases when $(G, b, \mu)$ is of PEL-type, by relating the functor $\CM^{\rm int}_{\CG, b, \mu}$ to Rapoport-Zink formal schemes. The rigid-analytic variety $ {\rm Sht}_K(G, b, \mu)$ is called the \emph{local Shimura variety} (associated to the local Shimura datum $(G, b, \mu)$ and the open compact subgroup $K$ of $G(\BQ_p)$) and the functor $\CM^{\rm int}_{\CG, b, \mu}$, resp. the formal scheme $\sM_{\CG, b, \mu}$  the \emph{integral local Shimura variety} (for the quasi-parahoric group scheme $\CG$ for $G$).

 It is interesting to observe that the ``classical'' approach of \cite{R-Z} proceeds in the reverse way compared to the approach in \cite{Schber}. Namely, when $(G, b, \mu)$ arises from (P)EL-data, then for certain quasi-parahorics $\CG$ for $G$ one constructs a formal scheme $\sM_{\CG, b, \mu}$ by posing a certain moduli problem of $p$-divisible groups with additional structure on the category ${\rm Nilp}_{O_{\br E}}$. Then ${\rm Sht}_{\CG(\BZ_p)}(G, b, \mu)$  is \emph{defined} to be the generic fiber of $\sM_{\CG, b, \mu}$ and the rest of the  tower ${\rm Sht}_K(G, b, \mu)$ is constructed by imposing  level $K$ structures on the $p$-adic Tate module of the generic fiber of the universal $p$-divisible group over $\sM_{\CG, b, \mu}$.  This definition depends a priori on the choice of (P)EL data and  is not  obviously functorial in the triple $(\CG, b, \mu)$.  On the other hand,  this    approach has the dividend that  the structure of the formal schemes $\sM_{\CG, b, \mu}$ can be studied by making use of the  theory of $p$-divisible groups. The model for such a structure result  is Drinfeld's description of his integral model of the Drinfeld $p$-adic halfspace in \cite{Drin}. This approach has been used in a number of problems of arithmetic, e.g. in applications to the Zhang Arithmetic Fundamental Lemma conjecture \cite{Zha} and the Arithmetic Transfer conjecture \cite{RSZ} and to the Kudla-Rapoport Divisor Intersection conjecture \cite{KuRa}. The Rapoport-Zink approach has been generalized to certain Hodge type cases by W. Kim \cite{Kim}, Howard-Pappas \cite{HP}, Hamacher-Kim \cite{HamaKim}, and to some abelian type cases by Shen \cite{Shen}.
 Also, under a mild condition on $b$, there is a purely group-theoretical definition by B\"ultel-Pappas \cite{BP} of the moduli problem on ${\rm Nilp}_{O_{\br E}}$ for $\sM_{\CG, b, \mu}$ for hyperspecial parahoric group schemes $\CG$. In the Hodge type case, it is shown in \cite{BP} that this moduli problem, restricted to Noetherian test rings, is representable by a formal scheme. It coincides with the formal schemes in \cite{Kim}, \cite{HP} and \cite{HamaKim}.
 
 The Scholze-Weinstein approach has the advantage that the input data are purely group-theoretical and that the result is functorial.
  In this way, they are able in certain cases to identify two RZ formal schemes for different (P)EL data which define closely related group-theoretical data. For instance, using this approach  they  prove the conjectures  of Rapoport-Zink \cite{RZdrin} and of Kudla-Rapoport-Zink \cite{KRZ} which postulated such hidden identifications, cf. \cite[\S 25.4-25.5]{Schber}. The downside of this approach is that the global structure of the formal schemes $\sM_{\CG, b, \mu}$ is   harder to study. 
 
 We note that one expects a precise relation  between integral local Shimura varieties and formal completions of global Shimura varieties along isogeny loci of their reduction modulo $p$. This is provided  by the theory of non-archimedean uniformization in the Rapoport-Zink framework; something analogous is conjectured to hold in the general Scholze-Weinstein context, and is known   in the Hodge type case, cf. \cite[Thm. 1.3.3]{PRglsv}.

 In this paper, we are concerned with passing from the (P)EL case to the more general case when the local Shimura datum $(G, b, \mu)$ is of abelian type. Here the definition of this last term is modeled on the case of (global) Shimura varieties, cf. \cite{HPR}, \cite{HLR}.  Namely,  $(G, b, \mu)$ is of abelian type if the associated adjoint local Shimura datum  $(G_\ad, b_\ad, \mu_\ad)$ is isomorphic to the associated adjoint local Shimura datum  $(G_{1, \ad}, b_{1, \ad}, \mu_{1, \ad})$ to a local Shimura datum $(G_1, b_1, \mu_1)$, where  $(G_1, b_1, \mu_1)$ is of Hodge type (i.e., $(G_1,  \mu_1)$ admits an embedding into $(\GL_n,\mu_d)$, where $\mu_d$ is a minuscule coweight of $\GL_n$). In particular, $G_\ad$ is a  classical  group. We prove that Scholze's conjecture holds true when $(G, b, \mu)$ is of abelian type and either $p\neq 2$ or $p=2$ and $G_\ad$ is of type $A$ or $C$. We allow here general quasi-parahoric group schemes, even those outside the class singled out  in \cite[\S25.3]{Schber} (those for which the group $\Pi_\CG$ below is trivial). Allowing general quasi-parahorics  is important for two related reasons:  moduli problems leading to Rapoport-Zink spaces  often correspond to quasi-parahorics,  and allowing quasi-parahorics  is important for  devissage in the proofs.

 We also show that $\sM_{\CG, b, \mu}$ is an integral model of ${\rm Sht}_{K}(G, b, \mu)
$, in the following sense. Let
 \begin{equation}\label{mentionPi}
 \Pi_\CG=\ker ({\rm H}_{\et}^1(\BZ_p, \CG)\to {\rm H}_{\et}^1(\BQ_p, G) ) ,
 \end{equation} 
 a finite abelian group, cf. \cite[25.3]{Schber}. To every $\bar\beta\in \Pi_\CG$, we associate a quasi-parahoric group scheme $\CG_\beta$ for $G$ over $\BZ_p$ such that $\br K_\beta=\CG_\beta(\br\BZ_p)$ is conjugate to $\br K=\CG(\br\BZ_p)$ in $G(\br\BQ_p)$, and we prove that 
 \begin{equation}\label{mentiondec}
 \sM_{\CG, b, \mu}^{\rm rig}\simeq \bigsqcup_{\bar\beta\in\Pi_\CG}{\rm Sht}_{K_\beta}(G, b, \mu) .
 \end{equation}
This formula is reminiscent of the formula of Kottwitz \cite{KotPoints}, according to which the generic fiber of a PEL-moduli scheme in \cite{KotPoints} is a disjoint sum of  copies of Shimura varieties enumerated by $\ker ({\rm H}^1(\BQ, G)\to \prod_v{\rm H}^1(\BQ_v, G) ) $ (these copies are mutually isomorphic in types A and C, but not necessarily in type D). 

In fact, we prove that   the decomposition \eqref{mentiondec} comes by passing to the generic fiber of a decomposition of functors
 \begin{equation}\label{mentionformdec}
 \CM_{\CG, b, \mu}^{\rm int}\simeq \bigsqcup_{\bar\beta\in\Pi_\CG}\CM^{\rm int}_{\CG^o_\beta, b, \mu}/\pi_0(\CG_\beta)^\phi .
 \end{equation}
 Here $\CG^o_\beta$ denotes the parahoric group scheme associated to the quasi-parahoric group scheme $\CG_\beta$ (the neutral  connected component). This formula also gives a reduction of Scholze's conjecture from quasi-parahoric group schemes to parahoric group schemes. 
 
 We know quite a bit about the local structure of the formal scheme $\sM_{\CG, b, \mu}$. Indeed, let $\BM^\loc_{\CG, \mu}=\BM^\loc_{\CG^o, \mu}$ be the local model associated to the parahoric group scheme $\CG^o$ associated to $\CG$.  Here the associated $v$-sheaf on ${\rm Perfd}_k$ is defined in \cite[\S 21]{Schber} and the representability by a weakly normal scheme $\BM^\loc_{\CG^o, \mu}$ flat over $O_{\br E}$ is established in the case of a general local Shimura datum by J.~Ansch\"utz, I.~Gleason, J.~Louren\c co, T.~Richarz in  \cite{AGLR}. 
 In fact, $\BM^\loc_{\CG^o, \mu}$  is always normal with reduced special fiber 
 (by \cite{AGLR} and \cite{GL22} which settled some remaining cases for $p=2$ and $p=3$).   When $p\neq 2$, this local model also coincides for local Shimura data of abelian type with the local model in the style of Pappas and Zhu \cite{PZ} (modified in \cite{HPR}), as extended to groups which arise by restriction of scalars from wild extensions by B.~Levin \cite{Levin}. When $p=2$ and $G_\ad$ is of type $A$ or $C$, this local model also coincides with the local model obtained by taking the closure of the generic fiber in the naive local model of \cite{R-Z}. We also know that, if  $p\neq 2$,  $\BM^\loc_{\CG, \mu}$ is  Cohen-Macaulay \cite{HR2}. Then, under our assumptions above, we prove that for every $x\in  \CM_{\CG, b, \mu}^{\rm int}(k)$, there exists $y\in \BM^\loc_{\CG, \mu}(k)$ and an isomorphism of formal completions
 $$
\CM_{\CG, b, \mu/x}^{\rm int}\simeq \BM^{\rm loc}_{\CG,\mu/y} .
 $$
When $\CG$ is a parahoric group scheme, Gleason (\cite{Gl, Gl21}) has defined the formal completion $\CM^{\rm int}_{\CG, b, \mu/x}$  as a $v$-sheaf. In our approach, we first show that his definition extends to the case of an arbitrary quasi-parahoric group scheme and then show that $\CM^{\rm int}_{\CG, b, \mu/x}$ is representable (by the formal spectrum of a complete Noetherian local ring which is the completion of a corresponding local model, as above). The representability  of $\CM^{\rm int}_{\CG, b, \mu/x}$ for all $x$ is closely related to the representability of $\CM^{\rm int}_{\CG, b, \mu}$. In fact, in \cite{PRglsv}, we show  in the Hodge type case under certain hypotheses that, conversely, if all formal completions  are representable, then so is $\CM^{\rm int}_{\CG, b, \mu}$. Using this statement, we prove in \cite{PRglsv} the representability of $\CM^{\rm int}_{\CG, b, \mu}$ in many cases of Hodge type by using global methods. 

By contrast, the approach in the present paper is purely local and direct,  with more general results. The key case for $p>2$ occurs for $(G,  b,\mu, \CG)$ of Hodge type when the closed immersion $G\hookrightarrow \GL_n$ extends to a closed immersion $\CG\hookrightarrow \GL(\Lambda)$ for a $\BZ_p$-lattice $\Lambda\subset\BQ_p^n$ and satisfies certain additional conditions.  In this case, the representability of $\CM^{\rm int}_{\CG, b, \mu}$ is established by imitating the construction in \cite{KP}, as extended in \cite{KZhou}, \cite{KPZ}, of integral models of global Shimura varieties of Hodge type.  The crucial step here is the construction of a suitable versal deformation of a $p$-divisible group equipped with crystalline tensors. This is done by using Zink's displays and requires the assumption $p>2$. The general Hodge type case is reduced to this case by devissage. In this devissage, there are two steps. In a first step, one shows that the assertion is independent of the quasi-parahoric $\CG$ within the class of all quasi-parahorics sharing a fixed parahoric as their neutral component. In a second step, one shows that the assertion is independent of the group $G$ within the class of all groups sharing the same adjoint group. In these devissage steps one has to deal with affine Deligne-Lusztig varieties for quasi-parahorics, as already considered by U.~G\"ortz, X.~He and S.~Nie in \cite{GHN}. Finally, the general abelian type is reduced to the Hodge type case by following Deligne's analogous reduction \cite{DeligneCorvallis}  in the case of global Shimura varieties.
   
 What remains to be done for the construction of integral local Shimura varieties and Scholze's representability conjecture in all generality? For local Shimura data of abelian type, there are still  cases open  for $p=2$.
Some of these appear accessible but there are several technical complications. Outside the abelian type cases, we have only cases that involve exceptional groups of type $E_6$ or $E_7$,  even orthogonal groups with cocharacters which are  ``mixed'' of type ${D}^{\BR}$ with ${D}^{\BH}$, and trialitarian forms. These are more mysterious.   Especially the cases of exceptional groups are wide open, even though recently the representability of the formal completions $ \CM^{\rm int}_{\CG, b, \mu/x} $ has been proved in the case that $\CG$ is a hyperspecial parahoric by S. Bartling \cite[Thm. 1.4]{Ba} (for $p\geq 3$) and K. Ito \cite[Thm. 5.3.5]{Ito}. 
 
 On the other hand, it is encouraging that the representability of $(\CM^{\rm int}_{\CG, b, \mu})_\red$ is known in general (even when $\mu$ is not minuscule). More precisely, Gleason \cite{Gl21} defines $(\CM^{\rm int}_{\CG, b, \mu})_\red$ as a a
\emph{scheme-theoretic $v$-sheaf} and proves that it is representable by a perfect $k$-scheme $X_\CG(b, \mu^{-1})$ contained in the Witt vector affine Grassmannian $X_\CG=LG/L^+\CG$ (in \cite{Gl21}, Gleason assumes that $\CG$ is parahoric but the result holds for general quasi-parahorics, cf. Proposition \ref{MX}). Furthermore, $X_\CG(b, \mu^{-1})(k)$ can be identified with the corresponding affine Deligne-Lusztig set, cf. \eqref{defADL}. When $\mu$ is minuscule, one expects   a natural deperfection of $X_\CG(b, \mu^{-1})$ as a scheme locally of finite type over $k$ but this seems only known as a consequence of the representability of $\CM^{\rm int}_{\CG, b, \mu}$.

 We thank P.~Scholze for helpful comments, and R.~Zhou for useful discussions about \cite{KZhou}.  The first author also acknowledges support by NSF grant \#DMS-2100743.

 \section{Statements of the main results}\label{statements}

\subsection{Hodge and Abelian type}  Let $(G, b, \{\mu\})$ be a local Shimura datum over $\BQ_p$, cf.  \cite{RV}. Recall that this means that $G$ is a reductive group, that $b\in G(\br\BQ_p)$ and that $\{\mu\}$ is a conjugacy class of a minuscule cocharacter $\mu:\BG_{m/\bar\BQ_p}\to G_{\bar\BQ_p}$. It is assumed that $b$ is neutral acceptable, i.e., the $\sigma$-conjugacy class $[b]$ lies in $B(G, \mu^{-1})$. 
We will often omit the bracket and simply write $(G, b, \mu)$.

We will be concerned with local Shimura data of a particular type. 

\begin{definition}\label{def:hodgetype}
The local Shimura datum $(G, b, \mu)$ is called of  \emph{Hodge type} if there is an embedding
\[
i: (G, \mu)\into (\GL_h, \mu_d),
\]
for some $0\leq d\leq h$, where  $\mu_d(a)=\diag(a^{(d)}, 1^{(h-d)})$ is the standard minuscule cocharacter of $\GL_h$.
\end{definition}

\begin{remark}
This is the local analogue of the corresponding terminology in the classical theory of Shimura varieties, where one requires an embedding into a similitude symplectic group. 
\end{remark}

 \begin{definition}(\cite[\S 2.7]{HPR}, \cite[Def. 9.6]{HLR})\label{def:abtype} The local Shimura datum $(G, b, \mu)$ is called of \emph{abelian type} if there is a local Shimura datum  $(G_1, b_1, \mu_1)$ of Hodge type and an isomorphism $(G_{1,\ad}, b_{1,\ad}, \mu_{1, \ad})\simeq  (G_\ad, b_{\ad}, \mu_\ad)$. In this case,  
 $ (G_1, b_1, \mu_1)$ is called a \emph{central lift of Hodge type} for $(G, b, \mu)$, cf. \cite[\S 9]{HLR}. 
 
\smallskip

 Note that the existence of $b_1$ that lifts $b_{\ad}$ is automatic since
 $B(G,\mu)\simeq B(G_{\ad},\mu_\ad)$ by \cite[(6.5.1)]{KotIsoII}, so this is really a property of the pair $(G, \mu)$. 
 \end{definition}
 \begin{remark}
 Note that this is weaker than the analogous notion in the global case, where one asks that there is a  morphism $(G_1, \mu_1)\to (G_\ad, \mu_\ad)$ such that the  induced morphism $G_{1, {\rm der}}\to G_\ad$   admits a factorization through $G_{\rm der}$ and induces an isomorphism $(G_{1,\ad},  \mu_{1, \ad})\simeq  (G_\ad,  \mu_\ad)$, comp. \cite{KP}. In \cite{HPR}, it is this stronger version that is imposed also in the local case; however, in the local case, this stronger notion seems unnecessary, as pointed out in \cite{HLR}. 
 \end{remark}
 
 \subsection{Quasi-parahoric subgroups}\label{def:quasi} Let $\br G$ be a reductive group over $\br\BQ_p$. By definition, a quasi-parahoric  subgroup  of $\br G(\br\BQ_p)$ is a subgroup $\br K$ which is squeezed as
 $$
 \br G(\br \BQ_p)^0\cap {\rm Stab}_{\frak F}\subset \br K\subset \br G(\br \BQ_p)^1\cap {\rm Stab}_{\frak F} .
 $$
 Here ${\rm Stab}_{\frak F}$ is the stabilizer of  a facet $\frak F$ in the building  $\sB(\br G_\ad, \br\BQ_p)$, and   
 $$
 \br G(\br\BQ_p)^0=\ker(\kappa: \br G(\br\BQ_p)\to \pi_1(\br G)_I), \quad \br G(\br\BQ_p)^1=\ker(\kappa: \br G(\br\BQ_p)\to \pi_1(\br G)_I\otimes\BQ) .
 $$
Here, $\kappa$ is the Kottwitz homomorphism. Equivalently, it is a subgroup of finite index in $\br G(\br\BQ_p)^1\cap {\rm Stab}_{{\frak F}}$.
Still another way of characterizing quasi-parahoric subgroups is to  say that they are of finite index in the stabilizer of a point in the extended building $\sB^e(\br G, \br\BQ_p)$.  In the case that the quasi-parahoric subgroup coincides with $\br G(\br\BQ_p)^1\cap {\rm Stab}_{\bf x}$, with ${\bf x}\in \sB(\br G_\ad, \br\BQ_p)$ (equivalently, if it coincides with the stabilizer subgroup in $\br G(\br\BQ_p)$ of a point in the extended building $\sB^e(\br G, \br\BQ_p)$), it is called a \emph{stabilizer quasi-parahoric}. The parahoric subgroup $\br K^o=\br G(\br \BQ_p)^0\cap {\rm Stab}_{\frak F}$ is called the parahoric subgroup \emph{associated} to the quasi-parahoric subgroup $\br K$. Note that if $\pi_1(\br G)_I$ is torsion-free, then any quasi-parahoric is a parahoric. 
 
  By Bruhat-Tits theory \cite{BT2}, there is a unique smooth group scheme $\br\CG$  over $\br\BZ_p$ with generic fiber $\br G$ such that $\br\CG(\br\BZ_p)=\br K$. 
  
 Now let $G$ be a reductive group over $\BQ_p$. Then,  in analogy with the Bruhat-Tits definition of a $\BQ_p$-parahoric subgroup \cite[right after Def. 5.2.6]{BT2},  there is the notion of a $\BQ_p$-quasi-parahoric subgroup  of $G$. It can be equivalently defined as  a quasi-parahoric subgroup $\br K$ of $G(\br\BQ_p)$ which is invariant under Frobenius, or as a subgroup of finite index in $G(\br\BQ_p)^1\cap {\rm Stab}_{{\frak F}}$,  where  $ {\rm Stab}_{\frak F}$ is the stabilizer of 
    a facet $\frak F$ in the building  $\sB(G_\ad, \BQ_p)$.   By descent from $\br\BZ_p$ to $\BZ_p$, we have  a corresponding smooth group scheme $\CG$ over $\Spec(\BZ_p)$.   Then 
 $\CG(\BZ_p)=\br K\cap G(\BQ_p)$. 
 
 In the sequel, we simply call $\CG$ a \emph{quasi-parahoric group scheme}  for $G$. Note that a parahoric group scheme for $G$ is uniquely defined by its associated subgroup $\CG(\BZ_p)$ of $G(\BQ_p)$, cf. \cite[Prop. 5.2.8]{BT2}. However, for  general quasi-parahoric group schemes for $G$, the association $\CG\mapsto \CG(\BZ_p)$ ceases to be injective  (for instance $\CG^o(\BZ_p)= \CG(\BZ_p)$ when $\pi_0(\CG)^\phi$ is trivial, which may happen even when $\pi_0(\CG)$ is non-trivial).    Still, we write $K=\CG(\BZ_p)=\br K\cap G(\BQ_p)$ and sometimes use the notation $K$  to refer to our choice of $\CG$. This does not lead to confusions.

\subsection{The local model } \label{ss:LM}

  Recall\footnote{In fact, in  \cite[\S 20]{Schber} only the case where $\CG$ is reductive is considered; in  \cite[\S 21]{Schber} the case where $\CG$ is an arbitrary parahoric and where $\mu$ is minuscule is considered. The general case is treated in  \cite[\S 4.2]{AGLR}.} the Scholze-Weinstein $v$-sheaf local model $\BM^v_{\CG,\mu}$ over $\Spd(O_E)$ (in \cite{Schber} it is denoted by ${\rm Gr}_{\CG, \Spd(O_E),\leq \mu}$).  

Assume $\mu$ is minuscule. In this case, 
${\rm Gr}_{\CG, \Spd(O_E),\leq \mu}={\rm Gr}_{\CG, \Spd(O_E), \mu}$ and the $v$-sheaf local model $\BM^v_{\CG,\mu}={\rm Gr}_{\CG, \Spd(O_E), \mu}$
is given as the closure inside the Beilinson-Drinfeld style affine Grassmannian ${\rm Gr}_{\CG, \Spd(O_E)}$ of the $v$-sheaf $X_\mu^\diam$ associated to the symmetric space $X_\mu={\rm Gr}_{G,\mu}$ of parabolics of type $\mu$ over $\Spec(E)$. 
 Then $\BM^v_{\CG,\mu}$ is representable by a   normal flat projective $O_E$-scheme $\BM^\loc_{\CG,\mu}$ with reduced special fiber.
This was conjectured by Scholze-Weinstein \cite[Conj. 21.4.1]{Schber} and was shown in \cite{AGLR} and \cite{GL22} (which settled the normality in some remaining cases in characteristics $2$ and $3$). Here $\BM^\loc_{\CG,\mu}$ is uniquely determined. 
Assume in addition that $(G, \mu)$ is of abelian type and satisfies Condition (A) or (B), cf. Definition \ref{def:acc} below. Then the normality of $\BM^\loc_{\CG,\mu}$ is given by \cite[Thm. 7.23]{AGLR}. Furthermore, in the case (A)   $\BM^\loc_{\CG,\mu}$ can be obtained by the procedure of \cite{PZ}, extended to restrictions of scalars from wild extensions in \cite{Levin} and as modified in \cite{HPR} when $p\mid |\pi_1(G_{\rm der})|$. This statement uses Remark \ref{acc}.  In the case (B) of  Definition \ref{def:acc}, $\BM^\loc_{\CG,\mu}$ can also be obtained by taking the closure of the generic fiber in the naive local model of \cite{R-Z}.
There is no difference when discussing quasi-parahorics because the natural map is an isomorphism, 
\begin{equation}
\label{locmodGtoG0}
\BM^{v}_{\CG^o, \mu}\isoarrow  \BM^{v}_{\CG, \mu},
\end{equation}
cf. \cite[Prop. 21.4.3]{Schber}. (This proposition assumes that $\mu$ is minuscule, but the argument applies to general $\mu$, see also \cite[\S4]{AGLR}). If $\mu$ is minuscule, this also yields an isomorphism for the corresponding scheme local models, 
\[
\BM^\loc_{\CG^o, \mu}\isoarrow  \BM^\loc_{\CG, \mu}.
\]

\subsection{The integral local Shimura variety $\CM^{\rm int}_{\CG, b, \mu}$}\label{ss:inlocshim}

  Let $\CG$ be a smooth affine group scheme over $\BZ_p$ with generic fiber a reductive group $G$, let $b\in G(\breve \BQ_p)$, and let $\mu$ be a conjugacy class of cocharacters of $G$. It is assumed that the Frobenius-conjugacy class of $b$ lies in $B(G, \mu^{-1})$. We call the triple $(\CG, b, \mu)$ an \emph{integral local shtuka datum} and $(G, b, \mu)$ a \emph{rational local shtuka datum}. As usual, we denote by $E$ the field of definition of $\mu$ and by $\breve E$ the completion of the maximal unramified extension. 
  
Then Scholze-Weinstein associate to  $(\CG, b, \mu)$   an ``integral" moduli space of shtukas $\CM^{\rm int}_{\CG, b, \mu}$ over $O_{\breve E}$, \cite[Def. 25.1.1]{Schber}. It is given as a ``$v$-sheaf moduli space" of certain $\CG$-shtukas with one leg bounded by $\mu$ with a fixed associated Frobenius element, cf.  \cite[\S\S 23.1, 23.2, 23.3]{Schber}, as follows.

 \begin{definition}\label{defintSht} 
  The \emph{integral moduli 
  space of local shtuka} $\CM^{\rm int}_{\CG, b, \mu}$ is
the  functor  that sends $S\in {\rm Perfd}_k$ to the  set of isomorphism classes of tuples 
 \begin{equation*}
 (S^\sharp,  \sP , \phi_{\sP }, i_r) ,
 \end{equation*}
 where
 \begin{itemize}
 
\item[1)] $S^\sharp$ is an untilt of $S$ over $\Spa( O_{\br E})$,
\item[2)] $(\sP , \phi_{\sP })$ is a $\CG$-shtuka over $S$ with one leg along $S^\sharp$ bounded 
by $\mu$,
\item[3)] $i_r$ is a ``framing", i.e. an isomorphism of $G$-torsors
\begin{equation}
i_r: G_{\CY_{[r,\infty)}(S)}\xrightarrow{\sim} \sP_{\, |\CY_{[r,\infty)}(S)}
\end{equation}
for large enough $r$ (for an implicit choice of pseudouniformizer $\varpi$), under which
$\phi_{\sP }$ is identified with $\phi_b=b\times {\rm Frob}_S$.
 \end{itemize}
 \end{definition}
 Here, $\CY_{[r,\infty)}(S)$ is as defined in \cite{Schber}, \cite[II]{FS}. We have denoted by $G_{\CY_{[r,\infty)}(S)}$ the trivial $G$-torsor over ${\CY_{[r,\infty)}(S)}$ (denoted $G\times {\CY_{[r,\infty)}(S)}$ in \cite[App. to \S 19]{Schber}), and by 
 \begin{equation}\label{phi_G}
 \phi_b=\phi_{G, b}\colon \phi^*(G_{\CY_{[r,\infty)}(S)})=G_{\CY_{[\frac{1}{p}r,\infty)}(S)}\xrightarrow{b\phi} G_{\CY_{[r,\infty)}(S)}
 \end{equation}
   the  isomorphism given by the $\phi$-linear isomorphism induced by right multiplication by $b$ (denoted by $b\times {\rm Frob}$ in \cite[Def. 23.1.1]{Schber}).
In 3) we mean more precisely an equivalence class, where $i_r$ and $i'_{r'}$ are called equivalent if there exists $r''\geq r, r'$ such that ${i_r}_{\, |\CY_{[r'',\infty)}(S)}={i'_{r'}}_{\, |\CY_{[r'',\infty)}(S)}$. 
 Also, in 2) the precise definition of ``bounded by $\mu$" is given via the local model $\BM^v_{\CG, \mu}$ (see \cite[Def. 25.1.1]{Schber}, \cite[p. 46]{PRglsv}).

 Note that the formation of $\CM^{\rm int}_{\CG, b,\mu}$ is functorial, i.e., for a morphism $(\CG, b, \mu)\to (\CG', b', \mu')$ we have $\CM^{\rm int}_{\CG, b,\mu}\to \CM^{\rm int}_{\CG', b',\mu'}\times_{\Spd(O_{\br E})} \Spd(O_{\br E'})$. 
 It is also compatible with products, i.e., 
$
\CM^{\rm int}_{\CG_1\times \CG_2, b_1\times b_2,\mu_1\times \mu_2}$ is isomorphic to the product 
$\CM^{\rm int}_{\CG_1, b_1,\mu_1}\times \CM^{\rm int}_{\CG_2, b_2,\mu_2}
$
after base changing to the compositum of the reflex fields (in this, we omit the base changes from the notation, for simplicity).  
 In addition, the $\phi$-centralizer group $J_b(\BQ_p)$ acts on 
$\CM^{\rm int}_{\CG, b,\mu}$ by changing the framing.

 By loc. cit.,   $\CM^{\rm int}_{\CG, b, \mu}$ is a $v$-sheaf over $\Spd( O_{\br E})$.  In fact, as in \cite[Prop. 2.23]{Gl21},  $\CM^{\rm int}_{\CG, b, \mu}$ is a small $v$-sheaf. 
 In addition,  $\CM^{\rm int}_{\CG, b,\mu}$
supports a Weil descent datum from $\CO_{\br E}$ down to $O_E$, as explained in \cite[\S 3.1, \S 3.2]{PRglsv}. In this paper, even though this is not always pointed out explicitly,  the smooth group scheme $\CG$ will always be a quasi-parahoric group scheme for $G$. 

In fact, most of the time we are interested in the case that $(G, b, \mu)$ is a local Shimura datum, i.e. we also assume that $\mu$ is minuscule.
In addition, we take $\CG$ to be a quasi-parahoric group scheme for $G$.   In this case we call ${\CM^{\rm int}_{\CG, b, \mu}}$ the \emph{integral local Shimura variety} associated to the local Shimura datum $(G, b, \mu)$ and the quasi-parahoric group scheme $\CG$.

 \subsection{Statement of the main results}
 
 Our concern is with the following conjecture. 
  
\begin{conjecture}{\rm (Scholze)}\label{repconj}
Let $(G, b, \mu)$ be a local Shimura datum and $\CG$ a quasi-parahoric group scheme for $G$. 
 There exists  a formal scheme $\sM_{\CG, b, \mu}$ which is normal  and flat locally formally of finite type over $O_{\br E}$ with
 \[
 \CM^{\rm int}_{\CG, b, \mu}=\sM_{\CG, b, \mu}^\diam,
 \]
 as $v$-sheaves over $\Spd(O_{\br E})$. Then, $\sM_{\CG, b, \mu}$ is unique (\cite[Prop. 18.4.1]{Schber}).
 \end{conjecture}
  Our main result is a proof of this conjecture when $(G, b, \mu)$ is of abelian type under certain mild hypotheses.
 \begin{definition}\label{def:acc} Let $(G, b, \mu)$ be of abelian type. We introduce two kinds of conditions on $p$, $G$ and $\mu$. 
 \begin{itemize}
 \item[(A)]  $p\neq 2$. 

\item[(B)] $p=2$ and   $G_\ad=\prod_i {\rm Res}_{F_i/\BQ_p}H_i$, where for each $i$, $H_i= B_i^\times/F_i^\times$ 
with $B_i$ a simple algebra with center $F_i$, or 
$H_i={\rm PGSp}_{2n_i}$, or the corresponding component $\mu_{\ad, i}$ of $\mu_{\ad}$ is trivial.
 \end{itemize}
 Note that if $G=T$ is a torus, $T$ trivially satisfies (A) or (B). 
  \end{definition}

\begin{remark}\label{acc}
Recall from \cite[\S 5]{PRglsv}, comp. \cite[Def. 3.1.4]{KPZ}, that $G$ is called  essentially tamely ramified 
if $G_\ad\simeq\prod_i {\rm Res}_{F_i/\BQ_p}H_i$, where $H_i$ is absolutely simple and splits over a tamely ramified extension of $F_i$. Note that   $G$ is automatically   essentially tamely ramified   if $p\geq 5$, cf. \cite[\S 5]{PRglsv}. Moreover, if $p=3$, the   condition of essentially tame ramification 
 only excludes groups whose adjoint groups contain a ramified triality group (type $D_4^{(3)}$ or $D_4^{(6)}$) as
a factor. We note that by Serre \cite[\S 3, Cor. 2]{Serre} this last possibility does not occur when $(G, b, \mu)$ is of abelian type, unless the corresponding component of $\mu_{\rm ad}$ is trivial. In other words, Condition (A) implies that the simple factors of $G_{\rm ad}$ on which the projections of $\mu_{\rm ad}$ are non-trivial, are  essentially tamely ramified. Indeed, let $(G_1, \mu_1)$ be a central lift of Hodge type for $(G, \mu)$. Let $G_1'$ be the minimal normal subgroup of $G_1$ containing the conjugacy class $\mu_1$. Then by  \cite[\S 3, Cor. 2]{Serre}, $G'_{1,\ad}$ contains no ramified triality group as a factor. On the other hand, $G'_{1,\ad}=\prod_{\{i\mid \mu_{\ad, i}\neq 1\}} {\rm Res}_{F_i/\BQ_p}H_i$, which proves the claim.
\end{remark}

 \begin{theorem}\label{MainThm}  Let
$(G, b, \mu)$ be a local Shimura datum of abelian type satisfying Condition (A) or (B), and let $\CG$ be a quasi-parahoric group scheme for $G$. Then $\CM^{\rm int}_{\CG, b, \mu}$ satisfies Conjecture \ref{repconj}.
\end{theorem}

Theorem \ref{MainThm} is closely related to the following result. In fact,  it was shown in \cite{PRglsv} that the following result implies Theorem \ref{MainThm}  in the Hodge type case under some mild additional hypotheses, cf. \cite[proof of Thm. 3.7.1]{PRglsv}.
\begin{theorem}\label{thmRepgoal} Let
$(G, b, \mu)$ be a local Shimura datum of abelian type satisfying Condition (A) or (B),  and let $\CG$ be a quasi-parahoric group scheme for $G$. For any   $\Spd(k)$-valued point $x$ of  $\CM^{\rm int}_{\CG, b, \mu}$, there is a  
 $k$-valued point $y$
of $\BM^{\rm loc}_{\CG,\mu}$ such that
   \begin{equation}\label{Gliso}
   \CM^{\rm int}_{\CG, b, \mu/x} \simeq (\BM^{\rm loc}_{\CG,\mu/y})^\diam,
   \end{equation}
   as $v$-sheaves over $\Spd(O_{\br E})$.
   \end{theorem}
In this, $y$ is  a $k$-valued point
of $\BM^{\rm loc}_{\CG,\mu}$ obtained by fixing a trivialization of the $\CG$-torsor underlying the $\CG$-shtuka at $x$ (more precisely, $y$ is in the orbit  $\ell(x)$ given by (\ref{Lmap})).
  The notation $\CM^{\rm int}_{\CG, b, \mu/x}$ stands for the {\sl formal completion} of the $v$-sheaf
  $\CM^{\rm int}_{\CG, b, \mu}$ at $x$, as defined in \cite[Def. 4.18]{Gl}, cf. \cite[\S 3.3.1]{PRglsv}, 
  see also \S\ref{def:spec}.
 On the RHS, $(\BM^{\rm loc}_{\CG,\mu/y})^\diam$ denotes the $v$-sheaf attached to the adic space
$\Spa(A, A)$, where $A$ is the completion of the local ring of the scheme $\BM^{\rm loc}_{\CG, \mu}$ at
$y$, taken with the topology defined by the maximal ideal.

The combination of these two results also gives

\begin{corollary}
Let
$(G, b, \mu)$ be a local Shimura datum of abelian type satisfying Condition (A) or (B), and let $\CG$ be a quasi-parahoric group scheme for $G$.
 Let 
$\sM_{\CG, b, \mu}$ be the formal scheme that represents $\CM^{\rm int}_{\CG, b, \mu}$ by Theorem \ref{MainThm}. 
For any   $k$-valued point $x$ of $\sM_{\CG, b, \mu}$, there is 
a  
 $k$-valued point $y$
of $\BM^{\rm loc}_{\CG,\mu}$ such that
 $
   \sM_{\CG, b, \mu/x} \simeq \BM^{\rm loc}_{\CG,\mu/y}.
$
 \end{corollary}
  Note here that, by the full-faithfulness result of \cite[Prop. 18.4.1]{Schber}, the Weil descent datum for the $v$-sheaf $\CM^{\rm int}_{\CG, b,\mu}$ gives a corresponding  Weil descent datum  for 
the formal scheme $\sM_{\CG, b, \mu}$, from $\CO_{\br E}$ down to $O_E$. 

The following result shows that $\sM_{\CG, b, \mu}$ is an integral model of the local Shimura variety in a certain sense. Recall from the introduction that to $\CG$ is associated the finite abelian group $\Pi_\CG$, and  that to every $\bar\beta\in\Pi_\CG$, there is associated a quasi-parahoric group scheme $\CG_\beta$ over $\BZ_p$, as well as its associated parahoric group scheme $\CG^o_\beta$, and the finite group of its connected components $\pi_0(\CG_\beta)$ with its action by the Frobenius $\phi$. Also we let $K_\beta=\CG_\beta(\BZ_p)$. 
\begin{theorem} \label{decompintro}Let
$(G, b, \mu)$ be a local Shimura datum of abelian type satisfying Condition (A) or (B), and let $\CG$ be a quasi-parahoric group scheme for $G$. There is an isomorphism of rigid-analytic varieties over $\br E$,
$$
 \sM_{\CG, b, \mu}^{\rm rig}\simeq \bigsqcup_{\bar\beta\in\Pi_\CG}{\rm Sht}_{K_\beta}(G, b, \mu) .
$$
This isomorphism is induced by an isomorphism of formal schemes
$$
 \sM_{\CG, b, \mu}\simeq \bigsqcup_{\bar\beta\in\Pi_\CG}{\sM}_{\CG^o_\beta, b, \mu}/\pi_0(\CG_\beta)^\phi .
$$

\end{theorem}

 We finally mention the following result. Let
$(G, b, \mu)$ be a  local shtuka datum,
 and let $\CG$ be a quasi-parahoric group scheme for $G$. Consider the scheme-theoretic $v$-sheaf $(\CM^{\rm int}_{\CG, b, \mu})_\red$, cf. \cite{Gl21}. Then $(\CM^{\rm int}_{\CG, b, \mu})_\red$  is representable by a perfect $k$-scheme $X_\CG(b, \mu^{-1})$  which can be identified with the admissible set  \eqref{defADL} contained in the Witt vector affine Grassmannian $X_\CG=LG/L^+\CG$, cf. Proposition \ref{MX}. The problem of determining the set of connected components of  $X_\CG(b, \mu^{-1})$ is of considerable interest. 
  
  Let $\CG$ be a parahoric. Then there is a surjective map 
  \begin{equation}\label{mapX}
  X_\CG(b, \mu^{-1})\to c_{b, \mu}+\pi_1(G)_I^\phi ,
  \end{equation}
  where $c_{b, \mu}\in \pi_1(G)_I$ is a certain element unique up to $\pi_1(G)_I^\phi$, cf. \cite[Lem. 6.1]{HZ}. By a recent result of Gleason-Lim-Xu \cite[Thm. 1.2]{GLX},  the fibers of \eqref{mapX} are the connected components of $X_\CG(b, \mu^{-1})$ 
  when $(b, \mu)$ is Hodge-Newton irreducible (proving a conjecture of X. He). 
   The map \eqref{mapX}  is equivariant for the natural action of the $\phi$-centralizer group $J_b(\BQ_p)$ on source and target. We prove the following result.
   \begin{proposition}\label{actJintro}
 Assume that the center of $G$ is connected. Then the action of $J_b(\BQ_p)$ on $c_{b, \mu}+\pi_1(G)_I^\phi $ is transitive. In particular,  any two fibers of the map \eqref{mapX}  are isomorphic. 
 \end{proposition}  
 Composing  \eqref{mapX}  with the specialization map (\ref{specM}) gives  a $J_b(\BQ_p)$-equivariant map
   \begin{equation}\label{mapX2}
\CM^{\rm int}_{\CG, b, \mu}\to \underline{c_{b, \mu}+\pi_1(G)_I^\phi},
  \end{equation} 
  where the target is the corresponding constant $v$-sheaf. Hence, if the center of $G$ is connected, 
  any two fibers of the map \eqref{mapX2}  are also isomorphic.

 \subsection{The plan of the proof} 

The proof of the main theorems proceeds in the following steps.
\begin{altitemize}
\item {\sl Step 1.} We develop a devissage procedure that allows us to pass  

a) between $\CG$ and $\CG^o$.

b) between groups linked via an ad-isomorphism.

In most of the statements in this step, we deal with integral local shtuka data $(\CG, b, \mu)$, i.e. we do not assume that $\mu$ is minuscule.

\item {\sl Step 2.} We show the results in ``good" Hodge type cases: 
In such cases, the proof follows from the case of $\GL_n$ (already treated in \cite{Schber} 
by relating to RZ formal schemes) and from
the constructions in \cite{KP}, \cite{KZhou}, \cite{KPZ} (which use Zink displays and Breuil-Kisin modules).

\item {\sl Step 3.}  Using Step 1 we reduce the general case to:

a) the Hodge type cases handled in Step 2, (in case (A)),

b) EL/PEL cases for which we give directly $\sM_{\CG,b,\mu}$ as a Rapoport-Zink formal scheme, (in case (B)).
\end{altitemize}

\subsection{The lay-out of the paper} In \S \ref{s:prelim}, we generalize the Ansch\"utz purity theorem from parahoric group schemes to quasi-parahoric group schemes, and use this to transfer to this more general context the  formalism of specialization, formal completion and $v$-sheaf local model diagram that Gleason had established for hyperspecial parahoric $\CG$. We also make the link between the reduced locus of $ \CM^{\rm int}_{\CG, b, \mu}$ and affine Deligne-Lusztig varieties for quasi-parahorics. In \S \ref{s:parvs}, we study the devissage step of varying the quasi-parahorics with a given associated parahoric. We also prove at this point Theorem \ref{decompintro}. In \S \ref{s:adiso}, we study the devissage step of varying the groups  with a given adjoint group. In \S \ref{s:strivmu} we treat the case when $\mu_\ad$ is trivial. In \S \ref{s:cruc} we consider the crucial Hodge type case mentioned in the introduction. In \S \ref{s:exHE}, we develop methods to relate Hodge type cases and abelian type cases to the crucial Hodge type case. Everything comes together in \S \ref{s:proofs}, where we give the proofs of Theorems \ref{MainThm} and \ref{thmRepgoal}; this is done separately for $G$  satisfying Condition  (A) and  (B).

\section{Preliminaries}\label{s:prelim}

\subsection{Torsors under quasi-parahoric group schemes}

Let $\Omega=\Omega_G=\pi_1(G)_I$ (coinvariants of the \emph{algebraic fundamental group} under the inertia group, cf., e.g., \cite[\S 3]{PRTwisted}). We recall the Kottwitz homomorphism 
$$
\kappa\colon G(\br\BQ_p)\to \Omega_G .
$$
A parahoric subgroup lies in the kernel of $\kappa$ and in fact, the kernel of $\kappa$ is generated by all parahoric subgroups, cf. \cite[Lem. 17]{HR}. Also, for any quasi-parahoric subgroup $\br K$ of $G(\br\BQ_p)$ with associated parahoric subgroup $\br K^o$ we have $\br K^o=\ker (\kappa_{| \br K}\colon \br K\to \Omega)$. Let $\CG$ be the corresponding quasi-parahoric group scheme and consider the injective map 
\begin{equation} \label{pi0inO}
  \pi_0(\CG)=\br K/\br K^o\hookrightarrow \Omega_G.
 \end{equation}
 In \cite[25.3.1]{Schber}, Scholze-Weinstein point out the importance of the finite abelian group,
  \begin{equation}\label{defPi}
 \Pi_\CG:={\rm ker}(\pi_0(\CG)_\phi\to \Omega_\phi)
 \end{equation}
(coinvariants under the action of $\phi$).
 
 \begin{lemma}\label{lemma311} There is a natural isomorphism
\[
\Pi_\CG \cong \ker({\rm H}^1_{\et}(\BZ_p, \CG)\to {\rm H}^1_{\et}(\BQ_p, G)).
\]
\end{lemma}

\begin{proof}
This is sketched in \cite[25.3]{Schber}. The exact sequence $0\to \br K^o\to \br K\to \pi_0(\CG)\to 0$ induces an exact sequence
\[
{\rm H}^1_{\et}(\BZ_p, \CG^o)\to {\rm H}^1_{\et}(\BZ_p, \CG)\to {\rm H}^1_{\et}(\BZ_p, \pi_0(\CG)) .
\]
Here the left term vanishes  by Lang's theorem. The second map is surjective since ${\rm H}^1_{\et}(\BZ_p, \CG)=\br K/_{\!\!\phi}\br K$. We therefore obtain
$${\rm H}^1_{\et}(\BZ_p, \CG)\simeq {\rm H}^1_{\et}(\BF_p, \pi_0(\CG))\simeq {\rm H}^1(\langle \phi\rangle, \pi_0(\CG))\cong \pi_0(\CG)_\phi.
$$
On the other hand, we have an injection
\begin{equation}\label{arrowH1}
{\rm H}^1_{\et}(\BQ_p, G)\into {\rm H}^1(\langle \phi\rangle, G(\br\BQ_p))\cong B(G)
\end{equation}
with the image landing in the basic elements $B(G)_{\rm basic}\subset B(G)$, cf. \cite{KotIsoII}. Now
\[
B(G)_{\rm basic}\simeq \pi_1(G)_\Gamma=(\pi_1(G)_I)_\phi=\Omega_\phi,
\]
and the map \eqref{arrowH1} induces  an identification 
\[
{\rm H}^1_{\et}(\BQ_p, G)\simeq \Omega_{\phi, {\rm tors}} ,
\]
where $\Omega_{\phi, {\rm tors}}$ is the torsion subgroup of $\Omega_\phi$. It remains to observe that the following diagram is commutative,  
\begin{equation}\label{diag}
\begin{aligned}
 \xymatrix{
      {\rm H}^1_{\et}(\BZ_p, \CG)  \ar[r] \ar[d]_\simeq &  {\rm H}^1_{\et}(\BQ_p, G)\,\ar@{^{(}->}[r] \ar[d]^\simeq & B(G)_{\rm basic}\ar[d]^\simeq\\
        \pi_0(\CG)_\phi\ar[r]  &  \Omega_{\phi, {\rm tors}} \,\ar@{^{(}->}[r]&  \Omega_{\phi}.
        }
        \end{aligned}
\end{equation}
 This commutativity follows  from the fact that the vertical homomorphisms on the outer left and the outer right are both induced by the Kottwitz homomorphism $G(\br\BQ_p)\to \Omega$: this is obvious from the construction for the outer left map and follows from \cite[\S 7.5]{KotIsoII} for the outer right map.
\end{proof}

\subsection{Purity of torsors}

The following purity/extension result is crucial. It is a quick generalization of the corresponding result 
 for parahoric subgroups due to Ansch\"utz (\cite{An}).
 
 \begin{proposition}\label{Anext}
 Suppose   $\CG$ is a quasi-parahoric group scheme over $\BZ_p$, and consider a pair $(C, C^+)$ with $C$ an algebraically closed non-archimedean complete field
 of characteristic $p$ and $C^+$ an open and bounded  valuation  ring of $C$. Then every $\CG$-torsor over  $U=\Spec(W(C^+))\setminus V(p, [\varpi])$ extends to a $\CG$-torsor over $\Spec(W(C^+))$ and hence is trivial.
 \smallskip

 \noindent\emph{Here $\varpi$ denotes a pseudo-uniformizer of $C^+$.} 
 \end{proposition}
 
 \begin{proof} 
  We have an exact sequence
 \[
 1\to \CG^o\to \CG\to i_*(\pi_0(\CG))\to 1
 \]
 of \'etale sheaves, where $i: \Spec(k)\hookrightarrow \Spec(\br\BZ_p)$.
 Note  $U\times_{\Spec(\br\BZ_p)}\Spec(k)=\Spec(C^+[1/\varpi])=\Spec(C)$.
This gives the exact sequence
 \[
 {\rm H}^1(U, \CG^o)\to {\rm H}^1(U, \CG)\to {\rm H}^1(U, i_*(\pi_0(\CG)))={\rm H}^1(\Spec(C), \pi_0(\CG))=(0).
 \]
By Ansch\"utz's theorem \cite{An}, ${\rm H}^1(U, \CG^o)=(0)$ and   the result follows.
 \end{proof}
 
 This extends as follows to strictly totally disconnected affinoid perfectoids which are given as ``products of points".
 
 \begin{proposition}\label{AnextProduct}
 Let $(R, R^+)$ be a ``product of the points $(C_i, C^+_i)$, $i\in I$", with all $C_i$ algebraically closed  of characteristic $p$, i.e., $R^+=\prod_{i\in I}C_i^+$, $R=R^+[1/\varpi]$, with the $\varpi$-topology, where $\varpi=(\varpi_i)_i$, with $\varpi_i$ pseudouniformizers of $C^+_i$. Then every $\CG$-torsor over
  $\Spec(W(R^+))\setminus V(p, [\varpi])$ extends to a $\CG$-torsor over $\Spec(W(R^+))$ and is trivial.
 \end{proposition}
 
 \begin{proof}   In the parahoric case $\CG=\CG^o$, the extension follows from \cite[Prop. 11.5]{An}, see also \cite[Thm. 2.8]{Gl21}. 
Observe that all \'etale covers of
 \[
 \Spec(R^+)=\Spec(\prod_{i\in I}C_i^+ )
 \]
 split. Indeed, as in the proof of \cite[Lem. 6.2]{BS}, we can consider 
 \[
 \Spec(\prod_{i\in I}C_i^+)\to \pi_0(\Spec(\prod_i C_i^+))=\pi_0(\Spec(\prod_i k_i))=\beta I
 \]
 where $k_i=C_i^+/\fkm_i$ is the (algebraically closed) residue field. Here, $\beta I$ is the Stone-\v Cech compactification of the discrete set $I$. Each connected component of $\Spec(R^+)$ (i.e. fiber of this map) is the spectrum of a valuation ring $V$ with algebraically closed
 fraction field $K$; this valuation ring is an ultraproduct of $C^+_i$. It follows that all \'etale covers of $\Spec(V)$ and then also of $\Spec(R^+)$ split.
 (We see that $R^+$ is ``strictly $w$-local" in the terminology of \cite[2.2]{BSproet}, and, in fact, $\Spec(R^+)$ is ``$w$-contractible", \cite[Lem. 2.4.8]{BSproet}.) 
 
 A similar picture holds for $R=R^+[1/\varpi]$: in this case, we have \[\pi_0(\Spec(R))\simeq \pi_0(\Spec(R^+))\simeq \beta I\] (by considering idempotents) and
 \[
 \Spec(R)\to \pi_0(\Spec(R))\simeq \beta I,
 \]
 has every fiber isomorphic to $\Spec(V[1/\varpi])$, with $V[1/\varpi]$ a valuation ring with (algebraically closed)
 fraction field $K$. \'Etale covers over each such $\Spec(V[1/\varpi])$ split, and then  \'etale covers over $\Spec(R)$ also split.   
 
 Now it also follows that every $\CG$-torsor over $\Spec(W(R^+))$ is trivial. 
 Indeed, since $\CG$ is smooth and $W(R^+)$ is $p$-adically complete, it is enough 
 to show that all $\CG$-torsors over $R^+=\prod_{i\in I} C^+_i$ are trivial. As above, 
 we see that all \'etale covers of $\Spec(R^+)$ split and so this follows using the smoothness of $\CG$; the same applies of course to $\CG^o$-torsors.
 The argument in the proof of Proposition \ref{Anext} above now 
 extends to $U=\Spec(W(R^+))\setminus V(p,[\varpi])$, to complete the proof.
 \end{proof}
 
 \subsection{The  affine Witt Grassmannian and affine Deligne-Lusztig varieties}\label{4.3}
 
 Let
 \begin{equation}
 X_\CG:={\rm Gr}^W_\CG =LG/L^+\CG
 \end{equation}
 be the Witt vector affine Grassmannian for $\CG$, cf. \cite{BS, ZhuAfGr}. Here $LG(R)=G(W(R)[\frac{1}{p}])$ and $L^+\CG(R)=\CG(W(R))$ for any perfect $k$-algebra $R$. Note that here $\CG$ is only assumed to be a quasi-parahoric group scheme, whereas in loc.~cit. it is assumed that $\CG$ is a parahoric group scheme. Also, note that $k$-points of $X_\CG$ are given by isomorphism classes of pairs $(\CP,\alpha)$ of  a $\CG$-torsor $\CP$ over $W(k)$ with a trivialization $\alpha$ of its restriction to $W(k)[1/p]$.
 
 Let $(\CG, b, \mu)$ be an integral local shtuka datum such that   $\CG$ is a quasi-parahoric group scheme for $G$.  Inside the Witt vector affine Grassmannian we consider the affine Deligne-Lusztig variety $X_{\CG}(b, \mu^{-1})$ (the $\mu^{-1}$-admissible locus). This is
 a perfect scheme which is locally (perfectly) of finite type (cf. \cite{ZhuAfGr}, \cite{HamaVie}) over $k$ with 
\begin{equation}\label{defADL}
X_{\CG}(b, \mu^{-1})(k)=\{g\breve K\in G(\br\BQ_p)/\breve K\mid g^{-1}b\phi(g)\in \Adm^K(\mu^{-1})\}\subset X_\CG(k) .
\end{equation}
Here $\Adm^K(\mu^{-1})=\breve K\Adm(\mu^{-1})\breve K$, where $\Adm(\mu^{-1})\subset \wt W$ is the $\mu^{-1}$-admissible subset of the Iwahori Weyl group of $G$. 

 \begin{proposition}\label{MX}
 Let $(\CG, b, \mu)$ be an integral local shtuka datum such that   $\CG$ is a quasi-parahoric group scheme for $G$. Then the {\sl reduced locus} $(\CM^{\rm int}_{\CG, b, \mu})_\red$ 
 is represented by the perfect $k$-scheme $X_{\CG}(b, \mu^{-1}) $, and hence
  \[
 \CM^{\rm int}_{\CG, b, \mu}(\Spd(k))=X_{\CG}(b, \mu^{-1})(k) .
 \] 
\end{proposition}
\begin{proof} Here, the reduced locus $\CF_\red$ of a small $v$-sheaf $\CF$ is a ``scheme-theoretic" $v$-sheaf which is defined as in \cite[Def. 3.12]{Gl}.  For   parahoric $\CG$ the proposition is shown in \cite{Gl21}, cf. \cite[Prop. 2.30]{Gl21}.
One main ingredient in the proof of this is the Ansch\"utz purity theorem for $\CG$ and its extension to products of points. By Propositions \ref{Anext} and \ref{AnextProduct},  these purity statements remain true for quasi-parahoric group schemes. Also, by \cite[Thm. 6.16]{AGLR} combined with (\ref{locmodGtoG0}), the reduced locus $(\BM^v_{\CG, \mu})_\red$ is represented by the perfect $k$-scheme which is the  $\mu$-admissible  locus in $X_\CG$. With these ingredients, the proof in \cite{Gl21} now extends to this case.  
\end{proof}
\begin{remark}
In the identification of Prop. \ref{MX} above, the inverse $\mu^{-1}$ appears on the RHS because of the convention in the definition of ``bounded by $\mu$" for $\CG$-shtuka, comp. \cite[2.4.4]{PRglsv}. 
 \end{remark}
 
\subsection{The specialization map and formal completions} \label{def:spec}
 
 Gleason explains certain conditions on a small $v$-sheaf $\CF$ to construct a continuous specialization map on the underlying spaces,
 \begin{equation*}
{ \rm sp}_\CF\colon |\CF|\to |\CF_\red| ,
 \end{equation*}
 cf. \cite[\S 4.2]{Gl}, see also \cite[\S 2.3]{AGLR}.

 Futhermore, he proves that these conditions are satisfied when $\CF=\BM^v_{\CG, \mu}$ 
 with $\CG$ reductive (hyperspecial parahoric). This result is extended to parahoric $\CG$ by \cite[Prop. 4.14]{AGLR}.
 In view of (\ref{locmodGtoG0}), it also holds for quasi-parahoric $\CG$. The scheme-theoretic $v$-sheaf $(\BM^v_{\CG, \mu})_\red$ is represented by a perfect $k$-scheme (by \cite{AGLR} this is the $\mu$-admissible locus) which is a closed subscheme of the Witt affine Grassmannian $X_\CG={\rm Gr}^W_\CG$. 
   
 By \cite[Prop. 2.30]{Gl21}, these conditions are also satisfied in the case when $\CF=\CM^{\rm int}_{\CG, b, \mu}$ where $\CG$ is a hyperspecial parahoric subgroup. Again, one main ingredient is the Ansch\"utz purity theorem for $\CG$ and its extension to product of points. 
  By Propositions \ref{Anext} and \ref{AnextProduct},  these purity statements remain true for quasi-parahoric group schemes, and the proof extends. 
  
  In fact, by Proposition \ref{MX},  $(\CM^{\rm int}_{\CG, b, \mu})_\red$ is represented by the affine Deligne-Lusztig variety (ADLV)  $X_\CG(b, \mu^{-1})$ in the Witt affine Grassmannian $X_\CG$. Hence we get a continuous map
 \begin{equation}\label{specM}
 {\rm sp}_{\CM^{\rm int}_{\CG, b, \mu}}\colon |\CM^{\rm int}_{\CG, b, \mu}|\to |(\CM^{\rm int}_{\CG, b, \mu})_\red|=|X_\CG(b,\mu^{-1})| .
 \end{equation}

 Using this, we define the formal completion of $\CM^{\rm int}_{\CG, b, \mu}$ along a point $x\in \CM^{\rm int}_{\CG, b, \mu}(\Spd(k))=X_\CG(b, \mu^{-1})(k)$ as the sub-$v$-sheaf of $\CM^{\rm int}_{\CG, b, \mu}$, with
$$
\CM^{\rm int}_{\CG, b, \mu/x}(S)=\{y\colon S\to \CM^{\rm int}_{\CG, b, \mu}\mid { \rm sp}_{\CM^{\rm int}_{\CG, b, \mu}}\circ y(|S|)\subset \{x\} \} .
$$
Similarly, we can define the formal completion $\BM^v_{\CG, \mu/y}$ of $\BM^v_{\CG, \mu}$ along a point 
\[
y\in \BM^v_{\CG, \mu}(\Spd(k))\subset X_\CG(k)=G(W(k)[1/p])/\CG(W(k)).
\] 
 Using the condition ``bounded by $\mu$", we obtain a map
\begin{equation}\label{Lmap}
\ell: \CM^{\rm int}_{\CG, b, \mu}(\Spd(k))\to \CG(W(k))\backslash 
\BM^v_{\CG, \mu}(\Spd(k)). 
\end{equation}
The set of orbits which appears as the target of $\ell$ is
the $\mu$-admissible set for $K$. The image of a  point in  $\CM^{\rm int}_{\CG, b, \mu}(\Spd(k))$ under $\ell$ is obtained by choosing a trivialization of the $\CG$-shtuka and taking the coset given by the inverse of the Frobenius map.

If $\mu$ is minuscule, then $\BM^v_{\CG, \mu}$ is representable by the $O_E$-scheme $\BM^{\rm loc}_{\CG, \mu}$ (\cite{AGLR}), and the formal completion $\BM^v_{\CG, \mu/y}$
is given as $\Spd(A, A)$, where $A$ is the completion of the local ring of $\BM^{\rm loc}_{\CG, \mu}\otimes_{O_E}{\br O_E}$ at the corresponding point, taken with the topology given by the maximal ideal.  In this case, we also have 
\[
\CG(W(k))\backslash 
\BM^v_{\CG, \mu}(\Spd(k))=\CG(k)\backslash \BM^{\rm loc}_{\CG, \mu}(k),
\]
so $\ell(x)$ above can be also considered as a $\CG(k)$-orbit in $\BM^{\rm loc}_{\CG, \mu}(k)$.

\subsection{Change of base point}\label{ss:basepoint}

Note that, when $b\in {\rm Adm}^K(\mu^{-1})$, then $\CM^{\rm int}_{\CG, b, \mu}$ has a canonical $\Spd(k)$-valued ``base point" $x_0$. Under the identification $\CM^{\rm int}_{\CG, b, \mu}(\Spd(k))=X_\CG(b, \mu^{-1})(k)\subset X_\CG(k)$ it corresponds to the ``base point" of the ADLV given by the trivial coset. 

In general, let $x\in \CM^{\rm int}_{\CG, b, \mu}(\Spd(k))=X_\CG(b, \mu^{-1})(k)\subset X_\CG(k)$ correspond to the isomorphism class of a pair $(\CP, \alpha)$,
where $\CP$ is a $\CG$-torsor over $W(k)$ and 
\[
\alpha:  \CP[1/p] \xrightarrow{\sim}  \CG\times \Spec(W(k)[1/p])
\]
 is a trivialization of the restriction $\CP[1/p]$ of $\CP$ to $\Spec(W(k)[1/p])$
  such that
\[
\phi_\CP=\alpha^{-1}\cdot \phi_b\cdot \phi^*(\alpha): \phi^*(\CP)[1/p]\xrightarrow{\sim} \CP[1/p]
\]
has pole at $p=0$ bounded by $\mu$. Choose a trivialization of the $\CG$-torsor $\CP$ over $W(k)$. Then $\alpha$ is given by $g\in \CG(W(k)[1/p])$ and $\phi_\CP$ by $b_x\times \phi\colon \CG[1/p]\to\CG[1/p]$ such that  $b_x=g^{-1}b\phi(g)$. Hence we have $b_x=g^{-1}b\phi(g)\in {\rm Adm}^K(\mu^{-1})$. We obtain an isomorphism
\[
\tau_g\colon \CM^{\rm int}_{\CG, b, \mu}\to \CM^{\rm int}_{\CG, b_x, \mu} ,\quad (\sP, \phi_\sP, i_r)\mapsto (\sP, \phi_\sP, i_r\circ g^{-1}) 
\]
which sends $x$ to the base point $x_0$  of $\CM^{\rm int}_{\CG, b_x, \mu}$, cf. \cite[proof of Prop. 3.4.1]{PRglsv}. 
This gives an isomorphism (depending on our choices) 
\begin{equation}\label{eqBPt}
\CM^{\rm int}_{\CG, b, \mu/x}\simeq  \CM^{\rm int}_{\CG, b_x, \mu/x_0} .
\end{equation}
In particular, $ \CM^{\rm int}_{\CG, b, \mu/x}$ is representable (by a normal complete Noetherian local ring) for given $\CG$ and any $b$ in a given $\phi$-conjugacy class in $B(G)$ and arbitrary $x\in  \CM^{\rm int}_{\CG, b, \mu}(\Spd(k))$ if and only if this holds for the base point $x_0$ of  $\CM^{\rm int}_{\CG, b, \mu}$, for given $\CG$ and any $b\in {\rm Adm}^K(\mu^{-1})$ in the given $\phi$-conjugacy class.   Note that, by He's theorem \cite{He}, any $\phi$-conjugacy class $[b]\in B(G, \mu^{-1})$ contains elements $b\in {\rm Adm}^K(\mu^{-1})$.  

\subsection{A $v$-sheaf ``local model diagram"}

Let $\CH$ be an affine group scheme over $\br\BZ_p$. For an affinoid perfectoid $(R, R^+)$ over $k$, consider
\[
\BW^+\CH(\Spa(R, R^+))=\CH(W(R^+)),
\]
and
 \[
 \wh \BW^+\CH(R, R^+)=\{h\in \CH(W(R^+))\ |\ h\equiv 1\, \hbox{\rm mod}\, [\varpi_h]\},
 \]
 ($\varpi_h$ stands for some pseudo-uniformizer of $R^+$ that depends on $h$).

These definitions extend to perfectoid spaces over $k$ and define $v$-sheaves of groups $\BW^+\CH$ and $ \wh \BW^+\CH$ on ${\rm Perfd}_k$. We set
\[
\wh {L_W^+}\CH= \wh \BW^+\CH\times \Spd(\BZ_p)
\]
which is a $v$-sheaf of groups over $\Spd(\BZ_p)$. (This definition appears
in \cite[2.3.15]{Gl}.) We have 
\begin{equation}
\wh {L_W^+}\CH(S)=\{((S^\sharp, y), h)\ |\ h\in \CH(W(R^+)),\ h\equiv 1\, {\mathrm {mod}}\, [\varpi_h] \},
\end{equation}
where $S=\Spa(R, R^+)$ and $(S^\sharp, y)$ is an $S$-valued point of $\Spd(\BZ_p)$, and where $\varpi_h$ is a pseudouniformizer of $R^+$ (that depends on $h$).  Here, $S^\sharp=\Spa(R^\sharp,R^{\sharp +})$ together with $y:(S^\sharp)^\flat\xrightarrow{\sim} S$ is an untilt of $S$, see
 \cite[Prop. 11.3.1]{Schber}.

 We will later need the following lemma. 
 
 \begin{lemma}\label{recursive}
 For each $h\in  \wh \BW^+\CH(R, R^+)$, there is a unique $\lambda\in \wh \BW^+\CH(R, R^+)$
 such that
 $
 h=\lambda^{-1}\cdot \phi(\lambda).
 $
  \end{lemma}
  
  \begin{proof} The ring $W(R^+)$ is complete and separated for the $[\varpi]$--topology, where $\varpi$ is a pseudo-uniformizer
  of $R^+$.
  Set inductively $\lambda_0=1$, $\eta_0=h^{-1}$, and
  \[
  \eta_n=\lambda^{-1}_n\phi(\lambda_n)\cdot h^{-1},\quad \lambda_{n+1}=\lambda_n\cdot\eta_n,
  \]
 Then we have
  \[
  \eta_{n+1}=h\cdot \phi(\eta_n)\cdot h^{-1}.
  \]
  Since $\eta_0\equiv 1\, \hbox{\rm mod}\, [\varpi_h]$, this gives $\eta_n\equiv 1\, \hbox{\rm mod}\, [\varpi^{p^n}_h]$.
  Hence, $\eta_n$ converges to $1$ and $\lambda_n$ converges to an element $\lambda$ with
  \[
  h=\lambda^{-1}\phi(\lambda).
  \]
  (cf. \cite[Lem. 2.15]{Gl21}.)
  If $\lambda=\phi(\lambda)$ with $\lambda\equiv 1\, \hbox{\rm mod}\, [\varpi]$,
  then we easily see inductively $\lambda\equiv 1\, \hbox{\rm mod}\, [\varpi^{p^n}]$
  for all $n$,
  so $\lambda=1$, so uniqueness follows.
  \end{proof}

 Let  $(\CG, b, \mu)$ be an integral local shtuka datum such that $\CG$ is quasi-parahoric.  We define as follows  a functor $\wh{L\CG}_{b, \mu}$  on ${\rm Perfd}_k$ over $\Spd(O_{\br E})$, cf. \cite[\S 2.4]{Gl21}. It  assigns to a affinoid perfectoid $S=\Spa(R, R^+)$ over $k$ the set
\[
\wh{L\CG}_{b, \mu}(S)=\{\hbox{\rm isomorphism classes of}\ ((S^\sharp, y), \CP, \psi, \sigma)\}
\]
where 
\begin{itemize}

\item $(S^\sharp, y)$ is an untilt of $S$ over $ O_{\br E}$,  

\item $\CP$ is a $\CG$-torsor over $\Spec(W(R^+))$, 

\item
$
\psi: \CP[1/\xi_{R^\sharp}]\xrightarrow{\sim} \CG\times \Spec(W(R^+)[1/\xi_{R^\sharp}])$, and $\sigma: \CP\xrightarrow{\sim}  \phi^*(\CG\times \Spec(W(R^+)))
$

are both $\CG$-isomorphisms such that:

\smallskip

1) $(\CP, \psi)$ is bounded by $\mu$ along $\xi_{R^\sharp}=0$,

\smallskip

 2) there is a pseudo-uniformizer $\varpi\in R^+$ such that $\phi_b \circ \sigma\equiv\psi \, \hbox{\rm mod}\,  [\varpi]$.
\end{itemize}

Here, we denote by $\xi_{R^\sharp}$ a generator of  the  map $W(R^+)\to R^{\sharp+}$ given by the untilt $(S^\sharp, y)$ of $S$.
Also ``bounded by $\mu$" means, by definition, that the
point of the ${\BB}_{\rm dR}$-affine Grassmanian ${\rm Gr}_{\CG, \Spd(\breve O_E)}$ given by $((S^\sharp, y), \CP, \psi)$ factors through the $v$-sheaf local model $\BM^v_{\CG, \mu}\subset {\rm Gr}_{\CG, \Spd(\breve O_E)}$. More precisely, choose (locally) a section $\tau: \CG\xrightarrow{\sim}\CP$ 
and consider $g=\psi\circ \tau(1)$; we ask that $g^{-1}\CG({\BB}^+_{\rm dR}(S^\sharp))$ lies in $\BM^v_{\CG, \mu}(R^\sharp)$.

As in \cite{Gl21}, we can see that $\wh{L\CG}_{b, \mu}$ is a $v$-sheaf over $\Spd(O_{\br E})$
(this $v$-sheaf is denoted by $\wh {\rm WSht}^{\CD, \leq \mu}$
 in \cite[Def. 2.34 (3), (4)]{Gl21}, and by $L\CM^{\rm int}_{\CG, b, \mu/x_0}$ in \cite{PRglsv}).

It is useful to observe that by using the mapping 
\[
(\CP, \psi, \sigma)\mapsto h=(\psi\circ \sigma^{-1})(1)\in \CG(W(R^+)[1/\xi_{R^\sharp}]) ,
\]
we obtain the following simpler description:
\begin{equation}\label{vlocalDescr}
\wh{L\CG}_{b, \mu}(S)=\{((S^\sharp, y), h)\ |\ h\in \CG(W(R^+)[1/\xi_{R^\sharp}]),\ h
\equiv b\, {\mathrm {mod}}\, [\varpi_h], \ [h^{-1}]\in \BM^v_{\CG, \mu}(S)\},
\end{equation}
where $(S^\sharp, y)$ is an untilt of $S$ over $ O_{\br E}$, $\varpi_h$ is a pseudo-uniformizer of $R^+$, and $[h^{-1}]$ is the $S$-point of the ${\BB}_{\rm dR}$-affine Grassmanian ${\rm Gr}_{\CG, \Spd(\breve O_E)}$ defined by the coset $h^{-1} \CG(\BB^+_{\rm dR}(R^\sharp))$.

\begin{theorem}\label{vLMD} 
There is a diagram of $v$-sheaves over $\Spd(\breve O_E)$
\begin{equation}\label{GDia}
\begin{gathered}
   \xymatrix{
	     &\wh{L\CG}_{b, \mu} \ar[dl]_-{\text{$\pi_\bullet$}} \ar[dr]^-{\text{$\pi_\star$}}\\
	  \CM^{\rm int}_{\CG, b, \mu/x_0}  & & \ \BM^v_{\CG, \mu/y_0} \ .
	}
\end{gathered}
\end{equation}
where both $\pi_{\bullet}$, $\pi_\star$ are $\LG$-torsors  (for the $v$-topology)
for two corresponding actions (see (d) below).
\smallskip

a) Here $x_0\in \CM^{\rm int}_{\CG, b, \mu}(\Spd(k))$ denotes the base point as in \S\ref{ss:basepoint} above. Similarly  $y_0\in \BM^v_{\CG,  \mu}(\Spd(k))$  denotes the corresponding point of $\BM^v_{\CG, \mu}(\Spd(k))\subset X_{\CG}(k)$ given by the pair $(\CG, b^{-1})$, i.e. by the coset $b^{-1}\,\CG(W(k))$.
Also,  $\CM^{\rm int}_{\CG, b, \mu/x_0}$, resp. $\BM^v_{\CG, \mu/y_0}$, denotes the $v$-sheaf given by the formal completion
of $\CM^{\rm int}_{\CG, b, \mu}$, resp. $\BM^v_{\CG, \mu}$, at these points.
(See \S\ref{def:spec} above.)

b) The map $\pi_\star$ is given by
\[
\pi_\star(((S^\sharp, y), \CP, \psi, \sigma))=((S^\sharp, y), \CP, \psi).
\]

c) The map $\pi_\bullet$ is given by
\[
\pi_\bullet((S^\sharp, y), \CP, \psi, \sigma)=((S^\sharp, y), (\sP,\phi_{\sP}),  i_r).
\]
Here,  $(\sP,\phi_{\sP})$ is the $\CG$-shtuka over $S$ with leg at $y$ given as follows: The $\CG$-torsor $\sP$ is $(\phi^{-1})^*(\CP)$ restricted to $\CY_{[0,\infty)}(S)$, i.e. 
\[
\sP=(\phi^{-1})^*(\CP)_{|\CY_{[0,\infty)}(S)}.
\] 
The Frobenius $\phi_{\sP}$ is
given by a similar restriction of the composition
\[
\Phi: \CP[1/\xi_{R^\sharp}]\xrightarrow{\psi} \CG\times \Spec(W(R^+)[1/\xi_{R^\sharp}])\xrightarrow{((\phi^{-1})^*\sigma)^{-1}} (\phi^{-1})^*(\CP)[1/\xi_{R^\sharp}].
\]
(Note that $\phi^*(\sP)=\CP_{|\CY_{[0,\infty)}(S)}$.) The framing $i_r$ is constructed in \cite{Gl21}.

d) The $v$-sheaf of groups $\LG$ acts on $\wh{L\CG}_{b, \mu} $ on the right by
\[
( \CP, \psi, \sigma)\star g=(  \CP, \psi, \phi^*(r_{\phi^{-1}(g)})\circ \sigma), \quad
( \CP, \psi, \sigma)\bullet g= (  \CP, r_{g}\circ \psi, \phi^*(r_{g})\circ\sigma),
\]
where, $r_a$ is the map given by right multiplication by $a$, and, for simplicity, we omit the untilt $(S^\sharp, y)$ from the notation.
\end{theorem}

Note that the composition
$
\Phi: \CP[1/\xi_{R^\sharp}] \to   (\phi^{-1})^*(\CP)[1/\xi_{R^\sharp}]
$
in c), is fixed under the $\bullet$-action. This follows from the identity
\[
((\phi^{-1})^*(\phi^*(r_{g})\circ \sigma))^{-1}\circ (r_g\circ \psi)=(r_g\circ (\phi^{-1})^*\sigma)^{-1}\circ r_g\circ \psi=((\phi^{-1})^*\sigma)^{-1}\circ\psi.
\]

\begin{proof}
The statement of the theorem is shown by Gleason for $\CG$ parahoric (see \cite[Thm. 2.33, Lem. 2.35, Lem. 2.36]{Gl21}). The proof generalizes to $\CG$ quasi-parahoric by using Proposition \ref{Anext} and its extension Proposition \ref{AnextProduct} to products of points. 
\end{proof}

It is  convenient to express the torsors $\pi_\star$, $\pi_\bullet$, using the description (\ref{vlocalDescr})
of $\wh{L\CG}_{b, \mu} $.
In the description (\ref{vlocalDescr}), $\pi_\star$ is given by projection to the coset $[h^{-1}]=h^{-1} \CG(\BB^+_{\rm dR}(R^\sharp))$, and $\pi_\bullet$ is given by sending $h$ to the shtuka $\sP(h)=\CG\times \CY_{[0,\infty)}(S)=\CG_{ \CY_{[0,\infty)}(S)}$ (the trivial torsor) with
 Frobenius defined by $\phi_{\sP(h)}=\phi_h=r_h\phi:  (\phi^*\CG_{\CY_{[0,\infty)}(S)})[1/\xi_{R^\sharp}]\xrightarrow{\sim} \CG_{\CY_{[0,\infty)}(S)}[1/\xi_{R^\sharp}]$. 
To check this description of $\pi_\bullet$, note that,  since $\phi_h=r_h \phi= \psi\circ \sigma^{-1}$, and $\Phi=((\phi^{-1})^*(\sigma))^{-1}\circ\psi$,
 the following diagram commutes
 \begin{displaymath}
 \xymatrix{
        \CP[1/\xi_{R^\sharp}] \ar[r]^{\Phi\ \ \ \ \ \ \  }\ar[d]_{\sigma} &   (\phi^{-1})^*(\CP)[1/\xi_{R^\sharp}]\ar[d]^{(\phi^{-1})^*\sigma}\\
     \phi^*(\CG_{ {W(R^+)}})[1/\xi_{R^\sharp}]   \ar[r]^{\ \ \ \ \ \ \phi_h\ \ \ \ }& \CG_{W(R^+)}[1/\xi_{R^\sharp}].
     }
\end{displaymath}
Hence, the $\CG$-shtukas $(\sP(h), \phi_{\sP(h)})$ and $(\sP, \phi_\sP)$ given above, are isomorphic. 

Under the above isomorphism, the framing $i_r$ of $(\sP, \phi_\sP)$ corresponds to a framing $i_r(h)$ of $(\sP(h), \phi_{\sP(h)})$; this is given by the unique lift of the identity trivialization modulo $[\varpi_h]$, cf. Lemma \ref{recursive}. 
More precisely, for $((S^\sharp, y), h)\in \wh{L\CG}_{b, \mu}(S)$, the framing 
of $(\sP(h), \phi_{\sP(h)})$ is given by   the unique element $ i_r(h)\in G(B^{[r,\infty)}_{(R, R^+)})$ with $i(h)
\equiv 1\, {\mathrm {mod}}\, [\varpi_h]$ and the property 
 \[
 h=i_r(h)^{-1}\cdot b\cdot \phi(i_r(h)). 
 \]
 (See \cite[Lem. 2.15]{Gl21}.) 
Finally, the actions correspond to $h\star g=g^{-1}h$ and $h\bullet g=\phi(g)^{-1}h g$. 
Note that $[(h\star g)^{-1}]=[(g^{-1}h)^{-1}]=[h^{-1}g]=[h^{-1}]$.

\begin{remark}\label{rem343}
a) Gleason calls the diagram (\ref{GDia}) a ($v$-sheaf) ``local model diagram". However, we warn the reader that (\ref{GDia}) does not compare directly to the local model diagrams in the theory of Shimura varieties
and of Rapoport-Zink spaces; these are of a different nature. In particular, the group acting there is $\CG$ which is, in a sense, ``finite dimensional", while here we have torsors for $\LG$.

b) The existence of the diagram (\ref{GDia}) with the properties listed in the above theorem, 
does not imply the isomorphism
 (\ref{Gliso}) of Theorem \ref{thmRepgoal}: for example, we cannot deduce from general principles that the torsors
$\pi_\star$ and $\pi_\bullet$ split, since the formal completions are not ``sufficiently local" for the $v$-topology.
\end{remark}

\section{Parahoric vs Quasi-parahoric group schemes}\label{s:parvs}

\subsection{The aim of this section}
The aim of this section is to prove the following devissage result. 
  
 \begin{theorem}\label{varyGthm} Let $(G, b, \mu)$ be a local Shimura datum. The following are equivalent:
 \begin{itemize}
 \item[1)] $\CM^{\rm int}_{\CG, b, \mu}$ satisfies the representability conjecture \ref{repconj}, for all parahoric $\CG$.
 
 \item[2)] $\CM^{\rm int}_{\CG, b, \mu}$ satisfies the representability conjecture  \ref{repconj}, for all stabilizer group schemes   $\CG$,
 
 \item[3)] $\CM^{\rm int}_{\CG, b, \mu}$ satisfies the representability conjecture  \ref{repconj}, for all quasi-parahoric $\CG$.
 \end{itemize}
  \end{theorem}
  Here by a stabilizer group scheme we understand the BT-group scheme associated to the stabilizer of a point in the extended building $\sB^e(G, \br\BQ_p)$ (a particular type of quasi-parahoric group schemes, cf. section  \ref{def:quasi}). The proof proceeds by the comparing formal completions of $\CM^{\rm int}_{\CG, b, \mu}$ and of $\CM^{\rm int}_{\CG^o, b, \mu}$ (cf. Section \ref{ss:formcompl}) and of their underlying reduced  schemes  (cf. Section \ref{ss:ADLV}). 
\subsection{Formal completions} \label{ss:formcompl}
Let $(\CG, b, \mu)$ be an integral local shtuka datum such that $\CG$ is quasi-parahoric. For simplicity of notation, we set $O=O_{\br E}$ with residue field $k=\bar k_E$. We consider the natural $v$-sheaf morphism
\begin{equation}
\pi: {\CM^{\rm int}_{\CG^o, b, \mu}}\to {\CM^{\rm int}_{\CG, b, \mu}}
\end{equation}
over $\Spd(O)$. 

\begin{proposition}\label{neutralIso}  
 a)  The map $\pi$ is  qcqs (quasi-compact quasi-separated \cite{Sch-Diam}).
 \smallskip
 
\noindent b) The map $\pi$ induces an isomorphism
\[
\CM^{\rm int}_{\CG^o, b,\mu/x}\xrightarrow{\sim} \CM^{\rm int}_{\CG, b,\mu/\pi(x)}
\]
for each $x\in \CM^{\rm int}_{\CG^o, b,\mu}(\Spd(k))$. 
\end{proposition}

\begin{proof}
(a) Observe that $\CG^o\to \CG$ is an isomorphism on the generic fibers and can be identified with the dilation $\CG^{\rm dil}$
of $\CG$ along its closed subscheme given by the neutral component $(\CG\otimes_{\BZ_p}{\mathbb F}_p)^o$ of the special fiber $\CG\otimes_{\BZ_p}{\mathbb F}_p$. (Indeed, both $\CG^o$ and the dilation $\CG^{\rm dil}$
 are smooth,
have the same generic fiber and the same $\br\BZ_p$-points.) In particular, we have $\CO_{\CG}\into \CO_{\CG^o}=\CO_{\CG}[f_1,\ldots, f_m]$,
and there is $N\geq 1$, such that $p^Nf_i\in \CO_{\CG}$, for all $i$. Then the argument in \cite[proof of Prop. 3.6.2]{PRglsv} applies to prove the claim.

(b) As in section \ref{ss:basepoint} we have an isomorphism 
\begin{equation}\label{eqBPt2}
\CM^{\rm int}_{\CG, b, \mu/x}\simeq  \CM^{\rm int}_{\CG, b_x, \mu/x_0}
\end{equation}
and similarly for $\CG^o$.
Hence  it is enough to show the isomorphism for the base point $x=x_0\in \CM^{\rm int}_{\CG^o, b,\mu}(\Spd(k))$ (at the cost of changing $b=b_{x_0}$ to $b_x$). By  Theorem \ref{vLMD}
\[
\wh{L\CG}^o_{b,\mu}\to \CM^{\rm int}_{\CG^o, b,\mu/x_0}
\]
is a $\wh \BW^+\CG^o\times\Spd(O_{\br E})$-torsor. Also, $\wh{L\CG}^o_{b,\mu}$ is a  $\wh \BW^+\CG^o\times\Spd(O_{\br E})$-torsor over 
$\BM^{v}_{\CG^o, \mu/x_0}$, and there are corresponding statements for
$\CG$. By \eqref{locmodGtoG0}, the natural map $\CG^o\to \CG$ induces an isomorphism 
\[
\BM^{v}_{\CG^o, \mu}\xrightarrow{\sim}  \BM^{v}_{\CG, \mu}.
\]
On the other hand, it is easy to see that $\CG^o\to \CG$ induces 
$\wh \BW^+\CG^o\xrightarrow{\sim}  \wh \BW^+\CG$ and, hence, $\wh{L\CG}^o_{b,\mu}\xrightarrow{\sim} \wh{L\CG}_{b,\mu}$. The result  follows.
\end{proof}

\begin{remark}\label{largepointremark} Suppose $\kappa$ is an algebraically closed field over $k$ and set $O(\kappa)= O_{\br E}\otimes_{W(k)}W(\kappa)$. 
We can consider the base change 
\[
(\CM^{\rm int}_{\CG, b, \mu})_{O(\kappa)}:=\CM^{\rm int}_{\CG, b, \mu}\times_{\Spd(O)}\Spd( O(\kappa)),
\]
and similary for $\CG^o$.
Given $x\in \CM^{\rm int}_{\CG^o, b, \mu}(\Spd(\kappa))$, we obtain a corresponding point $x\in (\CM^{\rm int}_{\CG^o, b, \mu})_{O(\kappa)}(\Spd(\kappa))$. The formal completions $( \CM^{\rm int}_{\CG^o, b, \mu})_{ O(\kappa)/x}$ and $( \CM^{\rm int}_{\CG, b, \mu})_{ O(\kappa)/\pi(x)}$ now make sense and the argument in the proof of (b)
above also applies to give 
\begin{equation}\label{largepointeq}
( \CM^{\rm int}_{\CG^o, b, \mu})_{ O(\kappa)/x}\xrightarrow{\sim}( \CM^{\rm int}_{\CG, b, \mu})_{ O(\kappa)/\pi(x)}.
\end{equation}

\end{remark}

\subsection{ADLV for quasi-parahorics}\label{ss:ADLV}
In this subsection, we express the ADLV $X_{\CG}(b, \mu)$ for a quasi-parahoric $\CG$ in terms of ADLV attached to parahoric subgroups. Note that the results will eventually be applied for $\mu$ replaced by $\mu^{-1}$, to relate to integral local Shimura varieties via Prop. \ref{MX}.

Recall the Kottwitz map $\kappa_G: G(\br \BQ_p)\to \Omega_G$. This map can be enhanced to a morphism $LG\to \Omega_G$ which factors through 
$X_{\CG^o}=LG/L^+\CG^o$, and induces a map    with connected fibers,
\begin{equation}
\kappa_G\colon X_{\CG^o}\to \Omega_G ,
\end{equation}
 cf. \cite[Prop. 1.21]{ZhuAfGr}, comp. also \cite[\S 5]{PRTwisted}.

Recall $\pi_0(\CG)=\breve K/\breve K^o\subset \Omega$, cf. \eqref{pi0inO}. Then $\pi_0(\CG)$ acts on $X_{\CG^o}$ by $g\br K^o\mapsto g\dot\gamma\br K^o$, where $\dot\gamma\in\br K$ is a representative of a given element in $\pi_0(\CG)$. This action is permuting connected components and the quotient is $X_\CG$, 
\[
 X_{\CG^o}\to X_\CG=X_{\CG^o}/\pi_0(\CG) .
\]
Each connected component of $X_{\CG^o}$ maps isomorphically to a connected component
of $X_\CG$.

 Recall that there exists $c_{b, \mu}\in \Omega$ with $\phi(c_{b, \mu})-c_{b, \mu}=\kappa_G(b)-\kappa_G(\mu)$, comp. \cite[\S6]{HZ}. Here $\kappa_G(\mu)\in\Omega$ is the element associated to the conjugacy class $\mu$, i.e., the residue class of $\{\mu\}$ modulo the affine Weyl group. Furthermore, the class modulo $\Omega^\phi$ of $c_{b, \mu}$ is uniquely determined. 

Now $X_{\CG^o}(b, \mu)$ lies in the union of certain connected components of $X_{\CG^o}$.  From $-\kappa_G(g)+\phi(\kappa_G(g))=-\kappa_G(b)+\kappa_G(\mu)$, we obtain $\kappa_G(X_{\CG^o}(b, \mu))\subset  c_{b, \mu}+\Omega^\phi.$ In fact, by \cite[Lem. 6.1]{HZ}, we have an  equality, 
\begin{equation}\label{coset}
\kappa_G(X_{\CG^o}(b, \mu))= c_{b, \mu}+\Omega^\phi.
\end{equation}
 
We now extend these considerations to quasi-parahorics. Let $C=C_\CG=\Omega/\pi_0(\CG)$. Then the Kottwitz homomorphism  induces a map 
\begin{equation}\label{kappaCG}
\kappa_G\colon X_\CG\to C_\CG .
\end{equation}
The image of $X_\CG(b, \mu)$ 
  in $C_\phi$ is equal to  the residue class of $\kappa_G(b)-\kappa_G(\mu)$, so
\begin{equation}\label{barc}
\kappa_G(X_\CG(b, \mu))\subset \bar c_{b, \mu}+C^\phi ,
\end{equation}
where $\bar c_{b, \mu}\in C_\CG$ is the image of $ c_{b, \mu}$. This inclusion  will turn out to be  an equality, cf. Corollary \ref{imageka} below.  

Recall $\Pi_\CG=\ker(\pi_0(\CG)_\phi\to \Omega_\phi)$, cf. \eqref{defPi}. From the exact sequence $0\to\pi_0(\CG)\to \Omega\to C_\CG\to 0$ we have an exact sequence
\begin{equation}\label{exseq}
0\to \Omega^\phi/\pi_0(\CG)^\phi\to C_\CG^\phi\to \Pi_\CG\to 0 .
\end{equation}
Let $\bar\beta\in \Pi_\CG$. Take $\beta\in \pi_0(\CG)\subset \Omega$ lifting $\bar\beta$.
Then
\[
\beta=(1-\phi)\gamma
\]
for some $\gamma\in \Omega$.  Note that, since $\phi({\gamma})-{\gamma}$ is in $\pi_0(\CG)$,
  the image $\bar\gamma$ of $\gamma$ in $C$ is $\phi$-invariant and maps to $\bar\beta$ under the connecting homomorphism $C^\phi\to\Pi_\CG$.   Conversely, the image $\bar\beta$ of  an element $\bar\gamma\in C^\phi$ in $\Pi_\CG$ is described as follows. Let $\gamma\in\Omega$ be a lift of $\bar\gamma$. Then $\bar\beta$ is the class of $\beta=\phi(\gamma)-\gamma\in\pi_0(\CG)$. 
  
  Recall that the quasi-parahoric subgroup $\br K$ is defined via a point $\bf x$ in the building of $G_\ad$. Write the Iwahori-Weyl group as $\wt W=W_a\rtimes \Omega$ and
\[
1\to T(\br\BQ_p)^0\to N(\br\BQ_p)\to \wt W=W_a\rtimes \Omega\to 1
\]
with $T=Z_G(S_{\br\BQ_p})$ as in \cite{HR}. Here the choices are made such that  $\bf x$ is in the base alcove $\frak{a}_0$ of the apartment of $S$. Then $T(\br\BQ_p)^0\subset \breve K^o$. When we consider $\Omega$ as a subset of $\wt W$, we write the group law in a multiplicative way. 

\begin{lemma} \label{liftgamma}
 Let $\beta\in\pi_0(\CG)$ be of the form $\beta=\phi(\gamma)-\gamma$, for $\gamma\in\Omega$.  There is a lift  $\dot{\gamma}\in N(\br\BQ_p)$ of $\gamma\in \Omega\subset \wt W$ such that
\[
\phi(\dot{\gamma})^{-1}\dot{\gamma}\in \breve K.
\]
\end{lemma}

\begin{proof}
Lift $\gamma\in \Omega$ to $\dot{\gamma}\in N(\br\BQ_p)$. Then $\phi(\dot{\gamma})^{-1}\dot{\gamma}$
lifts $\beta\in \pi_0(\CG)\subset \Omega$. Now we use the exact sequence
$$
1\to T(\br\BQ_p)^0\to N(\br\BQ_p)\cap\br K\to \pi_0(\CG)\to 1 .
$$ This holds since, denoting by $U=U_{K^o}$ the pro-unipotent radical of $\breve K^o$, we have  $\breve K\subset U\cdot N(\br\BQ_p)$, cf.  \cite[proof of Prop. 8]{HR}. We can lift $\beta$ to $\dot{\beta}\in N(\br\BQ_p)\cap \br K$. Then $\phi(\dot{\gamma})^{-1}\dot{\gamma}\dot\beta^{-1}\in T(\br\BQ_p)^0$. Since
$T(\br\BQ_p)^0\subset \br K^o$, the result follows.   
\end{proof}

 For $\beta$, $\gamma$, $\dot\gamma$ as above, it follows that
\[
\br K\dot{\gamma}^{-1}=\br K\phi(\dot{\gamma})^{-1}.
\]
Consider now the conjugates of $\br K$, resp. $\br K^o$, 
\begin{equation}
\br K_\gamma:=\dot{\gamma} \br K\dot{\gamma}^{-1}, \quad \br K^o_\gamma:=\dot{\gamma} \br K^o\dot{\gamma}^{-1}.
\end{equation}
 These subgroups of $G(\br\BQ_p)$ depend only on $\gamma$. Indeed, if $\dot\gamma$ is replaced by $\dot\gamma\dot\delta$, with $\dot\delta\in T(\br\BQ_p)^0$, then $\dot{\gamma} \br K^o\dot{\gamma}^{-1}$ is replaced by $\dot{\gamma} \dot\delta \br K^o\dot\delta^{-1} \dot{\gamma}^{-1}=\dot{\gamma} \br K^o\dot{\gamma}^{-1}$.
We have
\[
\phi(\dot{\gamma} \br K\dot{\gamma}^{-1})=\dot{\gamma}\br K\dot{\gamma}^{-1},\quad 
\phi(\dot{\gamma} \br K^o\dot{\gamma}^{-1})=\dot{\gamma} \br K^o\dot{\gamma}^{-1}.
\]
Hence, $\br K_\gamma$ and $\br K^o_\gamma$ are  rational, i.e., correspond to subgroups of  $G(\BQ_p)$.  The parahorics $ K^o_\gamma$  are conjugate to $K^o$ in $G(\br\BQ_p)$ but not necessarily in 
$G(\BQ_p)$.
 \begin{proposition} The $G(\BQ_p)$-conjugacy class of  $\br K_\gamma$, resp. $\br K^o_\gamma$, only depends  on the class $\bar\beta$ of $\beta$ in $\Pi_\CG$.  
\end{proposition}

\begin{proof} We go through all choices made in the construction. 
\begin{altitemize}

\item{\it independence of $\gamma$: } If $\gamma$ is replaced by $\gamma'=\gamma+\delta$, with $\delta\in\Omega^\phi$, then as choice for $\dot\gamma'$ we can take $\dot\gamma'=\dot\delta\dot\gamma$, where  $\dot\delta\in N(\BQ_p)$
 (i.e., $\BQ_p$-rational). This follows from \cite[Rem. 9]{HR}. Hence $\dot{\gamma} \br K^o\dot{\gamma}^{-1}$ is replaced by $ \dot\delta \dot{\gamma}\br K^o \dot{\gamma}^{-1}\dot\delta^{-1}$, hence is conjugate under $G(\BQ_p)$ to $\br K_\gamma^o$. 

\item{\it independence of $\beta$:} If $\beta$ is replaced by $\beta+(\phi-1)\delta$ with  $\delta\in\pi_0(\CG)$, then $\gamma$ is replaced by $\gamma+\delta$. Then we may replace $\dot\gamma$ by $\dot\gamma\dot\delta $, where $\dot\delta\in N(\br\BQ_p)\cap\br K$. Since $\dot\delta$ normalizes $\br K^o$, the parahoric  $\br K^o_\gamma$  is unchanged. 
\end{altitemize}
This handles the case of $\br K^o_\gamma$; the case of $\br K_\gamma$ is the same. \end{proof}

\begin{remark} The above proof seems to use that we multiply $\dot\gamma$ by $\dot\delta$ from the right instead of from the left. However, $\Omega$ is an abelian group and what is being used here is that  $\dot\delta\in T(\br\BQ_p)^0$ acts trivially on $\frak{a}_0$, and  that $\dot\delta\in N(\br\BQ_p)\cap\br K$ preserves   $\frak{a}_0$  and fixes $\bf{x}$ and hence also $\dot\gamma \bf{x}$, by the commutativity of $\Omega$. 

\end{remark}
In the sequel, we make a fixed choice of $\dot\gamma$, for each element $\bar\beta\in\Pi_\CG$. We denote the corresponding groups by $\br K_\beta$, resp. $\br \CG_\beta$, resp. $\br K^o_\beta$, resp. $\br \CG^o_\beta$,  by slightly abusing notation. We consider the map
\begin{equation}\label{Xbeta}
\begin{aligned}
\pi_\gamma: X_{\CG_\gamma}=LG/L^+\br K_\gamma=LG/&L^+(\dot{\gamma} \br K\dot{\gamma}^{-1})\to LG/L^+ \br K=X_{\CG},\\
& g(\dot{\gamma} \br K\dot{\gamma}^{-1})\mapsto g\dot{\gamma} \br K .
\end{aligned}
\end{equation}
Let us check the dependency of $\dot\gamma$. If $\dot\gamma$ is replaced by $\dot\gamma'=\dot\gamma\dot\delta$ with $\dot\delta\in T(\br\BQ_p)^0$, then $g\dot\gamma'\br K\dot\gamma'^{-1}\in X_{\CG_{\gamma'}}$ is mapped under $\pi_{\gamma'}$ to $g\dot\gamma'\br K=g\dot\gamma\br K$; hence $\pi_\gamma=\pi_{\gamma'}$ under the identity identification $X_{\CG_\gamma}=X_{\CG_{\gamma'}}$. Similarly, if $\gamma$ is replaced by $\gamma+\delta$ with $\delta\in\Omega^\phi$, then, choosing $\dot\gamma'=\dot\delta\dot\gamma$ with $\dot\delta\in N(\BQ_p)\cap\br K$,  we can identify $X_{\CG_\gamma}$  with $X_{\CG_{\gamma'}}$ via $g \br K_\gamma\mapsto g\dot\delta^{-1} \br K_{\gamma'}$. The first element is mapped under $\pi_\gamma$ to $g\dot\gamma\br K$, the second element is mapped under $\pi_{\gamma'}$ to $(g\dot\delta^{-1})\dot\gamma'\br K=g\dot\gamma\br K$, hence $\pi_\gamma=\pi_{\gamma'}$. Note that $\dot\delta$ is not unique. But if $\dot\delta$ is replaced by $\dot\delta'=\dot\delta\varepsilon$ with $\varepsilon\in T(\br\BQ_p)^0$, then the identification of $X_{\CG_\gamma}$ and $X_{\CG_{\gamma'}}$ is not affected. Finally, if $\beta$ is replaced by $\beta+(\phi-1)\delta$ with  $\delta\in\pi_0(\CG)$, then $\gamma$ is replaced by $\gamma+\delta$. Then we may replace $\dot\gamma$ by $\dot\gamma'=\dot\gamma\dot\delta $, where $\dot\delta\in N(\br\BQ_p)\cap\br K$. Since $\dot\delta$ normalizes $\br K^o$, the spaces $X_{\CG_{\gamma}}$ and $X_{\CG_{\gamma'}}$ are identified compatibly with the maps $\pi_\gamma$, resp. $\pi_{\gamma'}$. 
\begin{proposition}
The above map defines by restriction a map
\[
\pi_\beta: X_{\CG_\beta}(b, \mu)\to X_\CG(b, \mu) .
\]
\end{proposition}
\begin{proof}
Suppose that $g\br K_\beta\in X_{\CG_\beta}(b, \mu)$, i.e.,
\[
g^{-1}b\phi(g)\in \br K_\beta \Adm(\mu) \br K_\beta=\dot{\gamma} \br K\dot{\gamma}^{-1} \Adm(\mu) \dot{\gamma} \br K\dot{\gamma}^{-1}.
\]
Since $\Adm(\mu)$ is stable under the conjugation action of $\Omega$, we see that  
$$\br K\dot{\gamma}^{-1} \Adm(\mu) \dot{\gamma} \br K=\br K\Adm(\mu) \br K.
$$ Hence
\[
g^{-1}b\phi(g)\in \dot{\gamma} \br K  \Adm(\mu)   \br K\dot{\gamma}^{-1},
\]
 so
 \[
 \dot{\gamma}^{-1}g^{-1}b\phi(g)\phi(\dot{\gamma})\in \br K  \Adm(\mu)   \br K\dot{\gamma}^{-1}\phi(\dot{\gamma})\subset 
 \br K  \Adm(\mu)   \br K  ,
 \]
 since $\br K\dot{\gamma}^{-1}\phi(\dot{\gamma})=\br K.$
 Hence
 \[
 (g\dot{\gamma} )^{-1} b\phi(g\dot{\gamma})\in \br K  \Adm(\mu)   \br K ,
 \]
i.e.,  $g\dot{\gamma}\br K\in X_\CG(b, \mu)$, as had to be shown.
 \end{proof}
  Composing the above map with $ X_{\CG^o_\beta}(b, \mu)\to  X_{\CG_\beta}(b, \mu)$ and letting $\bar\beta$ vary, we obtain  the map 
 \begin{equation}\label{disjsum}
\pi: \bigsqcup_{\bar\beta\in \Pi_\CG}  X_{\CG^o_{\beta}}(b, \mu)\to X_\CG(b, \mu).
 \end{equation}
 In the sequel, we call a {\it component} of $X_{\CG^o_{\beta}}(b, \mu) $, resp. of $X_{\CG}(b, \mu) $, the non-empty fibers of the Kottwitz map to $\Omega$, resp. to $C_\CG=\Omega/\pi_0(\CG)$. In other words, we partition the points of $X_{\CG^o_{\beta}}(b, \mu) $, resp. of $X_{\CG}(b, \mu) $, according to which connected component of $X_{\CG^o_{\beta}}$, resp. of $X_{\CG}$, they lie in.

\begin{proposition}\label{separate} (i) For $\bar\beta_1\neq\bar\beta_2$, the images of $X_{\CG^o_{\beta_1}}(b, \mu) $ and $X_{\CG^o_{\beta_2}}(b, \mu) $ under $\pi$ fall into different  components
 of $X_\CG(b, \mu)$.
 \smallskip
 
 \noindent (ii) The map $\pi$ is surjective.

\end{proposition}  
\begin{proof}
For (i), let $x_1\in X_{\CG^o_{\beta_1}}(b, \mu)$ and  $x_2\in X_{\CG^o_{\beta_2}}(b, \mu)$ be such that their images in $C_\CG$ are the same. We deduce that
$$
\kappa_G(x_1)+\gamma_1= \kappa_G(x_2)+\gamma_2 \mod \pi_0(\CG) .
$$
On the other hand, $\kappa_G(x_i)=c_{b, \mu} -\lambda_i$ with $\lambda_i\in\Omega^\phi$, for $i=1, 2$. We therefore get 
$$
\gamma_1-\gamma_2=\lambda_1-\lambda_2+\mu , \quad \mu\in \pi_0(\CG) .
$$
Applying $1-\phi$ to this identity, we obtain
$$
\beta_1-\beta_2=(1-\phi)\mu ,
$$
i.e., $\bar\beta_1=\bar\beta_2$, as desired. 

 For (ii), let $g\br K\in X_\CG(b, \mu)$. By \cite[Lem. 5.10, (iii)]{GHN},  
  \begin{equation}
  \br K^o  \Adm(\mu)   \br K =\br K \Adm(\mu)   \br K .
 \end{equation}
Hence  $g^{-1}b\phi(g)\in \br K\Adm(\mu)\br K=\br K^o\Adm(\mu)\br K$. 
 Write accordingly
 \[
 g^{-1}b\phi(g)=k^o_1\omega k.
 \]
 Apply $\kappa: G(\br\BQ_p)\to \Omega$
 to get
 \[
 (\phi-1)\kappa(g)+\kappa(b)=\kappa(k)+\kappa(\mu) .
 \]
 Hence we get
 \begin{equation}
  (\phi-1)\kappa(g)=\kappa(k)-(\phi-1)c_{b,\mu} .
 \end{equation}
 Let $\beta=-\kappa(k)\in\pi_0(\CG)$ and denote by $\bar\beta\in \pi_0(\CG)_\phi$  its image. Then the last equation shows that $\bar\beta\in \Pi_\CG$. We  write $\beta=(1-\phi)\gamma$ with $\gamma\in\Omega$ and choose a lift $\dot\gamma\in N(\br\BQ_p)$ as in Lemma \ref{liftgamma}. Then  $ g\dot{\gamma}^{-1} (\dot{\gamma} K^o\dot{\gamma}^{-1})$ is in $X_{\dot{\gamma} K^o\dot{\gamma}^{-1}}(b, \mu)$
  and maps to $gK\in X_K(b, \mu)$. Indeed,  $(g\dot{\gamma}^{-1})^{-1} b\phi( g\dot{\gamma}^{-1})$ lies in $(\dot{\gamma}K^o\dot{\gamma}^{-1})\Adm(\mu)(\dot{\gamma}K^o\dot{\gamma}^{-1}) $ because 
  \[
  (g\dot{\gamma}^{-1})^{-1} b\phi( g\dot{\gamma}^{-1})=\dot{\gamma} g^{-1}b\phi(g)\phi( \dot{\gamma}^{-1})=  \dot{\gamma}k_1^o \omega k \phi( \dot{\gamma}^{-1})=   \]
  \[
  = (\dot{\gamma} k^o_1  \dot{\gamma}^{-1})(\dot{\gamma}\omega \dot{\gamma}^{-1}) (\dot{\gamma} k  \phi({\dot\gamma}^{-1})).
  \]
 The middle factor lies in $\Adm(\mu)$ because conjugation by an element in $\Omega$ preserves $\Adm(\mu)$. The last factor lies in $\dot{\gamma}\br K\dot{\gamma}^{-1} $. The result follows because the last factor lies even in $\dot{\gamma}\br K^o\dot{\gamma}^{-1} $ because $\kappa (\dot{\gamma} k  \phi({\dot\gamma}^{-1}))=(1-\phi)\gamma-\beta=0$.
 \end{proof}
 \begin{corollary}\label{imageka}
 The map $\kappa$ induces a surjective map $X_\CG(b, \mu)\to \bar c_{b, \mu}+C^\phi$. 
 \end{corollary}
 \begin{proof}
 Let $x=\bar c_{b, \mu}+y$, where $y\in C_\CG^\phi$. Let us first assume that  the image of $y$ in $\Pi_\CG$ is trivial, then there exists $\tilde y\in\Omega^\phi$ mapping to $y$. Then
 \[
 c_{b, \mu}+\tilde y\in \Omega 
 \]
 is a lift of $x$ which, by \eqref{coset}, can be lifted to a point of $X_{\CG^o}$ which then maps to a lift of $x$ in $X_{\CG}$, as required. 
 
 In general, let $\bar\beta\in\Pi_\CG$ be the image of $x$. 
Let $C^\phi\to\Pi_\CG$ be the natural surjective map from \eqref{exseq}, and let $(C^\phi)_{\bar\beta}$ be the fiber of this map over $\bar\beta\in\Pi_\CG$. Via our fixed choice of $\gamma$ for $\bar\beta$, this can be identified with $\Omega^\phi/\pi_0(\CG_\beta)^\phi$. By \eqref{coset},  the set of components of the inverse image of $\bar c_{b, \mu}+(C^\phi)_{\bar\beta}$ in $X_{\CG^o_{\beta}}(b, \mu)$ can then be identified with $ c_{b, \mu}+\Omega^\phi$. The assertion follows. 
 \end{proof}
 We introduce the intermediate group $K^o\subset K^\prime\subset K$, corresponding to the subgroup $\pi_0(\CG)^\phi$ of $\pi_0(\CG)$. We denote the corresponding group scheme by $\CG^\prime$. Then we obtain a factorization
 $$
 X_{\CG^o}\to X_{\CG'}\to X_\CG .
 $$
 Each of the maps induces an isomorphism of a component onto its image. In fact, the first map is the quotient by the finite abelian group $\pi_0(\CG)^\phi$, the second map is the quotient by the finite abelian group $\pi_0(\CG)/\pi_0(\CG)^\phi$, and the composed map is the quotient by the finite abelian group $\pi_0(\CG)$.  Here the actions of these covering groups are trivial, in the sense that they only permute connected components. 
 
 Similarly, for arbitrary $\bar\beta\in\Pi_\CG$, we introduce the intermediate subgroup $K_\beta^o\subset K_\beta^\prime\subset K_\beta$, corresponding to the subgroup $\pi_0(\CG_\beta)^\phi$ of $\pi_0(\CG_\beta)$. We obtain a sequence of maps,
 \begin{equation}
 X_{\CG_{\beta}^o}\to X_{\CG'_\beta}\to X_{\CG_\beta},
\end{equation}
where the first map is the quotient by the finite group $\pi_0(\CG_\beta)^\phi$,  the second map is the quotient by the finite abelian group $\pi_0(\CG_\beta)/\pi_0(\CG_\beta)^\phi$, and the composed map is the quotient by the finite abelian group $\pi_0(\CG_\beta)$. 
 \begin{proposition}\label{quasiADLV} 
  Let $(\CG, b, \mu)$ be an integral local shtuka datum such that   $\CG$ is a quasi-parahoric group scheme for $G$. 
 \smallskip
 
 \noindent (i) The ADLV $X_{\CG^o_{\beta}}(b, \mu)$ is invariant under the action of $\pi_0(\CG_\beta)^\phi$. 
 \smallskip
 
 \noindent (ii) The map \eqref{disjsum} induces an isomorphism 
 $$
 \bar\pi: \bigsqcup_{\bar\beta\in \Pi_\CG}  X_{\CG^o_{\beta}}(b, \mu)/\pi_0(\CG_\beta)^\phi\isoarrow X_\CG(b, \mu).
 $$
 \end{proposition}
 \begin{proof}
 (i) Let $h\in \br K_\beta^\prime$. Then $\phi(h)=hk$, with $k\in \br K_\beta^o$. The invariance follows from 
 $$
 (gh)^{-1}b \phi(gh)=h^{-1}g^{-1}b\phi(g)\phi(h)=h^{-1}g^{-1}b\phi(g)h k ,
 $$
 because the last element lies in $\br K_\beta^oh^{-1}\Adm(\mu)h \br K_\beta^o=\br K_\beta^o\Adm(\mu) \br K_\beta^o.$
 
 (ii) The surjectivity follows from Proposition \ref{separate}, (ii). For the injectivity, we note that by Proposition \ref{separate}, (i), this becomes a question one $\bar\beta$ at the time. But if the components of $X_{\CG^o_{\beta}}(b, \mu)$ corresponding to $\lambda_1$ and $\lambda_2$ in $\Omega$ go to the same component of $X_\CG(b, \mu)$, it follows that $\lambda_1-\lambda_2\in\pi_0(\CG)$. But by \eqref{coset}, this element lies in $\Omega^\phi$, hence also in $ \pi_0(\CG)^\phi$. 
 \end{proof}
 \begin{remark}
 Proposition \ref{separate} and Corollary \ref{imageka} imply  that the map $X_{\CG^o}(b, \mu)\to X_{\CG}(b, \mu)$ is surjective if and only if $\Pi_\CG$ is trivial. The question of the surjectivity of this map is also analyzed in \cite[\S 5]{GHN}. By \cite[Prop. 5.11]{GHN} it holds if there is a product decomposition $\Omega=\pi_0(\CG)\times C$ into subgroups stable under $\phi$ (then, of course, $\Pi_\CG=(0)$). In \cite[\S 5]{GHN} it is also analyzed when $X_{\CG^o}(b, \mu)$ is the full inverse image of $ X_{\CG}(b, \mu)$ under $X_{\CG^o}\to X_\CG$. By \cite[Prop. 5.13]{GHN} this holds when $\pi_0(\CG)^\phi=\pi_0(\CG)$ and $\Omega^\phi\to C_\CG^\phi$ is surjective (again, these conditions imply $\Pi_\CG=(0)$). 
 \end{remark}
 
 \subsection{Integral moduli spaces of local shtukas}
 We first define a morphism of integral LSV
 \begin{equation}\label{naivemap}
 \CM^{\rm int}_{\CG_{\beta}, b, \mu}\to \CM^{\rm int}_{\CG, b, \mu} .
 \end{equation}
Recall that
 $$
 \CM^{\rm int}_{\CG_\beta, b, \mu}(S)=\{\text{isomorphism classes of } (S^\sharp, \CP_\beta, \phi_{\CP_\beta}, i_{r, \beta}) \} .
 $$
 Starting with a point $(S^\sharp, \CP_\beta, \phi_{\CP_\beta}, i_{r, \beta})$ of $\CM^{\rm int}_{\CG_\beta, b, \mu}$, we define a  $\CG$-torsor $\CP$ by twisting the $\CG_\beta$-action by $\dot\gamma$,
 $$
 g\star x=( \dot\gamma g \dot\gamma^{-1})\cdot x, \quad g\in\CG ,
 $$
 where $\dot\gamma$ is the fixed choice above. 
 This makes sense since $\dot\gamma g\dot\gamma^{-1}\in\CG_\beta$. Let $\xi=\phi(\dot\gamma)^{-1}\dot\gamma$. Then $\xi\in \CG(\br\BZ_p)$, cf. Lemma \ref{liftgamma}. 
  We  define the new Frobenius $\phi_{\CP}$ as
 \begin{equation}\label{corrdef}
 \phi_\CP(x)=\xi\star\phi_{\CP_\beta}(x). 
 \end{equation} 
  This is indeed  a morphism of $\CG$-torsors: on the one hand 
 \begin{equation*}
 \begin{aligned}
 \phi_{\CP}(g\star x)&=\xi\star\phi_{\CP_\beta}(g\star x)=
 \dot\gamma\xi\dot\gamma^{-1} \phi_{\CP_\beta}(( \dot\gamma g  \dot\gamma^{-1})x)=\\
 &= \dot\gamma(\phi( \dot\gamma)^{-1} \dot\gamma) \dot\gamma^{-1}(\phi( \dot\gamma)\phi(g)\phi( \dot\gamma)^{-1})\phi_{\CP_\beta}(x)= \dot\gamma\phi(g)\phi( \dot\gamma)^{-1}\phi_{\CP_\beta}(x).
 \end{aligned}
 \end{equation*}
On the other hand,
\begin{equation*}
 \begin{aligned}
\phi(g)\star\phi_\CP(x)&=( \dot\gamma \phi(g)  \dot\gamma^{-1})( \dot\gamma\xi \dot\gamma^{-1})\phi_{\CP_\beta}(x)=\\
&= \dot\gamma \phi(g) (\phi( \dot\gamma)^{-1} \dot\gamma) \dot\gamma^{-1}\phi_{\CP_\beta}(x)
= \dot\gamma \phi(g) \phi( \dot\gamma)^{-1}\phi_{\CP_\beta}(x),
 \end{aligned}
 \end{equation*}
which proves the claim.

 We define the framing $i_r$ as  
 \begin{equation}
 i_r(x)=i_{r, \beta}( \dot\gamma x) .
 \end{equation}
 This is  indeed  an isomorphism of $\CG$-bundles $i_r: G_{\CY_{[r, \infty]}}\to \CP_{|\CY_{[r, \infty]}}$  : on the one hand 
 $$
 i_r(hg)=i_{r, \beta}( \dot\gamma hg)= \dot\gamma h  i_{r, \beta}(g).
 $$ On the other hand,
 $$
 h\star i_r(g)=( \dot\gamma h \dot\gamma^{-1}) i_{r, \beta}( \dot\gamma g)=( \dot\gamma h \dot\gamma^{-1}) \dot\gamma  i_{r, \beta}( g)= \dot\gamma h  i_{r, \beta}(g).
 $$
 The framing $i_r$ is also  compatible with the Frobenius: on the one hand,
 \begin{equation*}
 \begin{aligned}
 \phi_\CP(i_r(g))&=\xi\star\phi_{\CP_\beta}(i_{r, \beta}( \dot\gamma g))= \dot\gamma(\phi( \dot\gamma)^{-1} \dot\gamma) \dot\gamma^{-1}\phi_{\CP_\beta}(i_{r, \beta}(  \dot\gamma g))=\\
 &=   ( \dot\gamma\phi( \dot\gamma)^{-1})\phi( \dot\gamma) i_{r, \beta}( b\phi(g))=
  \dot\gamma i_{r, \beta}( b\phi(g)). 
 \end{aligned}
 \end{equation*} On the other hand,
 \begin{equation*}
 \begin{aligned}
 i_r(b\phi(g))=i_{r, \beta}(  \dot\gamma b\phi(g))= \dot\gamma i_{r, \beta}( b\phi(g)).
 \end{aligned}
 \end{equation*}

 The morphism \eqref{naivemap} is now defined by sending $(S^\sharp, \CP_\beta, \phi_{\CP_\beta}, i_{r, \beta})$ to $(S^\sharp, \CP, \phi_{\CP}, i_{r })$. We precompose  \eqref{naivemap}  with the natural morphism $ \CM^{\rm int}_{\CG^o_{\beta}, b, \mu}\to \CM^{\rm int}_{\CG_\beta, b, \mu}$ to obtain a morphism of integral LSV
 $$
 \CM^{\rm int}_{\CG^o_{\beta}, b, \mu}\to \CM^{\rm int}_{\CG, b, \mu} ,
 $$
 and hence a morphism
 \begin{equation}\label{mor421}
 \pi: \bigsqcup_{\bar\beta\in\Pi_\CG}\CM^{\rm int}_{\CG^o_{\beta}, b, \mu}\to \CM^{\rm int}_{\CG, b, \mu} .
 \end{equation}
 This morphism is compatible,  via the bijective map of Proposition \ref{MX}, with the morphism \eqref{disjsum} for $\mu^{-1}$. Indeed, let $(P_\beta, \phi_{P_\beta}, i_\beta)$ be an object of $\CM^{\rm int}_{\CG_\beta, b, \mu}(\Spd(k))$, which gives a $\CG_\beta$-torsor $P_\beta$ over $W(k)$, a Frobenius and a trivialization of $P_\beta$ over $W(k)[1/p]$. The image of $(P_\beta, \phi_{P_\beta}, i_\beta)$ in $X_{\CG_\beta}(b, \mu^{-1})$  is given by the unique element $g_\beta\in LG(k)/L^+\CG_\beta(k)$ such that $i_\beta^{-1}(P_\beta)=g_\beta^{-1} \cdot\CG_\beta$.  Let $(P, \phi_{P}, i)\in\CM^{\rm int}_{\CG, b, \mu}(k)$ be the  image of $(P_\beta, \phi_{P_\beta}, i_\beta)$ under \eqref{naivemap}. Then $i^{-1}(P)= \dot\gamma^{-1}i_\beta^{-1}(P_\beta)$. Hence the image $g\in LG(k)/L^+\CG(k)$ of  $(P, \phi_{P}, i)$ is equal to $g=g_\beta\gamma$, as required. 
 
  \begin{theorem}\label{quasiThm}
  Let $(\CG, b, \mu)$ be an integral local shtuka datum such that   $\CG$ is a quasi-parahoric group scheme for $G$. The morphism (\ref{mor421}) induces an isomorphism
 $$
\bar\pi:  \bigsqcup_{\bar\beta\in\Pi_\CG}\CM^{\rm int}_{\CG^o_{\beta}, b, \mu}/\pi_0(\CG_\beta)^\phi\isoarrow \CM^{\rm int}_{\CG, b, \mu}  ,
 $$
where the quotient is for an action of $\pi_0(\CG_\beta)^\phi$ which permutes connected components.
 \end{theorem}
 Let us first define the action of $\pi_0(\CG_\beta)^\phi$ on $\CM^{\rm int}_{\CG^o_{\beta}, b, \mu}$.  We use the exactness of the sequence  
 $$ 0\to \CG_{\beta}^o(\BZ_p)\to \CG_\beta(\BZ_p)\to \pi_0(\CG_\beta)^\phi\to 0 .
$$
Let $\delta\in G(\BQ_p)$ be in the normalizer of $\CG_\beta^o$. Recall that
 $$
 \CM^{\rm int}_{\CG_\beta^o, b, \mu}(S)=\{\text{isomorphism classes of } (S^\sharp, \CP, \phi_\CP, i_r) \} ,
 $$
 where $\CP$ is a $\CG_\beta^o$-torsor, etc. 
 We define a new $\CG_\beta^o$-torsor $\CP'=\CP_\delta$ by twisting the $\CG_\beta^o$-action by $\delta$,
 $$
 g\star x=(\delta g\delta^{-1})\cdot x .
 $$
 The new Frobenius $\phi_{\CP'}$ is taken to be identical to $\phi_\CP$. This is indeed a morphism of $\CG_\beta^o$-torsors since
 $$
 \phi_{\CP'}(g\star x)=\phi_\CP((\delta g\delta^{-1})\cdot x)=(\delta \phi(g)\delta^{-1})\cdot\phi_\CP(x)=\phi(g)\star\phi_{\CP'}(x),
 $$
 where we used that $\delta\in G(\BQ_p)$. Finally, the framing $i'_r$ is given by $i'_r(g)=i_r(\delta g)$. This is indeed an isomorphism of $\CG_\beta^o$-bundles, since
 $$
 i'_r(hg)=i_r(\delta hg)=i_r((\delta h\delta^{-1})(\delta g))=(\delta h\delta^{-1})\cdot i_r(\delta g)=h\star i'_r(g). 
 $$
 The compatibility with the Frobenius follows from
 $$
 \phi_{\CP'}(i'_r(g))=\phi_\CP(i_r(\delta g))=\delta \phi_\CP(i_r(g))=\delta i_r(b\phi(g))=i'_r(b\phi(g)) .
 $$

 Note that if $\delta\in\CG_\beta^o(\BZ_p)$, then the tuple $(S^\sharp, \CP', \phi_{\CP'}, i'_r)$ is isomorphic to $(S^\sharp, \CP, \phi_\CP, i_r)$. Indeed, the map $x\mapsto \delta\cdot x$ defines an isomorphism $\alpha:\CP\to \CP'$ compatible with $\phi_\CP$ and $\phi_{\CP'}$, and with $i_r$ and $i'_r$.  For instance
 $$
 \phi_{\CP'}(\alpha(x))=\phi_\CP(\delta\cdot x)=\phi(\delta)\cdot\phi_\CP(x)=\delta\cdot\phi_\CP(x)=\alpha(\phi_\CP(x)) .
 $$
 Therefore we obtain an action of the factor group $N_{G(\BQ_p)}(\CG_\beta^o)/\CG_\beta^o(\BZ_p)$ on $\CM^{\rm int}_{\CG^o, b, \mu}$. Restricting this action to $\CG_\beta(\BZ_p)/\CG_\beta^o(\BZ_p)$, we obtain the desired action of $\pi_0(\CG)^\phi$ on $\CM^{\rm int}_{\CG^o, b, \mu}$. 
 
 The same argument as the one above shows that $\delta\in\CG_\beta(\BZ_p)$ induces an isomorphism between the images of $(S^\sharp, \CP, \phi_{\CP}, i_r)$ and $(S^\sharp, \CP', \phi_{\CP'}, i'_r)$ in $\CM^{\rm int}_{\CG, b, \mu}$ under \eqref{naivemap}, hence we obtain a morphism $\CM^{\rm int}_{\CG_\beta^o, b, \mu}/\pi_0(\CG_\beta)^\phi\to \CM^{\rm int}_{\CG, b, \mu}$. Letting now $\bar\beta$ vary, we obtain a morphism from the LHS to the RHS in Theorem \ref{quasiThm}. 
  \smallskip
  
\noindent\begin{proof}[{\it Proof of Theorem \ref{quasiThm}}:]
 We deduce from Proposition \ref{MX} and Proposition \ref{quasiADLV} that the morphism induces a bijection on the set of   $\Spd(\kappa)$-points, for all algebraically closed extensions $\kappa/k$. 
 Also it is qcqs by Proposition \ref{neutralIso} (a). Using relative properness and \cite[Cor. 17.4.10]{Schber} we see it is enough to check that  it gives a bijection on  $\Spa(C, O_C)$-points.  Note that  any $\Spa(C, O_C)$-point $\wt x$ of $\CM^{\rm int}_{\CG, b, \mu}$ factors through the formal completion 
$ (\CM^{\rm int}_{\CG, b, \mu})_{O(k_C)/\bar x} $ of the base change of 
$\CM^{\rm int}_{\CG, b, \mu}$ by $O(k_C)=O_{\br E}\otimes_{W(k)}W(k_C)$, at its specialization at $\bar x={\rm sp}(\wt x)$. Here, $k_C=O_C/\fkm_C$ is the residue field. The corresponding fact is also true for $\Spa(C, O_C)$-points of $\CM^{\rm int}_{\CG_\beta^o, b, \mu}$.  

We can now complete the proof. First we show injectivity: Suppose $\wt x_i$, $i=1,2$, are two points in $(\sqcup_{\bar\beta\in \Pi_\CG}\CM^{\rm int}_{\CG_\beta^o, b, \mu}/\pi_0(\CG_\beta)^\phi)(C, O_C)$, mapping by $\bar\pi$
to the same point in $\CM^{\rm int}_{\CG, b, \mu}(C, O_C)$. By Proposition \ref{quasiADLV}, 
these two points have the same specialization $\bar x=\bar x_1=\bar x_2$. By the above, they factor through the  formal completion of (the base change) of $\CM^{\rm int}_{\CG_\beta^o, b, \mu}$ at a corresponding $\bar x$, for some common $\bar\beta\in \Pi_\CG$. By (\ref{largepointeq}) for $\kappa=k_C$, we see that these two points agree in the formal completion of $\CM^{\rm int}_{\CG_\beta^o, b, \mu}$ at $\bar x$, and therefore $\wt x_1=\wt x_2$. The proof of surjectivity  onto $\CM^{\rm int}_{\CG, b, \mu}(C, O_C)$ is by a similar argument.
 \end{proof} 
 
 We define a map of $v$-sheaves
 $$
 \kappa_G\circ {\rm sp}\colon \CM^{\rm int}_{\CG, b, \mu}\to \underline{C_{\CG}} ,
 $$
 where ${\rm sp}\colon |\CM^{\rm int}_{\CG, b, \mu}|\to  |X_\CG(b, \mu^{-1})|$ denotes the (continuous) specialization map under the identification $(\CM^{\rm int}_{\CG, b, \mu})_\red=X_{\CG}(b, \mu^{-1} )$, cf. Proposition \ref{MX}. 
 For $\tau\in C_{\CG}$,  we denote by
\begin{equation}\label{comptau}
 \CM^{\rm int,\tau}_{\CG, b, \mu}=   \CM^{\rm int}_{\CG, b, \mu}\times_{\underline{C_{\CG}}}\underline{\{\tau\}}\subset  \CM^{\rm int}_{\CG, b, \mu}
\end{equation}
 the fiber over $\tau$. This is an open and closed $v$-subsheaf of  $\CM^{\rm int}_{\CG, b, \mu}$. We call $ \CM^{\rm int,\tau}_{\CG, b, \mu}$ the \emph{component of $ \CM^{\rm int}_{\CG, b, \mu}$} corresponding to $\tau$. 
 \begin{corollary}\label{compwise}
  Let $(\CG, b, \mu)$ be an integral local shtuka datum such that   $\CG$ is a quasi-parahoric group scheme for $G$ and let $\CG\to\CG'$ be a morphism  extending the identity morphism of $G$ in the generic fibers of quasi-parahoric group schemes for $G$ with the same associated parahoric group scheme.  Then for every $\tau\in \Omega_G/\pi_0(\CG)$, with image $\tau'\in  \Omega_G/\pi_0(\CG')$, the natural morphism $\CM^{\rm int}_{\CG, b, \mu}\to \CM^{\rm int}_{\CG', b, \mu}$ induces an isomorphism
  $$ \CM^{{\rm int}, \tau}_{\CG, b, \mu}\isoarrow \CM^{{\rm int}, \tau'}_{\CG', b, \mu}    $$
  of $v$-sheaves. \qed
 \end{corollary}
    
 \smallskip
 
\noindent {\it Proof of Theorem \ref{varyGthm}}: For the representability conjecture, we assume that $(G, b, \mu)$ is a local Shimura datum, i.e., $\mu$ is minuscule. Now $\CM^{{\rm int}}_{\CG, b, \mu}$ is representable if and only if $\CM^{{\rm int}, \tau}_{\CG, b, \mu}$ is representable for every $\tau\in \Omega_G/\pi_0(\CG)$. Hence the assertion follows from Corollary \ref{compwise} and the fact that every component $\CM^{{\rm int}, \tau}_{\CG, b, \mu}$ is isomorphic to a component $\CM^{{\rm int}, \wt\tau}_{\CG^o, b, \mu}$ of $\CM^{{\rm int}}_{\CG^o, b, \mu}$ for $\wt\tau\in\Omega_G$ mapping to $\tau$ and, conversely, every component $\CM^{{\rm int}, \tau}_{\CG^o, b, \mu}$ is isomorphic to the component $\CM^{{\rm int}, \bar\tau}_{\CG, b, \mu}$ of $\CM^{{\rm int}}_{\CG, b, \mu}$, where $\bar\tau\in \Omega_G/\pi_0(\CG)$ is the image of $\tau$.
\qed
 \smallskip
 
We also mention the following naturality statement which is used in \S \ref{s:proofs}. It concerns the action of the Frobenius centralizer $J_b(\BQ_p)$ on $\CM^{{\rm int}}_{\CG, b, \mu}$. Recall that this action is via changes of the framing, i.e.,
 \[
 g: (S^\sharp, \CP, \phi_{\CP}, i_{r})\mapsto (S^\sharp, \CP, \phi_{\CP}, i_{r}\circ g^{-1}), \quad g\in J_b(\BQ_p) .
  \]
 This action is compatible with the action of $J_b(\BQ_p)$ on $\bar c_{b, \mu}+C^\phi$, via 
 \[
g: \tau\mapsto \tau +\tilde\kappa(g) ,
 \]
 where $\tilde\kappa$ is induced by the composition of maps $J_b(\BQ_p)\hookrightarrow G(\breve \BQ_p)\xrightarrow{\ \kappa_G\ }\Omega_G$.
 \begin{proposition}\label{actJ}
 Assume that the center of $G$ is connected. Then the action of $J_b(\BQ_p)$ on $\Omega_G^\phi$ is transitive. In particular, if $\CG$ is a parahoric, for any two $\tau, \tau'\in\bar c_{b, \mu}+\Omega_G^\phi$, the v-sheaves $\CM^{{\rm int}, \tau}_{\CG, b, \mu}$ and $\CM^{{\rm int}, \tau'}_{\CG, b, \mu}$ are isomorphic. 
 \end{proposition}
 \begin{proof}
 If $b$ is basic, then $J_b$ is an inner twist of $G$. Hence $\pi_1(J_b)_{I}^\phi=\pi_1(G)_{I}^\phi$, and the assertion follows from the surjectivity of the map $\kappa_H: H(\BQ_p)\to \pi_1(H)_{I}^\phi$, valid for any reductive group $H$ over $\BQ_p$, cf. \cite[\S 7.1]{KotIsoII}. 
 
 Now let $b$ be arbitrary. Let us first assume that $G$ is quasisplit, and fix a maximal split torus $A$ and a Borel subgroup of $G$ containing $A$, so that we have the notion of standard parabolics and standard Levi subgroups of $G$. Then there exists a standard Levi subgroup $M$ and a basic element $b_M\in M(\breve\BQ_p)$ which is $\phi$-conjugate to $b$ and such that we have an equality of $\phi$-centralizer groups $J_{M, b_M}=J_b$, cf. \cite[Prop. 6.2]{KotIsoI}. By the basic case treated above, we are reduced to proving the surjectivity of the map $\pi_1(M)_{I}^\phi\to \pi_1(G)_{I}^\phi$. Consider the surjective map
 \[
 \pi_1(M)\to\pi_1(G) .
 \]
 The kernel of this map  is equal to $Q:={\rm coker}(X_*(T_{M, {\rm sc}})\to X_*(T_{\rm sc}))$, where $T$, resp. $T_M$, denotes the centralizer of $A$ in $G$, resp. in $M$, and where the index ``sc'' denotes the inverse image in the simply connected cover of the derived group of $G$, resp. $M$. But $X_*(T_{ {\rm sc}})$, resp. $X_*(T_{M, {\rm sc}})$, has as basis the set of coroots of $G$, resp. of $M$. The coroots of $G$ which are not coroots of $M$ give a permutation basis for the representation of ${\rm Gal}(\bar\BQ_p/\BQ_p)$ on $Q$. It now follows that the above map continues to be surjective after first taking the coinvariants under the inertia $I$ and then the invariants under $\phi$.
 
 Finally, let us drop the assumption that $G$ is quasi-split. Let $G_0$ be the quasi-split inner form of $G$. Since the center of $G$ is connected, there exists $\gamma\in G_0(\breve \BQ_p)$ and an isomorphism $\Psi: G_0\otimes\breve\BQ_p\xrightarrow{\sim} G\otimes\breve\BQ_p$ such that the bijection $g\mapsto \Psi(\gamma g)$ is equivariant for the  action of the Frobenius on $G_0(\breve \BQ_p)$, resp. $G(\breve\BQ_p)$. Let $b_0$ be the preimage of $b$. Then this map induces an isomorphism of the $\phi$-centralizers $J_{0, b_0}$ and $J_b$. From the quasi-split case, we deduce a surjection 
 \[
 J_b(\BQ_p)=J_{0, b_0}(\BQ_p)\to\pi_1(G_0)_{I}^\phi=\pi_1(G)_{I}^\phi,
 \]
 as desired.
 \end{proof}

\subsection{Generic fibers}\label{ss:genfib}  In this subsection, we interpret the induced decomposition in Theorem \ref{quasiThm} in the generic fibers, in the case when $\mu$ is minuscule. Let $(G, b, \mu)$ be a local Shimura datum and let $\CG$ be a  quasi-parahoric for $G$. We  consider the generic fiber
\[
(\CM^{\rm int}_{\CG,b,\mu})_\eta=\CM^{\rm int}_{\CG,b,\mu}\times_{\Spd(O_{\br E})}\Spd(\br E)
\]
of $\CM^{\rm int}_{\CG,b,\mu}$.  Theorem \ref{quasiThm} implies
\begin{equation}\label{quasiGeneric}
(\CM^{\rm int}_{\CG,b,\mu})_\eta \simeq \bigsqcup_{\bar\beta\in \Pi_\CG} (\CM^{\rm int}_{\CG^o_{\beta},b,\mu})_\eta/\pi_0(\CG_\beta)^\phi
\end{equation}
Since $\CG^o_{\beta}$ is parahoric, it follows that 
 \[
 (\CM^{\rm int}_{\CG^o_{\beta},b,\mu})_\eta={\rm Sht}^\diam_{\CG^o_{\beta}(\BZ_p)}(G, b, \mu), 
 \]
 cf.  \cite[\S 23.3, \S 25.3]{Schber}. Here, $ {\rm Sht}^\diam_K(G, b, \mu)$ denotes, for any compact open subgroup $K\subset G(\BQ_p)$, 
 the diamond local Shimura variety, represented by a smooth rigid analytic variety  over $\Sp(\br E)$, cf. \cite[\S 24.1]{Schber}. We use the  notation 
 $ {\rm Sht}_K(G, b, \mu)$ for the representing rigid-analytic variety. 
The morphism
 \[
  {\rm Sht}_{\CG^o_{\beta}(\BZ_p)}(G, b, \mu)\to  {\rm Sht}_{\CG_{\beta}(\BZ_p)}(G, b, \mu)
\]
is a Galois \'etale cover for the finite abelian group
\[
\CG_{\beta}(\BZ_p)/\CG^o_{\beta}(\BZ_p)=(\CG_{\beta}(\br\BZ_p)/\CG^o_{\beta}(\br\BZ_p))^\phi=\pi_0(\CG_\beta)^\phi ,
\]
cf. \cite[23.3]{Schber}. We deduce the following identification  which is the desired interpretation of the decomposition in Theorem \ref{quasiThm} in the generic fiber. 
\begin{theorem} \label{quasiGeneric2}
 Let $(G, b, \mu)$ be a local Shimura datum and let $\CG$ be a quasi-parahoric for $G$. Then 
\begin{equation*}
(\CM^{\rm int}_{\CG,b,\mu})_\eta \simeq \bigsqcup_{\bar\beta\in \Pi_\CG} {\rm Sht}^\diam_{\CG_{\beta}(\BZ_p)}(G, b, \mu)  ,
\end{equation*}
In particular, if $\CM^{\rm int}_{\CG,b,\mu}$ is representable by the formal scheme $\sM_{\CG,b,\mu}$ over $O_{\br E}$, then
\begin{equation*}
\sM_{\CG,b,\mu}^{\rm rig} \simeq \bigsqcup_{\bar\beta\in \Pi_\CG}{\rm Sht}_{\CG_{\beta}(\BZ_p)}(G, b, \mu) .
\end{equation*}\qed
\end{theorem}

\subsection{Interpretation in terms of local systems} Consider a perfectoid $S\to \Spd(\br E)$, i.e. with untilt $S^\sharp$ in char. $0$.
Suppose $r>0$ is sufficiently small, so that $\CY_{[0, r]}(S)$ avoids the divisor corresponding to the untilt.
A $\CG$-shtuka over $S$ with leg at $S^\sharp$, restricts to a vector bundle over $\CY_{[0, r]}(S)$ with an isomorphism covering the action of $\phi^{-1}$ on $\CY_{[0, r]}(S)$; as in \cite[\S 22.1]{Schber}, we call this a $\phi^{-1}$ module over $\CY_{[0,r]}(S)$.
We see that 
an $S$-point of $\CM^{\rm int}_{\CG,b,\mu}$ over $\Spd(\br E)$ gives an exact tensor functor
\begin{equation*}
{\rm Rep}_{\BZ_p}(\CG)\to \text{ $\phi^{-1}$-${\rm mod}/\ \CY_{[0, r]}(S)$}
\end{equation*} 
On the other hand, by \cite[Prop. 22.3.2]{Schber}, we have an exact equivalence of  tensor categories, 
\begin{equation*}
\text{ $\phi^{-1}$-${\rm mod}/\ \CY_{[0, r]}(S)$}\isoarrow \text{ $\und\BZ_p$-${\Loc}(S)$} . 
\end{equation*}
Here, $\text{ $\und\BZ_p$-${\Loc}(S)$}$ stands for the category of pro-\'etale local systems for $\und\BZ_p$ over $S$, i.e. $\und\BZ_p$-torsors for the
pro-\'etale topology on $S$. 

Composing these two functors, we obtain an exact tensor functor
\begin{equation}\label{tensfun}
{\rm Rep}_{\BZ_p}(\CG)\to \text{ $\und\BZ_p$-${\Loc}(S)$}  . 
\end{equation}
The composition of this functor with $\text{ $\und\BZ_p$-${\Loc}(S)$}
\to \text{ $\und\BQ_p$-${\Loc}(S)$}$ is pro-\'etale locally on $S$ 
isomorphic to the forgetful functor (so, pro-\'etale locally, it gives the trivial torsor), cf. \cite[Thm. 22.5.2]{Schber}. 

This leads to an alternative description of the generic fiber analogous to  \cite[Prop. 23.3.1]{Schber}.
In the statement below, $X_{FF, S}:=\CY_{(0,\infty)}(S)/\phi^\BZ$ denotes the Fargues-Fontaine curve over $S$, \cite{FS}.  We will denote by $\CE^b\times_{\Spd(k)}S$, or simply $\CE^b$, the descent of $(G_{\CY_{(r,\infty)}(S)}, \phi_b)$ to $X_{FF, S}$.

\begin{proposition}\label{prop431}
The $S$-points of $(\CM^{\rm int}_{\CG,b,\mu})_\eta$ over $\Spd(\br E)$ are in bijection with isomorphism classes of $5$-tuples $(S^\sharp, \CE, \alpha, \BP, \iota)$ where

\begin{itemize}
\item $S^\sharp$ is an untilt of $S$ over $\br E$;

\item $\CE$ is a $G$-torsor over $X_{FF, S}$, trivial at every geometric point of $S$;

\item $\alpha$ is a $G$-torsor isomorphism
\[
\CE_{|{X_{FF, S}\setminus S^\sharp}}\xrightarrow{\ \sim } {\CE^b}_{|{X_{FF, S}\setminus S^\sharp}},
\]
which is meromorphic along $S^\sharp$ and bounded by $\mu$;

\item $\BP$   is an exact tensor functor 
\[
 {\rm Rep}_{\BZ_p}(\CG)\to \text{ $\und\BZ_p$-${\Loc}(S)$};
\]

\item $\iota$ is an isomorphism of the base change of $\BP$ to $\BQ_p$ with the functor given by the $\und {G(\BQ_p)}$-torsor associated to $\CE$ by \cite[Thm. 22.5.2]{Schber}.

\end{itemize}
\end{proposition}
\begin{proof}
This follows by the argument in the proof of \cite[Prop. 23.3.1]{Schber}. Observe
that, since we do not assume that $\CG$ has connected fibers, $\CG$-torsors over $\Spec(\BZ_p)$ are not necessarily trivial.
So instead of $\und {\CG(\BZ_p)}$-torsors as in loc. cit., we just obtain functors $\BP$ as above, compare also to
\cite[proof of Prop. 22.6.1]{Schber}.
\end{proof}
We can use Proposition  \ref{prop431} to reinterpret the isomorphism in Theorem \ref{quasiGeneric2}. 
For $S$ perfectoid over $k$, we can consider the fiber functor at  $s: \Spa(C, O_C)\to S$, 
 \[
F_s: \text{ $\und\BZ_p$-${\Loc}(S)$}\to \text{$\BZ_p$-${\rm mod}$}.
\]
Given
an $S$-point of $\CM^{\rm int}_{\CG,b,\mu}$ over $\Spd(\br E)$ as above,
the composition of \eqref{tensfun} and $F_s$ defines an exact tensor functor 
\begin{equation}
{\rm Rep}_{\BZ_p}(\CG)\to  \text{$\BZ_p$-${\rm mod}$}.
\end{equation}
It defines a $\CG$-torsor $\CT_s$ and we can consider the contracted product
\[
\CG_{s}:={\rm Aut}(\CT_s)=\CG\wedge^\CG \CT_s ,
\]
where $\CG$ acts on itself by conjugation.   Since $\CG$ is smooth, we have
\[
\CG_{s}\otimes_{\BZ_p}\br\BZ_p\simeq \CG\otimes_{\BZ_p}\br\BZ_p.
\]
The class $[\CT_s]$ of the torsor $\CT_s$ lies in $\ker ({\rm H}^1(\BZ_p, \CG)\to {\rm H}^1(\BQ_p, G))=\Pi_\CG$, comp. Lemma \ref{lemma311}. 

 A $4$-tuple $(S^\sharp, \CE, \alpha, \BP)\in (\CM^{\rm int}_{\CG,b,\mu})_\eta(S)$ defines now the  function 
 \begin{equation}\label{interpret}
 S\to \Pi_\CG , \quad s\mapsto [\CT_{s}].
 \end{equation}
   Assuming  a theory of pro-\'etale systems in this situation that resembles the classical theory
of (\'etale) $\BZ_p$-local systems, it would follow that the isomorphism class of the torsor $\CT_{s}$ is constant on connected components of $S$. More precisely, if we choose another point $s'$ of $S$, then $F_s\simeq F_{s'}$ (by a ``path connecting $s$ to $s'$) and this should give $\CT_{s}\simeq \CT_{s'}$ and $\CG_{s}(\BZ_p)=\CG_{s'}(\BZ_p)$ as subgroups of $G(\BQ_p)$. This would explain how to associate to a point  $s\in (\CM^{\rm int}_{\CG,b,\mu})_\eta(C, O_C)$ the index  $\bar\beta(s)\in\Pi_\CG$ on the RHS of \eqref{quasiGeneric}.  It would also explain how to associate to $s$ the point of ${\rm Sht}^\diam_{\CG_{\beta}(\BZ_p)}(G, b, \mu)$ on  the RHS of \eqref{quasiGeneric} (note that $\CG_s=\CG_{\beta(s)}$).

\subsection{Comparison with Wintenberger's theorem}\label{classPoints}

Let us discuss ``classical" points $x: \Spd(F)\to (\CM^{\rm int}_{\CG,b,\mu})_\eta$, where $F/\br E$ is a finite field extension. Consider  the crystalline period map
$$
\pi_{\rm GM}\colon (\CM^{\rm int}_{\CG,b,\mu})_\eta\to (X_\mu\times_{\Sp (E)} \Sp(\br E))^\diam.
$$
Under the identification of Theorem \ref{quasiGeneric2}, it is the disjoint sum of the period maps, one for each $\bar\beta\in\Pi_\CG$,
\begin{equation}
\pi_{{\rm GM}, \beta}\colon {\rm Sht}_{\CG_{\beta}(\BZ_p)}(G, b, \mu)\to X_\mu\times_{\Sp (E)} \Sp(\br E) .
\end{equation}
On the image of $\pi_{{\rm GM}, \beta}$ (the admissible locus), there is the corresponding local   $G(\BQ_p)$-system. Let $x\in {\rm Sht}_{\CG_{\beta}(\BZ_p)}(G, b, \mu)(F)$. Specializing this local system at $\pi_{{\rm GM}, \beta}(x)$ defines a Galois representation 
\begin{equation}\label{rhox}
\rho_x: {\rm Gal}(\bar F/F)\to K_\beta=\CG_{\beta}(\BZ_p)\subset G(\BQ_p).
\end{equation}
Recall the theorem  of Wintenberger (comp. \cite[Prop. 3.3.4]{KP}), according to which 
the image of the Galois representation $\rho_x$ lies in the neutral component $\CG^o_{\beta}(\BZ_p)$, i.e.,
\begin{equation}
\rho_x: {\rm Gal}(\bar F/F)\to \CG^o_{\beta}(\BZ_p)\subset G(\BQ_p).
\end{equation}
 This may be surprising at first glance, given the construction of \eqref{rhox}. However, consider the Galois \'etale cover for the group $\pi_0(\CG)^\phi$, 
$$
\pi: {\rm Sht}_{\CG^o_{\beta}(\BZ_p)}(G, b, \mu)\to {\rm Sht}_{\CG_{\beta}(\BZ_p)}(G, b, \mu).
$$
By our earlier discussion on connected components, this cover is totally split. This fits with the above result of Wintenberger  which directly
implies that, for every $F$-valued point of ${\rm Sht}_{\CG_{\beta}(\BZ_p)}(G, b, \mu)$, the inverse image $\pi^{-1}(x)$ is split over $F$.

\section{Ad-isomorphisms and integral moduli spaces of local shtukas}\label{s:adiso}

\subsection{Ad-isomorphisms} 
 
Recall from \cite{KotIsoII} that a homomorphism $f: G\to G'$ is an \emph{ad-isomorphism} if the center $Z_G$ of $G$ is mapped to the center $Z_{G'}$ of $G'$ and the induced homomorphism $G_\ad\to G'_\ad$ between the adjoint groups is an isomorphism. Equivalently, $\ker f$ is a central subgroup and $\Im f$ is a normal subgroup of $G'$ with torus cokernel. 
\begin{lemma}\label{adisolem}
An ad-isomorphism is the composition of ad-isomorphisms of one of the following types.
\begin{altenumerate}
\item $f: G\to G'$ is the inclusion of a normal subgroup with torus cokernel.
\item $f: G\to G'$ is a surjection with kernel a central torus which is   a product of induced tori
(a ``quasi-trivial" torus). 
\end{altenumerate}
\end{lemma}
\begin{proof} 
 Let $f:G\to G'$ be an ad-isomorphism, and let $Z=\ker f$. Find an embedding $Z\hookrightarrow T$ where $T$ is a quasi-trivial torus. Consider the push out of
\[
1\to Z\to G\to G/Z\to 1
\]
by $Z\hookrightarrow T$. This gives
\[
1\to T\to G''\to G/Z\to 1
\]
and we can view $G\to G'$ as a composition of ad-isomorphisms
\[
G\hookrightarrow G''\to G/Z\hookrightarrow G'
\]
where the first map is of type (i), the second map of type (ii) and the third map of type (i). 
\end{proof}

In this section, we consider an  ad-isomorphism $f: G\to G'$ and an  extension of $f$  to a morphism $f: \CG\to \CG'$ between quasi-parahoric group schemes. It is assumed that the  corresponding parahoric group schemes (the neutral connected components) $\CG^o$ and $ \CG'^o$ correspond to the same point in the common building of $G_\ad$ and $G'_\ad$. Let $(\CG, b, \mu)$ and $(\CG', b', \mu')$ be two integral local shtuka data such that $f(b)=b'$ and $\{\mu'\}=f(\{\mu\})$.  In this case, we have an inclusion of reflex fields $E'\subset E$. We then say that $f$ is an \emph{ad-isomorphism of integral local shtuka data}. We obtain a morphism 
\begin{equation}\label{adonLSV}
f: \CM^{\rm int}_{\CG, b, \mu}\xrightarrow{\ \ }  \CM^{\rm int}_{\CG', b', \mu'}\times_{\Spd(O')}\Spd(O)
\end{equation}
where $O=O_{\br E}$ and $O'=O_{\br E'}$. Consider the commutative diagram 
\begin{equation}
\begin{aligned}
 \xymatrix{
      | \CM^{\rm int}_{\CG, b, \mu}| \ar[r]^{\ \ } \ar[d]_{\rm sp} &  |\CM^{\rm int}_{\CG', b', \mu'}\times_{\Spd(O')}\Spd(O)|   \ar[d]^{\rm sp}\\
        |X_{\CG}(b, \mu^{-1})| \quad\ar[r]^{\ \ }  & |X_{\CG'}(b', \mu'^{-1})|.
        }
        \end{aligned}
\end{equation}
 Composing with the Kottwitz homomorphisms \eqref{kappaCG} gives a commutative diagram
\begin{equation}
\begin{aligned}
 \xymatrix{
        \CM^{\rm int}_{\CG, b, \mu}  \ar[r]^{\ \ } \ar[d]_{\kappa\circ\rm sp} &   \CM^{\rm int}_{\CG', b', \mu'}\times_{\Spd(O)}\Spd(O')    \ar[d]^{\kappa\circ\rm sp}\\
       \underline{C_{\CG}} \quad\ar[r]^{\ \ }  & \underline{C_{\CG'}}.
        }
        \end{aligned}
\end{equation}
Recall the components $\CM^{\rm int,\tau}_{\CG, b, \mu}$ for $\tau\in C_\CG$, resp.  $\CM^{\rm int,\tau'}_{\CG', b', \mu'}$ for $\tau'\in C_{\CG'}$, cf. \eqref{comptau}. 
 The aim of this section is to prove the following theorem.  
\begin{theorem}\label{thm451} Let $f\colon (\CG, b, \mu)\to (\CG', b', \mu')$ be an ad-isomorphism of integral local shtuka data. 
The  morphism \eqref{adonLSV} induces an isomorphism, 
\[
f: \CM^{\rm int,\tau}_{\CG, b, \mu}\xrightarrow{\sim} \CM^{\rm int,\tau'}_{\CG', b', \mu'}\times_{\Spd(O')}\Spd(O),
 \]
 where $\tau'=f(\tau)\in C_{\CG'}$. 
\end{theorem}
The proof proceeds in three steps. In subsection \ref{ss:adformcom} we show that $f$ induces an isomorphism of formal completions; in subsection \ref{ss:adADLV} we analyze the map on ADLV induced by $f$; in subsection \ref{ss:qcqs} we complete the proof by showing that the map in question is qcqs.

We note first that for ad-isomorphisms of type (i) and for the claim of qcqs, there is no need to consider components separately. 

\begin{proposition}\label{typeI}
Assume that $f$ is of type (i). Then the morphism
\[
f: \CM^{\rm int}_{\CG, b, \mu}\xrightarrow{\ \ }  \CM^{\rm int}_{\CG', b', \mu'}\times_{\Spd(O')}\Spd(O) 
\]
is qcqs. If in addition $\CG(\br\BZ_p)=G(\br\BQ_p)\cap f^{-1}(\CG'(\br\BZ_p))$, it is a closed immersion of $v$-sheaves.
\end{proposition}

\begin{proof} We first show that the morphism is qcqs. Quasi-separateness follows  from Lemma \ref{lemmaQS} below and 
the fact that $ \CM^{\rm int}_{\CG, b, \mu}\to \Spd(O_{\br E})$ is qs (cf. \cite[Prop. 2.25]{Gl21}). It remains to show quasi-compactness and the additional statement about closed immersions.
Under our assumption of type (i), we have a short exact sequence
\[
1\to G\to G'\to T\to 1
\]
with $T$ a torus. This gives an exact sequence 
\[
1\to \pi_1(G)\to \pi_1(G')\to X_*(T)\to 0.
\]
By taking $I$-coinvariants we obtain an exact sequence 
\[
 \Omega_G=\pi_1(G)_I\to \Omega_{G'}=\pi_1(G')_I\to X_*(T)_I\to 0
\]
where  the kernel of the first map is a finite abelian group. Hence, the inverse image $Y$ of $\pi_0(\CG')\subset \Omega_{G'}$ in $\Omega_G$ is a finite group which contains the subgroup $\pi_0(\CG)$ 
\[
\pi_0(\CG) \subset Y\subset \Omega_G.
\]
Using this and Bruhat-Tits theory, cf. the construction of \cite[Prop. (4.6.18)]{BT2}, we obtain a quasi-parahoric 
$\CG_1$ for $G$ with neutral component $\CG^o_1=\CG^o$ and $\pi_0(\CG_1)=Y$.
 Then, $\CG_1(\br \BZ_p)=G(\br\BQ_p)\cap f^{-1}(\CG'(\br\BZ_p))$. We have
\[
\CG^o_1=\CG^o\subset \CG\subset \CG_1
\]
and $\CG\to \CG'$ factors as $\CG\subset \CG_1\to \CG'$.
By applying Proposition \ref{neutralIso} to $\CG$ and $ \CG_1$, we see that it is enough to show quasi-compactness  for the morphism induced by $\CG_1\to \CG'$. Hence  we may assume from the start that $\CG(\br\BZ_p)=G(\br\BQ_p)\cap f^{-1}(\CG'(\br\BZ_p))$.
Then, as in \cite[Lem. 3.6.1]{PRglsv}, $f: \CG\to \CG'$ is a dilated immersion of smooth group schemes,
i.e. $f$ identifies $\CG$ with the Neron smoothening of the Zariski closure $\bar\CG$ of $f(G)$ in $\CG'$. Now the proof of \cite[Prop. 3.6.2]{PRglsv} applies to show that under this assumption the morphism is a closed immersion.
\end{proof}

\subsection{Ad-isomorphisms and generic fibers}\label{ss:adgeneric}

If $(G, b, \mu)$ is a rational local shtuka datum and $\CG$ is a smooth group scheme model over $\BZ_p$ with connected special fiber, we set, for $K=\CG(\BZ_p)$, 
\[ {\rm Sht}_K(G, b, \mu)=(\CM^{\rm int}_{\CG, b, \mu})_\eta .
\]
This is a locally spatial diamond over $\Spd(\br E)$, cf. \cite[\S 23.3]{Schber}.  Note that here,  in contrast to section \ref{ss:genfib}, $ {\rm Sht}_K(G, b, \mu)$ denotes a diamond which  is not representable by a rigid-analytic space, unless $\mu$ is minuscule. We also set 
\[
{\rm Sht}_\infty(G, b, \mu)=\varprojlim_K {\rm Sht}_K(G, b, \mu),
\]
as in \cite[\S 23.3]{Schber}.

\begin{proposition}\label{adIsoGen}
Let  $f\colon (G, b, \mu)\to (G', b', \mu')$ be an ad-isomorphism of rational local shtuka data. 
The natural morphism 
\[
{\rm Sht}_\infty(G, b, \mu)\to {\rm Sht}_\infty(G', b', \mu')\times_{\Spd(\br E')}\Spd(\br E)
\]
 induces an isomorphism of $v$-sheaves over $\Spd(\br E)$, 
\[
{\rm Sht}_\infty(G, b, \mu)\buildrel{\underline{G(\BQ_p)}}\over{\times}\underline{G'(\BQ_p)}\xrightarrow{\sim} {\rm Sht}_\infty(G', b', \mu')\times_{\Spd(\br E')}\Spd(\br E) .
\]

\end{proposition}

\begin{proof}
In the case of  minuscule $\mu$ this is  given by \cite[Prop. 3.1.2]{PRglsv}. The argument extends provided we explain that the corresponding admissible loci in the ``Schubert varieties" ${\rm Gr}_{G,\Spd(\br E),\leq \mu}$, resp. 
${\rm Gr}_{G', \Spd(\br E'),\leq\mu'}\times_{\Spd(\br E')}\Spd(\br E)$ agree. Recall that the latter objects are spatial diamonds \cite[Thm. 19.2.4]{Schber} which are proper over $\Spd(\br E)$, cf. \cite[Prop. 19.2.3]{Schber}. The admissible loci are open subobjects, cf. \cite[\S 23.3]{Schber}. 

Since $G\to G'$ is an ad-isomorphism, the morphism  
\[
{\rm Gr}_{G,\Spd(\br E)}\to{\rm Gr}_{G',{\Spd(\br E')}}\times_{\Spd(\br E')}\Spd(\br E)
\]
 induces an isomorphism between corresponding connected components.  In particular,  the morphism
\[{\rm Gr}_{G,\Spd(\br E),\leq \mu}\isoarrow{\rm Gr}_{G', \Spd(\br E'),\leq\mu'}\times_{\Spd(\br E')}\Spd(\br E)
\] is an isomorphism. Therefore, by \cite[Prop. 11.15, or Prop. 12.9]{Sch-Diam}, it remains to show that the induced map on admissible loci  is bijective on points with values in $\Spa(C, C^+)$, for any algebraically closed non-archimedean field $C$ and open bounded valuation ring $C^+$. 
Consider the following commutative diagram of (pro-systems of) locally spatial diamonds  over $\breve E$, in which the vertical arrows are the  period maps (\cite[\S 23.3]{Schber}), 
 \begin{displaymath}
   \xymatrix{
          {{\rm Sht}}_\infty(G, b, \mu) \ar[r]^{} \ar[d] & {{\rm Sht}}_\infty(G', b', \mu')\times_{\Spa(\breve{E'})}\Spa(\breve E)  \ar[d]\\
        {\rm Gr}_{G,\Spd(\br E), \leq\mu}\ar[r]^-{\sim} & {\rm Gr}_{G',\Spd(\br E'), \leq\mu'}\times_{\Spa({E'})}\Spa(\breve E).
        }
    \end{displaymath}
   The images of the vertical maps are the \emph{admissible sets}.
  We claim that under the lower horizontal map  the admissible sets correspond to each other. We have ${\rm Sht}_K(G, b, \mu)(C,  C^+)={\rm Sht}_K(G, b, \mu)(C, O_C)$ and ${\rm Sht}_{K'}(G', b', \mu')(C, C^+)={\rm Sht}_{K'}(G', b', \mu')(C, O_C)$, by their definitions in terms of shtuka. Hence, it is enough to check that the $\Spa(C, O_C)$-points of the admissible sets coincide.  Now a point $x$ of ${\rm Gr}_{G,\Spd(\br E), \leq\mu}$ with values in $\Spa(C, O_C)$ lies in the admissible locus if and only if the corresponding modification $\CE^b_{x}$ at $\infty$ of the $G$-bundle $\CE^b$ over the FF curve  is trivial, cf. \cite[Thm. 22.6.2]{Schber}. This in turn is equivalent to  $\CE^b_{x}$ being a semi-stable $G$-bundle on the FF curve (use $[b]\in B(G, \mu^{-1})$, cf. \cite[proof of Prop. 3.1.1]{PRglsv}).
  The image of $x$ lies in the admissible set in ${\rm Gr}_{G',\Spd(\br E'), \leq\mu'}$ if and only if the  $G'$-bundle $\CE^{b'}_{f_*(x)}=f_*(\CE^{b}_{ x})$ is a semi-stable $G'$-bundle on the FF curve. But since $f$ induces a bijection between parabolic subgroups of $G$, resp. $G'$, the semi-stability of $f_*(\CE^b_{ x})$ is equivalent to the semi-stability of $\CE^b_x$ (apply this bijection to the Harder-Narasimhan parabolics and the fact that semi-stability is equivalent to the fact that the HN-parabolic is the whole group).  Our claim follows. Now consider the following diagram, where in the lower line appears the isomorphism between admissible sets,
    \begin{displaymath}
    \xymatrix{
         {\rm Sht}_\infty(G, b, \mu)\buildrel{\underline{G(\BQ_p)}}\over{\times}\underline{G'(\BQ_p)} \ar[r]^{} \ar[d] & {{\rm Sht}}_\infty(G', b', \mu')\times_{\Spd(\breve{E'})}\Spd(\breve E)  \ar[d]\\
        ({\rm Gr}_{G,\Spd(\br E), \leq\mu})^{\rm adm}\ar[r]^-{\sim}& ({\rm Gr}_{G',\Spd(\br E'), \leq\mu'})^{\rm adm}\times_{\Spd({E'})}\Spd(\breve E).
        }
    \end{displaymath}
Now the fibers of the left vertical arrow are identified with $G(\BQ_p)\times^{G(\BQ_p)}G'(\BQ_p)=G'(\BQ_p)$, and hence map bijectively to the fibers of the right vertical arrow.
\end{proof}
\begin{remark}
In \cite[proof of Prop. 3.1.1]{PRglsv}, there is a different argument for the proof of the isomorphism between the admissible loci of the lower horizontal map  in the  last  diagram.  
\end{remark}

\subsection{Ad-isomorphisms and formal completions} \label{ss:adformcom}

\begin{proposition}\label{adIso}
Let $f\colon (\CG, b, \mu)\to (\CG', b', \mu')$ be an ad-isomorphism of integral local shtuka data.  For each $x\in \CM^{\rm int}_{\CG, b, \mu}(\Spd(k))$, $f$ induces an isomorphism
\[
\wh f: \CM^{\rm int}_{\CG, b, \mu/x}\xrightarrow{\sim} \CM^{\rm int}_{\CG', b', \mu'/x'}\times_{\Spd(O')}\Spd(O),
\]
where  $x'=f(x)$.
\end{proposition}

\begin{proof} 
As in \cite[proof of Prop. 3.4.1]{PRglsv} we have 
\[
 \CM^{\rm int}_{\CG, b, \mu/x}\simeq  \CM^{\rm int}_{\CG, b_x, \mu/x_0}
\]
and it is enough to show the isomorphism   for the base point $x=x_0$.
By  Lemma \ref{adisolem},  it is enough to show this when $f$ is of type (i) or of type (ii). 
\smallskip

\textit{Assume first that $f$ is of type (ii).} Then   we have an exact sequence
   \[
   1\to Z\to G\to G'\to 1, 
   \]
   with $Z=T$ a quasi-trivial central torus. In this case, we have an exact sequence, where $\CZ$ is the flat closure of $Z$, 
  \[
  1\to \CZ\to\CG\to\CG'.
  \]
Note that since $Z=T$ is quasi-trivial, i.e. $Z\simeq\prod_i\Res_{K_i/\BQ_p}\BG_m$, we have that $\Omega_Z=\pi_1(Z)_I $ is torsion free and $Z$ is an $R$-smooth torus in the sense of \cite[Def. 2.4.3]{KZhou}. By \cite[Prop. 2.4.12]{KZhou}, cf. \cite[Lem. 6.7]{PRTwisted}, we 
conclude that $\CZ$ is smooth and connected and, in fact,  is the connected Neron model of the torus $T$ and that we have a short fppf exact sequence of 
 smooth connected group schemes between the parahoric neutral components, 
\[
1\to \CZ\to \CG^o\to\CG'^o\to 1 .
\]
 Note that in this situation $\CZ\simeq\prod_i\Res_{\CO_{K_i}/\BZ_p}\BG_m$. The image of $\pi_0(\CG)\subset \Omega_G$ under $\Omega_G\to \Omega_{G'}$ is a finite subgroup $Y'$ of $\Omega_{G'}$. We can find a short fppf exact sequence of 
 (smooth) group schemes
\[
 1\to \CZ\to \CG\to \CG_1'\to 1 ,
 \]
 where $\CG_1'\subset \CG'$ is a quasi-parahoric group scheme with the same neutral component $\CG'^0$ and $\pi_0(\CG_1')=Y'$. By Proposition \ref{neutralIso}, there is an isomorphism
 $\CM^{\rm int}_{\CG_1', b'_x, \mu'/x'_0}\simeq \CM^{\rm int}_{\CG', b'_x, \mu'/x'_0}$. Hence we may replace $\CG'$ by $\CG'_1$, i.e., we may assume from the beginning that $\CG\to\CG'$ is surjective.

 Note that
 \[
 \wh \BW^+\CG'=\wh \BW^+\CG'^0.
 \]
We obtain a short exact sequence of group $v$-sheaves
  \begin{equation}\label{exactW+}
 1\to \wh \BW^+\CZ\to \wh \BW^+\CG\to \wh \BW^+\CG'\to 1.
 \end{equation} 
 The local models for $(\CG, \mu)$ and $(\CG', \mu')$ ``agree", i.e. 
 \[
 f: \BM^{v}_{\CG, \mu}\xrightarrow{\sim} \BM^{v}_{\CG', \mu'}\times_{\Spd O'}\Spd O,
  \]
  (\cite[21.5.1]{Schber}), so 
  we have 
  \[
 \wh f_0:  \BM^{v}_{\CG, \mu /y_0}\xrightarrow{\sim}  \BM^{v}_{\CG', \mu' /y_0}\times_{\Spd O'}\Spd O,
   \]
  for the completions of the local models at their common base point $y_0$.

We now use Theorem \ref{vLMD} for $\CG$ and $\CG'$: The morphism
\[
 \wh {L\CG}_{b,\mu}\to  \BM^{v}_{\CG, \mu /y_0}\xrightarrow{\sim}  \BM^{v}_{\CG', \mu' /y_0}\times_{\Spd O'}\Spd O
\]
is a $\wh{\BW^+}\CG\times \Spd(O)$-torsor and 
\[
 \wh {L\CG}'_{b',\mu'}\times_{\Spd(O')}\Spd(O)\to \BM^{v}_{\CG, \mu /y_0}\xrightarrow{\sim}  \BM^{v}_{\CG', \mu' /y_0}\times_{\Spd O'}\Spd O,
\]
a $\wh{\BW^+}\CG'\times_{\Spd(O')}\Spd(O)$-torsor. 
Hence, we have  an isomorphism of  $\wh{\BW^+}\CG'$-torsors, 
  \begin{equation}\label{LMWGtorsor}
 \wh {L\CG}'_{b',\mu'}\times_{\Spd(O')}\Spd(O)\simeq \wh {L\CG}_{b,\mu}\times^{\wh{\BW^+}\CG} \wh{\BW^+}\CG' .
 \end{equation}
It follows that, $v$-locally,
 \[
 \wh {L\CG}_{b,\mu}= g_0\cdot \wh \BW^+\CG , \quad\ 
  \wh {L\CG}'_{b',\mu'}= g'_0\cdot \wh \BW^+\CG' ,
 \]
 (after we choose a section $g_0$ of the first torsor).  
 By Theorem \ref{vLMD} also, we can recover $\CM^{\rm int}_{\CG, b, \mu/x_0}$
 as the quotient  of $\wh{L\CG}_{ b,\mu} $ by the $\phi$-conjugation action of $\wh \BW^+\CG\times\Spd(O) $. More precisely,
 \[
 \wh{L\CG}_{b,\mu} \to  \CM^{\rm int}_{\CG, b, \mu /x_0}
 \]
 is  a $\wh \BW^+\CG\times\Spd(O)$-torsor for the $\phi$-conjugation action. The corresponding statement  for $(\CG',b',\mu')$ also holds.
 The exact sequence (\ref{exactW+}) and the above now implies that the map $\wh f_0$  is $v$-locally surjective. 
 
 To show that $\wh f_0$ is  injective, it is enough 
 to check injectivity for $\phi$-conjugacy classes as follows (for simplicity, we omit the base change to $O$ from the notation). 
 Suppose that $g_0\cdot g_1$ and $g_0\cdot g_2\in \wh{L\CG}_{b, \mu}(R, R^+)$ map, after replacing $(R, R^+)$ by a $v$-cover, to the same $\phi$-conjugacy class (by $\wh \BW^+\CG'(R, R^+)$) in $\wh{L\CG}'_{ b',\mu'}(R, R^+)$.
 Then
 \begin{equation}\label{eqPhi1}
 g_0\cdot g_1= \phi(g'^+)^{-1}\cdot (g_0\cdot g_2)\cdot g'^+ \in \wh{L\CG}'_{ b',\mu'}(R, R^+)\subset \CG'(W(R^+)[1/\xi]),
 \end{equation}
 for some $g'^+\in \wh \BW^+\CG'(R, R^+)$.
 Using (\ref{exactW+}), after further replacing $(R, R^+)$ by a $v$-cover, we can lift  $g'^+$ to $g^+\in \wh \BW^+\CG(R, R^+)$.
 The identity (\ref{eqPhi1}) gives
 \begin{equation}\label{idg01}
 z=[g_1\cdot (g^+)^{-1} \cdot g_2^{-1} ]\cdot g_0^{-1}\cdot \phi(g^+) \cdot g_0
  \end{equation}
 in $\CG(W(R^+)[1/\xi])$ with $z \in  \CZ(W(R^+)[1/\xi])$ (central in $\CG(W(R^+)[1/\xi])$).
 In this, the term $k^+:=g_1\cdot (g^+)^{-1}\cdot g_2^{-1}$ in the bracket belongs to $\wh \BW^+\CG(R, R^+)$. Write the above
\[
z =k^+\cdot (g_0^{-1}\cdot \phi(g^+) \cdot g_0).
\]
We claim that this implies that $z=z^+\in \wh \BW^+\CZ(R, R^+)$. Indeed, for all $m\geq 1$, we obtain
 \[
 z^m=  (k^+)^m\cdot (g_0^{-1}\cdot \phi(g^+)^m \cdot g_0) .
 \]
 Hence, there is $N\geq 1$ depending on $g_0$ only, such that $z^m\in \CZ(W(R^+)[1/\xi])$ has matrix entries (after applying some faithful representation of $\CG$)
 with denominators which are powers of $\xi$ bounded by $N$, for all $m$. This implies  that $z\in \CZ(W(R^+))$. Indeed, write $\CZ=
\prod_{i}{\rm Res}_{O_{K_i}/\BZ_p}\BG_m$. Then,  if $z^m\in (O_{K_i}\otimes_{\BZ_p} W(R^+)[1/\xi])^\times$  has bounded denominators for all $m\geq 1$,
then  $z\in (O_{K_i}\otimes_{\BZ_p} W(R^+))^\times$. 
 We can see using (\ref{idg01}) that we also have $z\equiv 1\, {\rm mod}\, [\varpi]$, for some pseudo-uniformizer $\varpi$ of $R^+$. Hence, we conclude $z=z^+\in \wh \BW^+\CZ(R, R^+)$.
 By applying Lemma  \ref{recursive} for $\CZ$ to $z^+$, we can write $z^+=z_1^+\phi(z_1^+)^{-1}$. Set $h^+=g^+\cdot z^+_1$.
 Using that $Z$ is central, we can rewrite (\ref{idg01}) as 
 \[
   g_0g_1 =  \phi(h^+)^{-1}\cdot g_0 g_2\cdot h^+. 
 \]
 This gives that $g_0g_1$ and $g_0g_2$ are $\phi$-conjugate by $h^+\in \wh{\BW}^+\CG(R, R^+)$.
 This shows injectivity and concludes the proof.
\smallskip

\textit{Assume now that $f$ is of type (i). }
Using Proposition \ref{neutralIso}, we can replace $\CG$, $\CG'$ by their parahoric neutral components $\CG^o$, $\CG'^o$. 
For these we have $\CG^o(\br\BZ_p)\subset G(\br\BQ_p)\cap f^{-1} (\CG'^o(\br\BZ_p))$ with the intersection 
the $\br\BZ_p$-points of a quasi-parahoric $\wt\CG$ with $\wt\CG^o=\CG$. The morphism $f$
factors
\[
 \CM^{\rm int}_{\CG, b, \mu}\to \CM^{\rm int}_{\wt\CG, b, \mu}\xrightarrow{\ \wt f\ } \CM^{\rm int}_{\CG', b', \mu'}\times_{\Spd(O')}\Spd(O).
\] 
By Proposition \ref{typeI} applied to $\wt\CG\to \CG'$, the morphism $\wt f$ is a closed immersion. 
Using \cite[Prop. 4.20]{Gl}
we see that the base change of $\wt f$ along the completion $ \CM^{\rm int}_{\CG', b', \mu' /x_0'}\to \CM^{\rm int}_{\CG', b', \mu'}$ is the completion of $\wt f$ at the base points
 \[
\wh{{\wt f}}_0:  \CM^{\rm int}_{\wt\CG, b, \mu /\wt x_0} \to \CM^{\rm int}_{\CG', b', \mu' /x'_0}\times_{\Spd(O')}\Spd(O).
 \]
 Hence, this is also a closed immersion. 
By Proposition \ref{neutralIso} applied to the quasi-parahoric $\wt\CG$ we see that the source of this morphism can be identified with
 $ \wh{\CM^{\rm int}_{\CG, b, \mu}}_{/x_0} $ and the morphism with $\hat f_0$.
 Hence, $\hat f_0$  is also a closed immersion.

\begin{lemma}\label{lemmaTorKott}
Suppose $\wt x_1$, $\wt x_2\in ( \CM^{\rm int}_{\CG^o, b, \mu /x_0})_\eta(C, O_C)$ have the same image $\pi(\wt x_1)=\pi(\wt x_2)$ under the period map. Then there is $g\in G(\BQ_p)^0$ such that $\wt x_2=g\cdot \wt x_1$.
\end{lemma}

\begin{proof}
Since the period map $\pi: (\CM^{\rm int}_{\CG^o, b, \mu})_\eta\longrightarrow {\rm Gr}_{G,\Spd(\br E), \leq \mu}^{\rm adm}$ is \'etale with geometric fibers $G(\BQ_p)/K^o$, there is $g\in G(\BQ_p)$ with $\wt x_2=g\cdot \wt x_1$. 
We have ${\rm sp}(\wt x_1)={\rm sp}(\wt x_2)=x$ and it remains to show that $\kappa_G(g)=1$.
We first observe that the statement is true when $G=T$ is a torus. Indeed, in this case  
the period map is
\[
(\CM^{\rm int}_{\CT^o, b, \mu})_\eta\longrightarrow \Spd(\br E)
\]
with geometric fibers given by $\kappa_T: T(\BQ_p)/\CT^o(\BZ_p)\xrightarrow{\simeq} \Omega_T$ while 
$ (\wh{\CM^{\rm int}_{\CT^o, b, \mu}}_{/x_0})_\eta\simeq \Spd(\br E)$, see \cite[25.2]{Schber}.
Then, the result follows when $G_{\rm der}$ is simply connected by applying functoriality to $G\to G_{\rm ab}=D$ and observing that, in this case, $\Omega_G\simeq \Omega_{D}$. Finally, we reduce the general case to the case when the derived group is simply-connected by employing a $z$-extension
\[
1\to Z\to \wt G\to G\to 1
\]
and lifting $(\CG^o, b, \mu)$ to a corresponding triple $(\wt\CG^o, \wt b, \wt \mu)$. Here, $\wt G_{\rm der}$ is simply connected and $Z$ is a quasi-trivial torus. By case (ii) applied to $\wt G\to G$ we obtain
\[
( \CM^{\rm int}_{\wt\CG^o,\wt b,\wt \mu /x_0})_\eta\xrightarrow{\sim} ( \CM^{\rm int}_{\CG^o, b, \mu /x_0})_\eta,
 \]
 and the admissible sets for $\wt G$ and $G$ coincide by the proof of Proposition \ref{adIsoGen}.
Since we are assuming that the result holds for $\wt G$, we  deduce it for $G$.
\end{proof}

\begin{lemma}\label{lemmaBT2}
The map $f$ induces a surjection
$
G(\BQ_p)^0/K^o\to G'(\BQ_p)^0/K'^o.$
\end{lemma}

\begin{proof} For any reductive group $G$ over $\BQ_p$ we have
$$
G(\BQ_p)^0=K\cdot \varphi(G_{\rm sc}(\BQ_p)),
$$
where $K$ is an arbitrary parahoric subgroup of $G(\BQ_p)$ and $\phi: G_{\rm sc}\to G_{\rm der}\to G$ is the natural map. Note that both factors on the RHS are contained in the LHS, and the second factor is a normal subgroup, so the inclusion of the RHS into the LHS is clear. For the other inclusion, we choose an apartment $A^\natural$ of the building over $\BQ_p$ as in \cite{BT2} such that $K$ fixes a facet in $A^\natural$. Associated to  $A^\natural$ is the smooth connected group scheme $\sZ^o$ over $\BZ_p$. Now Bruhat-Tits \cite{BT2} prove the following two facts:
\begin{itemize}
\item $G(\BQ_p)^0=  \sZ^o(\BZ_p)\cdot\varphi(G_{\rm sc}(\BQ_p))$, cf. last line of \cite[(5.2.11)]{BT2}.

\item $\sZ^o(\BZ_p)\subset K$, cf. first display in \cite[(5.2.4)]{BT2}.
\end{itemize}
Hence, we also obtain the other inclusion. By applying the above to $G'$ and $K'^o$ and $G$, noting that
$\varphi': G_{\rm sc}=G'_{\rm sc}\to G'$ factors through $f: G\to G'$,  we obtain
\[
G'(\BQ_p)^0=f(G(\BQ_p)^0)\cdot K'^o,
\]
which gives the result.
\end{proof}

The above two lemmas, together with Proposition \ref{adIsoGen}, imply that the map
\[
(\hat f_0)_\eta: ( \CM^{\rm int}_{\CG, b, \mu /x_0})_\eta\to ( \CM^{\rm int}_{\CG', b', \mu' /x'_0})_\eta
\]
induced by $f$, gives a surjection on $(C, O_C)$-points. Since it is also a closed immersion between partially proper $v$-sheaves, it then follows that it is an isomorphism.

 Both the source $X:= \CM^{\rm int}_{\CG, b, \mu /x_0}$ and the target 
$X':= \CM^{\rm int}_{\CG', b', \mu' /x_0'}$ of the closed immersion
\[
\wh f_0: X\to X'
\]
 are topologically flat by \cite[Prop. 3.4.9]{PRglsv}, as extended in  \cite[Rem. 3.4.12]{PRglsv} also for the non-minuscule case (this extension uses the results of \cite{AGLR}). By the definition
of topological flatness, this implies that $|X_\eta|$, resp. $|X'_{\eta}|$, is dense in $|X|$, resp. $|X'|$. Since $(\wh f_0)_\eta: X_\eta\simeq X'_\eta$ is an isomorphism, we obtain $|X|\simeq |X'|$.   The result that $\wh f_0: X\to X'$ is an isomorphism now follows from \cite[Lem. 17.4.1]{Schber}.
\end{proof}

\subsection{Ad-isomorphisms and ADLV}\label{ss:adADLV}
Let $f\colon (\CG, b, \mu)\to (\CG', b', \mu')$ be an ad-isomorphism of integral local shtuka data.  
We have a commutative diagram with surjective vertical arrows,
\begin{equation}\label{Fldiag}
\begin{aligned}
 \xymatrix{
        X_{\CG^o} \ar[r]^{\ \ } \ar[d]_{\kappa_G} &  X_{\CG'^o}   \ar[d]^{\kappa_{G'}}\\
       \ \  \Omega_{G}\quad\ar[r]   &\ \   \Omega_{G'}
        }
        \end{aligned}
\end{equation}
where the lower row are discrete schemes (which we treat like sets).
\begin{lemma}\label{compiso}
 The upper arrow induces an isomorphism between corresponding connected components. The diagram is cartesian.
\end{lemma}
\begin{proof}
An ad-isomorphism induces an isomorphism $f_{\rm sc}: G_{\rm sc}\isoarrow G'_{\rm sc}$. This easily implies the first  assertion. The second assertion follows since $\pi_0(X_{\CG^o})=\Omega_G$ and  $\pi_0(X_{\CG'^o})=\Omega_{G'}$. 
\end{proof}
Let us consider the effect of an ad-isomorphism on Iwahori Weyl groups. An ad-isomorphism induces maps $N(\br\BQ_p)\to N'(\br\BQ_p)$,  and $T(\br\BQ_p)^0\to T'(\br\BQ_p)^0$,  and $\wt W\to \wt W'$ and an identification of affine Weyl groups $W_a=W'_a$, compatible with the semi-direct product decompositions,
$$
\wt W=W_a\rtimes\Omega_{G}\to W_a\rtimes\Omega_{G'}=\wt W' .
$$
Denoting by $K^o$, resp. $K'^o$ the parahoric subgroups of $G(\BQ_p)$, resp. $G'(\BQ_p)$, we have subgroups $W^{K^o}\subset W_a$ and $W^{K'^o}\subset W_a$. These subgroups of $\wt W$, resp. $\wt W'$, can be identified. 

The admissible sets are contained in a single $W_a$-coset and hence the map $\wt W\to \wt W'$ induces a bijection
\begin{equation*}
\Adm(\mu)\isoarrow\Adm(\mu') .
\end{equation*}
Similarly denoting by $\Adm_{K^o}(\mu)$, resp. $\Adm_{K'^o}(\mu')$, the image of $\Adm(\mu)$ in $W^{K^o}\backslash\wt W/W^{K^o}$, resp. of $\Adm(\mu')$ in $W^{K'^o}\backslash\wt W'/W^{K'^o}$, we have a bijection
\begin{equation}\label{adid}
\Adm_{K^o}(\mu)\isoarrow\Adm_{K'^o}(\mu') .
\end{equation}
The commutative diagram \eqref{Fldiag} induces a commutative diagram, cf. \eqref{coset}, 
\begin{equation}
\begin{aligned}
 \xymatrix{
        X_{\CG^o}(b, \mu) \ar[r]^{\ \ } \ar[d]_{\kappa_G} &  X_{\CG'^o}(b', \mu')   \ar[d]^{\kappa_{G'}}\\
       c_{b,\mu}+ \Omega_{G}^\phi\quad\ar[r]^{\ \ }  &  c_{b',\mu'}+ \Omega_{G'}^\phi .
        }
        \end{aligned}
\end{equation}

\begin{lemma}\label{forparas}
The upper arrow induces an isomorphism between corresponding components. The diagram is cartesian, with surjective vertical arrows.
\end{lemma}

\begin{proof}
Here recall that by a component of $X_{\CG}(b, \mu)$, we mean an intersection of $X_{\CG}(b, \mu)$ with a connected component of $X_{\CG}$. The surjectivity of the vertical arrows follows from \eqref{coset}. 

Let $\tau\in  c_{b,\mu}+\Omega_{G}^\phi$, with image $\tau'=f(\tau)\in  c_{b',\mu'}+\Omega_{G'}^\phi$. By Lemma \ref{compiso}, for every $g'\br K'^o\in X_{\CG'^o}(b', \mu')\cap X^{\tau'}_{\CG'^o}$, there exists a unique point $g\br K^o\in X^\tau_{\CG^o}$ lying over it. We have to compare the double cosets of $h:=g^{-1}b\phi(g)$ in $\br K^o\backslash G(\br\BQ_p)/\br K^o=W^{K^o}\backslash\wt W/W^{K^o}$ and of $f(h)=g'^{-1}b'\phi(g')$ in $\br K'^o\backslash G(\br\BQ_p)/\br K'^o=W^{K'^o}\backslash\wt W'/W^{K'^o}$. But by \eqref{adid} it follows that the class of $h$ lies in $\Adm_{K^o}(\mu)$ if and only if the class of $f(h)$ lies in $\Adm_{K'^o}(\mu')$. Hence $g\br K^o\in X_{\CG^o}(b, \mu)^\tau$, as desired.
\end{proof}

Now we pass from the parahorics $\CG^o$ and $\CG'^o$ to the quasi-parahorics $\CG$ and $\CG'$.  The homomorphism $f: G\to G'$ induces homomorphisms
$$
\Omega_G\to\Omega_{G'}, \quad \pi_0(\CG)\to\pi_0(\CG'),\quad  C_\CG\to C_{\CG'},  \quad \Pi_\CG\to\Pi_{\CG'} . 
$$
Here recall that $C_\CG=\Omega_G/\pi_0(\CG)$  and $C_{\CG'}=\Omega_{G'}/\pi_0(\CG')$, and $\Pi_\CG=\ker (\pi_0(\CG)_\phi\to \Omega_{G, \phi})$ and $\Pi_{\CG'}=\ker (\pi_0(\CG')_\phi\to \Omega_{G', \phi})$. We obtain from Corollary \ref{imageka} a diagram with surjective vertical arrows, 
\begin{equation}\label{cartdi}
\begin{aligned}
 \xymatrix{
       X_\CG(b, \mu) \ar[r]^{\ \ } \ar[d]_{\kappa_G} &  X_{\CG'}(b', \mu')    \ar[d]^{\kappa_{G'}}\\
       \bar c_{b,\mu}+ C_{\CG}^\phi\quad\ar[r]^{\ \ }  &  \bar c_{b',\mu'}+ C_{\CG'}^\phi .
        }
        \end{aligned}
\end{equation}

\begin{proposition}\label{cartesianADLVq}
The upper arrow in \eqref{cartdi} induces an isomorphism between corresponding components. The diagram is cartesian, with surjective vertical arrows.
\end{proposition}
\begin{proof}
By Theorem \ref{quasiThm} and the proof of Corollary \ref{imageka}, the diagram can be rewritten as  
\begin{equation*}
\begin{aligned}
 \xymatrix{
        \bigsqcup_{\bar\beta\in\Pi_\CG}X_{\CG_{\beta}^o}(b, \mu)/\pi_0(\CG_\beta)^\phi \ar[r]^{\ \ } \ar[d]_{\kappa_G} &  \bigsqcup_{\bar\beta'\in\Pi_{\CG'}}X_{\CG'^o_{\beta'}}(b', \mu')/\pi_0(\CG_\beta')^\phi     \ar[d]^{\kappa_{G'}}\\
        \bigsqcup_{\bar\beta\in\Pi_\CG}( c_{b,\mu}+ \Omega_{G}^\phi)/\pi_0(\CG_\beta)^\phi\quad\ar[r]^{\ \ }  &   \bigsqcup_{\bar\beta'\in\Pi_\CG'}( c_{b',\mu'}+ \Omega_{G'}^\phi)/\pi_0(\CG'_{\beta'})^\phi .
        }
        \end{aligned}
\end{equation*}

The maps in this diagram respect the disjoint sum decompositions. Hence the assertion follows from Lemma \ref{forparas}. 
\end{proof}

\subsection{Proof of Theorem \ref{thm451}}\label{ss:qcqs}

 The plan of the proof is as follows.
We will show in Proposition \ref{Propqcqs} below that
\[
f^\tau:  \CM^{\rm int,\tau}_{\CG, b, \mu}\xrightarrow{\ \ } \CM^{\rm int}_{\CG', b', \mu'}\times_{\Spd(O')}\Spd(O)
 \]
 is qcqs. Since $\CM^{\rm int, \tau'}_{\CG', b', \mu'}\to \CM^{\rm int}_{\CG', b', \mu'}$ is an open and closed immersion, 
\[
f^{\tau, \tau'}:  \CM^{\rm int,\tau}_{\CG, b, \mu}\xrightarrow{\ \ } \CM^{\rm int, \tau'}_{\CG', b', \mu'}\times_{\Spd(O')}\Spd(O)
 \]
 is also qcqs (cf. the proof of Lemma \ref{lemmaQS} below). Assuming this, by \cite[Cor. 17.4.10]{Schber} and the partial properness of 
 $\CM^{\rm int,\tau}_{\CG, b, \mu}$ and $\CM^{\rm int, \tau'}_{\CG', b', \mu'}$
 over $\Spd(O)$, resp. $\Spd(O')$, it is enough to show the following statement.

 \begin{proposition}\label{BijectionCpts}
  Under the above assumptions, $f^{\tau,\tau'}$ induces a bijection on $\Spa(C, O_C)$-points. 
  \end{proposition}
 
 \begin{proof}
 By Proposition \ref{MX},  if $\kappa$ is a discrete, algebraically closed field of characteristic $p$, we have bijections
\[
\CM^{\rm int,\tau}_{\CG, b, \mu}(\Spd(\kappa))=X^{\tau}_{\CG}(b, \mu^{-1})(\kappa),\quad \CM^{\rm int,\tau'}_{\CG', b', \mu'}(\Spd(\kappa))=X^{\tau'}_{\CG'}(b', \mu'^{-1})(\kappa)
\]
and $f^{\tau, \tau'}(\Spd(\kappa))$  is identified with
the corresponding map
 \begin{equation}\label{BijectionADLV}
 X^{\tau}_{\CG}(b, \mu^{-1})(\kappa) \xrightarrow{} X^{\tau'}_{\CG'}(b', \mu'^{-1})(\kappa) .
 \end{equation} 
But this map  is a bijection by Proposition \ref{cartesianADLVq}.
Hence, $f^{\tau, \tau'}$ induces a bijection on $\Spd(\kappa)$-points. We still have to treat   $\Spa(C, O_C)$-points.

Set $O(\kappa)=O\otimes_{W(k)}W(\kappa)$, and similarly for $O'(\kappa)$.
Note that the reduced locus of the base change 
\[
(\CM^{\rm int,\tau}_{\CG, b, \mu})_{O(\kappa)}:=\CM^{\rm int,\tau}_{\CG, b, \mu}\times_{\Spd(O)}\Spd( O(\kappa))
\]
is $ X^{\tau}_{\CG}(b, \mu^{-1})_\kappa:=X^{\tau}_{\CG}(b, \mu^{-1})\times_{\Spec(k)}\Spec(\kappa)$, and similarly for $(\CM^{\rm int,\tau'}_{\CG', b', \mu'})_{ O'(\kappa)}$. 

Now note that any $\Spa(C, O_C)$-point $\tilde x: \Spa(C,O_C)\to  \CM^{\rm int,\tau}_{\CG, b, \mu}$  factors through the formal completion $ ( \CM^{\rm int,\tau}_{\CG, b, \mu})_{ O(\kappa) /x}$ of $( \CM^{\rm int,\tau}_{\CG, b, \mu})_{ O(\kappa)}$ at $x:={\rm sp}(\tilde x)$, where $\kappa$ is the residue field $k(C)=O_C/\fkm_C$. 
 (Observe that $x$ now gives a closed point of the reduced locus of $(\CM^{\rm int,\tau}_{\CG, b, \mu})_{O(\kappa)}$
 which is $ X^{\tau}_{\CG}(b, \mu^{-1})_\kappa$.)
 Similarly, the corresponding fact is true for $\Spa(C, O_C)$-points of $\CM^{\rm int,\tau'}_{\CG', b', \mu'}$. 
If $x$, resp. $x'$, is a $\Spd(\kappa)$-point of $\CM^{\rm int,\tau}_{\CG, b, \mu}$, resp. $\CM^{\rm int,\tau'}_{\CG', b', \mu'}$, we have  
 \[
 (\CM^{\rm int,\tau}_{\CG, b, \mu})_{O(\kappa) /x}  \simeq  (\CM^{\rm int }_{\CG, b, \mu})_{O(\kappa) /x}, \quad {\rm resp.}\ \ 
 (\CM^{\rm int,\tau'}_{\CG', b', \mu'})_{O'(\kappa) /x'}  \simeq  (\CM^{\rm int }_{\CG', b', \mu'})_{O'(\kappa) /x'}.
 \]
 The result now follows from a simple extension of Proposition \ref{adIso} to the base changes by $W(\kappa)$. 
  Indeed,  this gives
 \[
f^{\tau, \tau'}:   (\CM^{\rm int,\tau}_{\CG, b, \mu})_{O(\kappa) /x}\xrightarrow{\ \sim\ }   (\CM^{\rm int,\tau'}_{\CG', b', \mu'})_{O'(\kappa) /f(x)}\times_{\Spd(O')}\Spd(O).
 \]
 This, combined with the bijection (\ref{BijectionADLV}) above for $\kappa=k(C)$, implies the result.
 \end{proof}
 
It remains to show the following statement.

\begin{proposition}\label{Propqcqs} The map of $v$-sheaves
\[
f^\tau:  \CM^{\rm int,\tau}_{\CG, b, \mu}\xrightarrow{\ \ } \CM^{\rm int}_{\CG', b', \mu'}\times_{\Spd(O')}\Spd(O)
 \]
 is qcqs.
\end{proposition}

\newcommand{\CMint}{\CM^{\rm int}}

 \begin{proof} 
  In the following, we occasionally just write $\CM^{\rm int}_\CG$, $\CM^{\rm int}_{\CG'}$, etc., for notational simplicity.
First we show that $f: \CM^{\rm int}_\CG\to \CM^{\rm int}_{\CG'}\times_{\Spd(O')}\Spd(O)$ is quasi-separated (qs). This quickly implies that the same is true for $f^\tau$.
 For this we use that 
$f: \CM^{\rm int}_\CG  \to \Spd(O)$ is qs (\cite[Prop. 2.25]{Gl21}) and the following lemma.

\begin{lemma}\label{lemmaQS} Consider morphisms of small $v$-sheaves
  \[
  X\xrightarrow{f}Y\xrightarrow{g} Z.
  \]
  
  1) If $g\circ f$ is qc and $g$ is qs, then $f$ is qc.
  
  2) If $g\circ f$ is qs, then so is $f$. 
\end{lemma}
\begin{proof} 
Note that, by definition, $f: X\to Y$ is called qs when the diagonal $X\to X\times_Y X$ is qc, cf. \cite[\S8]{Sch-Diam}.
Also note that compositions and base changes of qc, resp. qs, maps are also qc, resp. qs. 

{\sl Part 1): } Write $f$ as the composition $X\to X\times_ZY\to Y$, where the first map is ${\rm id}\times f$ and the second the projection. The projection is qc as the base change of the qc map $X\to Z$,  while the first map
  is a section of the projection $X\times_Z Y\to X$; this projection is qs as the base change of $Y\to Z$. It remains to observe
  that a section $s: S\to T$ of a qs map $T\to S$ is qc since it can be viewed as the base change of 
  the qc diagonal $T\to T\times_S T$ by $s\times {\rm id}: T=S\times_S T\to T\times_S T$. 

{\sl Part 2): } Since $g\circ f: X\to Z$ is qs, the map 
  $X\to X\times_Z X$ is qc.  We can write this as a composition
  \[
  X\to X\times_Y X\to X\times_Z X.
  \]
  Here $X\times_Y X\to X\times_Z X$ is an injection and hence it is qs, since then the diagonal map is 
  an isomorphism. We now apply (a) to this composition to deduce that $X\to X\times_Y X$ is qc, hence 
  $X\to Y$ is qs.
  \end{proof}

It remains to show that $f^\tau$ is qc. For this we write again $G\to G'$ as a composition of ad-isomorphisms which are either of type (i), i.e.  closed embeddings (with cokernel a torus), or type (ii), i.e. fppf surjections with kernel a central quasi-trivial torus $Z=T$. Using functoriality and the fact that compositions of qcqs morphisms are again qcqs, we see that it is enough to treat these two cases separately. In the  type (i) case the result follows  from Proposition \ref{typeI}. It remains to deal with cases of type (ii).
Then we have an exact
   \[
   1\to Z\to G\to G'\to 1
   \]
   with $Z=T$ a quasi-trivial torus. For a morphism $\CG'_1\to\CG'$ of quasi-parahoric group schemes  corresponding to a fixed parahoric group scheme for $G'$, the morphism  
 $ \CM^{\rm int }_{\CG_1'}\to  \CM^{\rm int }_{\CG'}$ is qc. Hence the argument in the proof of Proposition \ref{adIso} (for type (ii)) shows that we may assume that the sequence 
 \[
 1\to \CZ\to \CG\to \CG'\to 1
 \]
  is exact.
  
  To establish quasi-compactness we use a ``sequence of points" argument (compare to \cite[proof of Thm. 21.2.1]{Schber}). Let $Y=\Spa(R, R^+)$ be  affinoid perfectoid 
    over $k$ and let 
    \[
    Y\to \CM^{\rm int}_{\CG', b', \mu'}\times_{\Spd(O')}\Spd(O)
    \]
     be a morphism given by an untilt of $Y$ over $O$ and a $\CG'$-shtuka over $\CY_{[0,\infty)}(R, R^+)$
  with a framing. Consider also the small $v$-sheaf
    \[
   X= Y\times_{(\CM^{\rm int}_{\CG', b', \mu'}\times_{\Spd(O')}\Spd(O))}\CM^{\rm int, \tau}_{\CG, b, \mu}\to Y. 
    \]
  Take $I=|X|$ and, for each $i\in I$, choose a point of $X$  given by $x_i: \Spa(C_i, C_i^+)\to \CM^{\rm int, \tau}_{\CG, b, \mu}$ and
    $y_i: \Spa(C_i, C^+_i)\to Y=\Spa(R, R^+)$, with matching compositions to $\CM^{\rm int}_{\CG', b', \mu'}\times_{\Spd(O')}\Spd(O)$, which corresponds to $i\in I=|X|$. 
    Consider the product of points
\[
(D, D^+)=((\prod_{i\in I}C^+_i)[1/(\varpi_i)], \prod_{i\in I}C^+_i).
\]
Here the pseudo-uniformizers $\varpi_i\in C^+_i$ are given by a pseudouniformizer of $R^+$. The collection of $y_i$ extends to $y:\Spa(D, D^+)\to \Spa(R, R^+)$.
The compositions
\[
\Spa(C_i, C^+_i)\to \Spa(D, D^+)\to \Spa(R, R^+)\to \Spd(O)
\]
specify untilts of $(C_i, C^+_i)$ given by $\xi_i\in W(C^+_i)$.  So, we have a $\CG'$-shtuka $\CP'$ with framing over $\Spa(D, D^+)$ obtained from 
\[
\Spa(D, D^+)\to \Spa(R, R^+)\to \CM^{\rm int}_{\CG', b', \mu'}\times_{\Spd(O')}\Spd(O),
\]
and a collection of $\CG$-shtukas $\CP_i$ with framings over each $(C_i, C^+_i)$, and with legs at $\xi_i$, given by $x_i$. These are compatible via the pushout of torsors by $\CG\to \CG'$. We want to show that $(x_i)$ extend to 
    \[
   x: \Spa(D, D^+)\to \CM^{{\rm int}, \tau}_\CG.
    \]
    Since the resulting map
\[
\Spa(D, D^+)\to  X= Y\times_{(\CM^{\rm int}_{\CG', b', \mu'}\times_{\Spd(O')}\Spd(O))}\CM^{\rm int, \tau}_{\CG, b, \mu}
\]
is  a $v$-cover, this fact will imply the quasi-compactness of $X$ and hence of $f^\tau$.

In the above, using the partial properness of $\CMint_\CG$ and $\CMint_{\CG'}$, we can replace $C_i^+$ by $O_{C_i}$. Using Proposition \ref{Anext}, we see that the pairs $(\Phi_i, i_{r_i})$ of $\CG$-shtuka
with framings given by  
    $x_i$, are described by pairs $(\Phi_i, g_i)$, where  $\Phi_i\in\CG(W(O_{C_i})[1/\xi_i])$ and $g_i\in G(\CY_{[r_i,\infty]}(C_i, O_{C_i}))$. We denote by  $(\Phi_i', g_i')$ the images of these 
pairs under  the map induced by $\CG\to\CG'$.
We have
     \begin{equation}\label{frameEq}
     \Phi_i =g_i^{-1} \cdot b\cdot \phi(g_i) .
     \end{equation}
Set
\[
\Phi=(\Phi_i)\in \prod_i \CG(W(O_{C_i})[1/\xi_i])= \CG(\prod_i (W(O_{C_i})[1/\xi_i])).
\]
Since the $\xi_i$-denominators of $\Phi_i$ are uniformily bounded in terms of the coweight $\mu$,
we see that 
\[
\Phi\in   \CG((W(\prod_i O_{C_i}))[1/\xi_i])\subset \CG(\prod_i (W(O_{C_i})[1/\xi_i])).
\]
This says
\[
\Phi\in \CG(W(D^+)[1/\xi]), \quad\hbox{\rm with }\quad \xi=(\xi_i)_i.
\]
The element $\Phi$ defines a $\CG$-shtuka $\CP(\Phi)$ over $\Spa(D, D^+)$. We claim that the corresponding $\CG'$-shtuka $\CP(\Phi')$ given by $\Phi'$  is isomorphic to the $\CG'$-shtuka $\CP'$, given by $\Spa(D, D^+)\to \CMint_{\CG'}$ above. Indeed, let  $\CP'$ be given by $\Psi'\in \CG'(W(D^+)[1/\xi])$ (Note that all $\CG$-torsors, resp. $\CG'$-torsors, over $W(D^+)$ are trivial by Prop. \ref{AnextProduct}.)
By construction, the images $\Phi'_i$, $\Psi_i'$, or $\Phi'$, $\Psi'$, under the projections 
$W(D^+)[1/\xi]\to W(O_{C_i})[1/\xi_i]$ satisfy
\[
\Phi'_i=h'_i\cdot \Psi'_i\cdot \phi(h'_i)^{-1}
\]
with $h'_i\in \CG'(W(O_{C_i}))$.
Since
\[
W(D^+)[1/\xi]\into \prod_i (W(O_{C_i})[1/\xi_i]).
\]
this gives $\Phi'=h'\cdot \Psi'\cdot \phi(h')^{-1}$, for $h'=(h'_i)_i$ in $\CG'(W(D^+))=\prod_i \CG'(W(O_{C_i}))$, and  the claim follows. 

It remains to show that the framings $(i_{r_i})_i$ extend to a framing of $\CP(\Phi)$
which lifts the framing $i'_r$ of $\CP(\Phi')$ given by our $\Spa(D, D^+)$-point of $\CMint_{\CG'}$.
We can understand framings as isomorphisms of corresponding $G$-bundles $\CE$, resp.  $G'$-bundles $\CE'$ over 
the Fargues-Fontaine  curve. Denote by $\CE^b$, resp. $\CE^{b'}$, the $G$-bundle, resp. $G'$-bundle,  over  $X_{FF,\Spd(k)}$
given by $b$, resp. $b'$. Then we have isomorphisms  of $G$-bundles on   $X_{FF, \Spa(D, D^+)}$, resp. $X_{FF, \Spa(C_i, O_{C_i})}$,
\[
g': \CE'\iso \CE^{b'} \times_{\Spd(k)} \Spa(D, D^+), \quad g_i: \CE_i\iso \CE^b\times_{\Spd(k)} \Spa(C_i, O_{C_i}) ,
\]
and we would like to find
\[
g: \CE\iso \CE^{b} \times_{\Spd(k)} \Spa(D, D^+)
\]
which lifts $g'$ and projects to $g_i$ by $\Spa(C_i, O_{C_i})\to \Spa(D, D^+)$, for all $i\in I$.
(Here, we denote these framings again by $g'$, $g_i$, hopefully this does not introduce confusion.)
 We first show that the $G$-bundle $\CE$ over $X_{FF, \Spd(D, D^+)}$ has all its geometric fibers isomorphic to $\CE^b$: The $G$-bundle $\CE$ gives $f: \Spa(D,D^+)\to {\rm Bun}_G$ and the desired statement follows if we show that $f(|\Spa(D,D^+)|)= \{b\}\subset |{\rm Bun}_G|\cong B(G)$. 

Let $T=\Spa(D, D^+)$. For $i\in I$, the projection $(D, D^+)\to (C_i, O_{C_i})$ gives $\Spa(C_i, O_{C_i})\to T$. These combine to give an injection
\[
I\into |T|.
\]
This identifies the discrete set $I$ with a dense subspace 
of $|T|$ and of $\pi_0(|T|)$. As in the proof of  \cite[Prop. 1.5]{Gl}, we see that 
\[
\pi_0(|T|)\simeq \beta I
\]
 is the Stone-\v Cech compactification of the discrete set $I$. 
 
Consider the composition $f': \Spa(D, D^+)\to {\rm Bun}_{G'}$ of $f$ with the natural map ${\rm Bun}_G\to {\rm Bun}_{G'}$ induced by $G\to G'$. The construction of $\CE$ gives that $f'(|T|)=\{b'\}\subset B(G')$ and that 
$f(|\Spa(C_i,O_{C_i})|)=\{b\}$, for all $i\in I$. As above, the points $|\Spa(C_i,O_{C_i})|$, $i\in I$, are dense in $|T|=|\Spa(D,D^+)|$. The result will follow if we establish that any two points of $|{\rm Bun}_G|\cong B(G)$ which map to the same $b'\in |{\rm Bun}_{G'}|\cong B(G')$ and lie in
the same connected component of $|{\rm Bun}_G|$ as $b$, are equal to each other. 
To see this, we use the  commutative diagram 
\begin{equation*}\label{BGconstant}
\begin{aligned}
 \xymatrix{
       & |{\rm Bun}_Z|\cong B(Z)  \ar[r] \ar[d]^{\kappa_Z} &  |{\rm Bun}_G|\cong B(G)\,\ar[r] \ar[d]^{\kappa_G} & |{\rm Bun}_{G'}|\cong B(G')\ar[d]^{\kappa_{G'}} \\
     0\ar[r]   &X_*(Z)_\Gamma\ar[r]  &  \pi_1(G)_\Gamma\,\ar[r]&  \pi_1(G')_\Gamma\ar[r] & 0.
        }
        \end{aligned}
\end{equation*}
Under our assumptions, the second row is exact. The vertical arrows $\kappa_Z$, $\kappa_G$, $\kappa_{G'}$ are the Kottwitz invariant maps which are locally constant by \cite[Thm III.2.7]{FS}. Also, $\kappa_Z$ is a bijection $B(Z)\xrightarrow{\sim} X_*(Z)_\Gamma$, the map $B(G)\to B(G')$ in the top row is surjective and its fibers are identified with $X_*(Z)_\Gamma$, see \cite[Prop. 4.10]{KotIsoII}, cf. \cite[Lem. III.2.10 and the comment below that lemma]{FS}. The result now follows.

For  a perfectoid space $S$ over $k$ and a $G$-torsor $\CE$ over the FF curve $X_{FF, S}$ which has  all geometric fibers isomorphic to $\CE^b$,
  we can
consider the torsors over $S$ under $\wt G_b=\und{\Aut}(\CE^b)$, resp. $\wt G'_{b'}=\und{\Aut}(\CE'^{b'})$, 
\[
Q_S:=\underline{\rm Isom}_{\rm fil}(\CE, \CE^b\times_{\Spd(k)} S), \quad Q'_S:=\underline{\rm Isom}_{\rm fil}(\CE', \CE^{b'}\times_{\Spd(k)} S) ,
\]
cf. \cite[Thm. III.0.2(v)]{FS}. Note here that there is a HN filtration on $\CE$ and $\CE'$, see \cite[proof of Prop. III. 5.3]{FS},
and we ask the isomorphisms to respect the filtrations. This preservation  is automatic over a point $S=\Spa(C, O_C)$, and 
also for $\CE'$ over $S=\Spa(D, D^+)$ since all the fibers have the same HN polygon given by $b'$, cf.  \cite[Thm. II.2.19]{FS}.
We have a map of torsors $Q_S\to Q'_S $ covering the homomorphism $\wt G_b\to \wt G'_{b'}$
By \cite[proof of Prop. III.5.3]{FS}, these are $\wt G_{b}$-,  resp. $\wt G'_{b'}$-torsors which are trivial 
pro-\'etale locally on $S$. 

In the description of $\wt G_b$ and $\wt G'_{b'}$ given by \cite[Prop. III.5.1]{FS}, 
we see that $\wt G_b\to \wt G'_{b'}$ induces an isomorphism $\wt G^{>0}_b\iso \wt G'^{>0}_{b'}$.
Indeed, the central torus $Z$ acts trivially on ${\rm Lie}(G)$ with the adjoint action and so 
for $\lambda> 0$, we have
\[
({\rm Lie}(G)\otimes_{\BQ_p}\br\BQ_p, {\rm Ad}(b)\sigma)^{-\lambda}=({\rm Lie}(G')\otimes_{\BQ_p}\br\BQ_p, {\rm Ad}(b')\sigma)^{-\lambda}
\]
with the notations as in loc. cit..
It follows by loc. cit. that the kernel of $\wt G_b\to \wt G'_{b'}$ is $\underline {Z(\BQ_p)}$. 
This is also the kernel of
$
\underline{G_b(\BQ_p)}\to \underline{G'_{b'}(\BQ_p)}.
$

Recall we let $T=\Spa(D, D^+)$. 
By pulling back the map of torsors $f: Q_T\to Q'_T$ over $T$ 
along
\[
q: T\to Q'_T
\]
given by $g$, we obtain a $\underline {Z(\BQ_p)}$-torsor $\CQ= T\times_{Q'_T}Q_T\to T$. 
Since the perfectoid space $T$ is  strictly totally disconnected, the $\underline {Z(\BQ_p)}$-torsor
$\CQ$ is trivial, cf. \cite[Lem. III.2.6]{FS} (see also the argument on top of loc. cit. p. 90.). 
The framings corresponding to $g_i$, $i\in I$, give $\Spa(C_i, O_{C_i})$-points $q_i$
of $\CQ$. We can think of $(q_i)_{i\in I}$ as giving a section of $\CQ$ over $I$. 

On the other hand,  we have a morphism $\wt y: \CQ\to \CM^{\rm int}_{\CG, b, \mu}$ which 
lifts 
the $T$-point $y$ of $\CM^{\rm int}_{\CG', b', \mu'}$, i.e. it fits in a commutative
diagram
\begin{equation}\label{CDQ}
\begin{aligned}
 \xymatrix{
       \CQ \ar[r]^{\wt y } \ar[d] &  \CM^{\rm int}_{\CG, b, \mu}    \ar[d]^{f}\\
        T\ar[r]^{y\ \ }  &  \CM^{\rm int}_{\CG', b', \mu'} .
        }
        \end{aligned}
\end{equation}
The morphism $\wt y$ is obtained by combining the $\CG$-shtuka 
$\CP(\Phi)$ with the framing provided by the universal point of $\CQ$.
We claim that there is an extension of the given section of $\CQ$ over $I$ to 
a section of  $\CQ$ over $T$. This would give the desired lift $g$ of the framing $g'$ and finish the proof.

 Since $\CQ$ is the  trivial torsor, such a section is given by a continuous function $|T|\to Z(\BQ_p)$ extending the given function $I\to Z(\BQ_p)$. To construct it, we use  a compactness argument. 
Consider the composition
\[
\omega: |\CQ|\xrightarrow{|\wt y|} |\CM^{\rm int}_{\CG, b, \mu}|\xrightarrow{\kappa_G\circ {\rm sp}} C_\CG=\Omega_G/\pi_0(\CG).
\]
Note that
\[
\omega(z\cdot q)=\kappa_G(z)+\omega(q)
\]
for $z\in Z(\BQ_p)$. 

Now observe that, since the points $x_i$ in our construction lie in 
$\CM^{\rm int, \tau}_{\CG, b, \mu}(\Spa(C_i, O_{C_i}))$,  the corresponding points $q_i$ of $\CQ$ satisfy $\omega(q_i)=\tau$, for all $i\in I$. Set $\mathfrak T:=\omega^{-1}(\{\tau\})\subset |\CQ|$. This is a quasi-compact subset 
of $|\CQ|$ since $|\CQ|\simeq |T|\times Z(\BQ_p)$, where 
$|T|$ is quasi-compact, and where the restriction $ Z(\BQ_p)\to C_\CG$ of $\kappa_G$
has compact fibers. The last fact follows because $\Omega_Z=Z(\BQ_p)/\CZ(\BZ_p)$ and $\Omega_Z\to \Omega_G\to C_\CG$ has finite kernel, so each fiber is given by a finite union of cosets of $\CZ(\BZ_p)$ in $Z(\BQ_p)$. 

Since $\pi_0(\mathfrak T)$ is compact, by the universal property of the Stone-\v Cech compactification
we see that the composed map $I\to \mathfrak T\to \pi_0(\mathfrak T)$, 
given by $i\mapsto q_i$, uniquely extends to a continuous map $\beta I\simeq \pi_0(|T|)\to \pi_0(\mathfrak T)\subset \pi_0(|T|)\times Z(\BQ_p)$.
This  corresponds to the desired continuous extension $|T|\to \pi_0(|T|)\to Z(\BQ_p)$.
As above, this produces a section $T\to \CQ$ whose composition with $\wt y: \CQ\to \CM^{\rm int}_{\CG, b, \mu}$ gives $T\to \CM^{\rm int}_{\CG, b, \mu}$. This factors through the open and closed $\CM^{\rm int, \tau}_{\CG, b, \mu}\hookrightarrow  \CM^{\rm int}_{\CG, b, \mu}$ and provides the desired $x: T\to \CM^{\rm int, \tau}_{\CG, b, \mu}$.
\end{proof}

 \section{The case of trivial $\mu_\ad$}\label{s:strivmu}
In this section we prove Theorems   \ref{MainThm} and \ref{thmRepgoal} in the case when $\mu_\ad=1$ is trivial. Note that when  $\mu_\ad=1$ is trivial,  $B(G, \mu^{-1})$ consists of the unique basic element contained in it. In this section, $b$ denotes a representative of this unique element. 
\subsection{The case of trivial $\mu$}
Let first $\CG$ be a parahoric for $G$. It follows from the definition that $\BM^\loc_{\CG,1}=\Spec (\BZ_p)$. 
Furthermore, we have
\begin{equation}
\CM^{\rm int}_{\CG, 1, 1}\simeq \underline{G(\BQ_p)/\CG(\BZ_p)}
\end{equation}
over $\Spd(\br\BZ_p)$, cf. \cite[Prop. 25.2.1]{Schber} (in loc.~cit. only the case of a torus group is considered but the proof is valid in the general case). It follows that $\CM^{\rm int}_{\CG, 1, 1}$ is representable by the formal scheme
\[
\sM_{\CG, 1,1}=\coprod_{G(\BQ_p)/\CG(\BZ_p)} \Spf (\br\BZ_p) ,
\]
and that the formal completion at a $k$-point is isomorphic to $\br\BZ_p$. This proves Theorems \ref{MainThm} and \ref{thmRepgoal} in the case when $\CG$ is a parahoric group scheme. The case of a general quasi-parahoric follows from Theorem \ref{varyGthm}, Proposition \ref{neutralIso} and \eqref{locmodGtoG0}. 

\subsection{The general case}
The case when $\mu_\ad=1$ is reduced to the case of the adjoint group, as follows. Choose an extension $\CG\to\CG'$ to quasi-parahorics of the natural morphism $G\to G':=G_\ad$.  Consider the corresponding morphism of $v$-sheaves 
\[
\CM^{\rm int}_{\CG, b, \mu}\to \CM^{\rm int}_{\CG', 1, 1}. 
\]
We now use the following functorialities:
\begin{itemize}
\item there is an isomorphism $f: \CM^{\rm int,\tau}_{\CG, b, \mu}\xrightarrow{\sim} \CM^{\rm int,\tau'}_{\CG', 1, 1}\times_{\Spd(\br\BZ_p)}\Spd(O_{\br E}),$ for each $\tau\in C_\CG$, cf. Theorem \ref{thm451}. Hence the representability of $\CM^{\rm int}_{\CG', 1, 1}$ implies the representability of $\CM^{\rm int}_{\CG, b, \mu}$. 
\item there is an isomorphism $\wh f:  \CM^{\rm int}_{\CG, b, \mu /x}\xrightarrow{\sim}  \CM^{\rm int}_{\CG', 1, 1 /x'}\times_{\Spd(\br\BZ_p)}\Spd(O_{\br E}),$ for each $x\in \CM^{\rm int}_{\CG, b, \mu}(\Spd k)$, cf. Proposition \ref{adIso}. 
\item there is an isomorphism $\BM^v_{\CG, \mu}\xrightarrow{\sim} \BM^v_{\CG', 1}\times_{\Spd(\BZ_p)}\Spd(O_E)$, cf. \cite[Prop. 21.5.1]{Schber} (attributed in loc.~cit. to J.~Louren\c co).
\end{itemize}
Hence the case $(\CG', 1,1)$ implies the case $(\CG, b, \mu)$. 

\section{The  Hodge type case}\label{s:cruc}

 We now give the first non-banal cases when we can show the isomorphism of Theorem \ref{thmRepgoal}
 and show representability of the formal completions. Let $(G, b, \mu)$ be a local Shimura datum.  Also, throughout this section, we let $\CG=\CG_{\bf x}$ be a stabilizer Bruhat-Tits group scheme for a corresponding point  $\bf x$ in the extended building. We denote by $\CG^o$ the corresponding parahoric. We assume we are in case (A), so in particular $p>2$. 
 
 \subsection{The crucial Hodge type case}

 Let $\iota: (G, \mu)\hookrightarrow (\GL_h, \mu_d)$ be a Hodge embedding. We assume 
 that there is a $\BZ_p$-lattice $\Lambda\subset \BQ_p^h$ such that  the homomorphism $\iota$ extends to a closed immersion
 $
 \iota: \CG\hookrightarrow \GL(\Lambda).
$
 We then say that $\iota: (\CG,\mu)\hook (\GL(\La),\mu_d)$ is an \emph{integral Hodge embedding}.
Then we also have
\[
\CG(\br\BZ_p)=G(\br\BQ_p)\cap \iota^{-1}(\GL(\Lambda\otimes_{\BZ_p}\br\BZ_p)).
\]
By \cite[Prop. 3.6.2]{PRglsv}, this gives a closed immersion
\[
\iota: \CM^{\rm int}_{\CG, b, \mu}\hookrightarrow ({\mathcal M}^{\rm int}_{\GL(\Lambda), \iota(b),\mu_d})_O:=
{\mathcal M}^{\rm int}_{\GL(\Lambda), \iota(b),\mu_d}\times_{\Spd(W(k))}\Spd(O) ,
 \]
   where $O=O_{\br E}$. 
 The formal completion $ \CM^{\rm int}_{\CG, b, \mu /x_0}$ at the base point $x_0$ fits in a cartesian diagram 
 \begin{equation}\label{fibered0}
   \begin{aligned}
   \xymatrix{
     \CM^{\rm int}_{\CG, b, \mu /x_0}    \ar[r]^{\hat\iota\ \ \ \ \ \ } \ar[d] & ( {\mathcal M}^{\rm int}_{\GL(\Lambda), \iota(b),\mu_d /\iota(x_0)})_O \ar[d] \\
    \CM^{\rm int}_{\CG, b, \mu} \ar[r]^{\iota\ \ \ \ \ \ } & {({\mathcal M}^{\rm int}_{\GL(\Lambda), \iota(b),\mu_d}})_O .}
      \end{aligned}
    \end{equation}
 Hence, we also have a closed immersion
  \[
\hat\iota:  \CM^{\rm int}_{\CG, b, \mu /x_0}\hookrightarrow
 {\mathcal M}^{\rm int}_{\GL(\Lambda), \iota(b),\mu_d /\iota(x_0)}\times_{\Spd(W(k))}\Spd(O).
 \]
 Here, by \cite[Thm. 25.1.2]{Schber}   (relating ${\mathcal M}^{\rm int}_{\GL(\Lambda), \iota(b),\mu_d}$ to the RZ-space of EL-type $\CM_\BX$ of $p$-divisible groups of dimension $d$ and height $h$)  and the definition of formal completion, we have 
 \begin{equation}\label{identVersal}
 {\mathcal M}^{\rm int}_{\GL(\Lambda), \iota(b),\mu_d /\iota(x_0)}\times_{\Spd(W(k))}\Spd(O)\simeq \Spd(R\otimes_{W(k)} O),
 \end{equation}
 where $R\simeq W(k)\lps x_1,\ldots , x_m\rps$, for $m=d(h-d)$. The point $\iota(x_0)$ gives a $p$-divisible group $\sG_0$ of height $h$ over $k$ and $R$ is naturally identified with the universal formal deformation ring of $\sG_0$.

 \begin{proposition}
 There is a commutative
 diagram    of morphisms of smooth rigid analytic spaces over $\Sp(\breve E)$,
  \begin{equation}\label{fibered}
  \begin{aligned}
   \xymatrix{
       ( \CM^{\rm int}_{\CG, b, \mu /x_0})_\eta    \ar[r]^{\hat\iota_\eta} \ar[d]_{\pi_\CG} & \Spf(R)^{\rm rig}_{\breve E}
      \ar[d]^{\pi_{\GL_h}} \\
    X_\mu^{\rm an}\ar[r]^{\iota^{\rm an}} & {\rm Gr}(d, h)^{\rm an}_{\breve E} .
    }
        \end{aligned}
    \end{equation}
  The vertical maps are the restrictions of the \'etale period maps to the tubular neighborhoods. The bottom horizontal map is the analytification of the Zariski closed immersion
    \[
    X_\mu\hookrightarrow {\rm Gr}(d, h)_{E} ,
    \]
    which is obtained from   $\iota: G\hookrightarrow \GL_h$. The top horizontal map is a closed immersion.
    \end{proposition}
    
    Here 
   $\Spf(R)^{\rm rig}$ is the rigid analytic generic fiber of $\Spf(R)$ (in Berthelot's sense).
    For simplicity, we use the same symbol for the $v$-sheaf $( \CM^{\rm int}_{\CG, b, \mu /x_0})_\eta$,
    and for the smooth rigid analytic space over $\Sp(\br E)$ that represents it.

    \begin{proof}
   Consider the commutative diagram of $v$-sheaves
    \begin{equation}\label{fibered2}
  \begin{aligned}
   \xymatrix{
       (\CM^{\rm int}_{\CG, b, \mu})_\eta   \ar[r]^{\iota\ \ \ \ \ \ \ \ \ \ \ \ \ } \ar[d]_{\pi_\CG} &  \CM^{\rm int}_{\GL_h, \iota(b), \mu_d}\times_{\Spd(\BZ_p)}\Spd(\br E)
      \ar[d]^{\pi_{\GL_h}} \\
    X_\mu^\diamondsuit\ar[r]^{\iota\ \ \ \ \ \ \ \ \ \ \ \ \ } & {\rm Gr}(d, h)^\diamondsuit\times_{\Spd(\BQ_p)}\Spd(\br E).
    }
        \end{aligned}
    \end{equation}
  In this, the two horizontal maps are closed immersions. In particular, the bottom horizontal map is represented by a Zariski closed immersion. The two vertical period maps are given
  by the \'etale period maps (on the left side we combine the period maps on the components, as in Theorem \ref{quasiGeneric2}).  
   Their images are the corresponding open admissible sets
    and we have
    \[
    X_\mu^{\rm adm}=({\rm Gr}(d, h)\times_{\Spd(\BQ_p)}\Spd(\br E))^{\rm adm}\cap X_\mu, 
    \]
    cf. \cite[proof of Prop. 3.1.1]{PRglsv}. Note that $( \CM^{\rm int}_{\CG, b, \mu /x_0})_\eta\hookrightarrow (\CM^{\rm int}_{\CG, b, \mu})_\eta$ is an open immersion by the argument of \cite[Prop. 4.22]{Gl}, cf. \cite[Lem. 2.31 and its proof]{Gl21}. Hence, all the $v$-sheaves in (\ref{fibered}) are representable by smooth rigid analytic spaces. By full-faithfulness \cite[Prop. 10.2.3]{Schber} the top horizontal map is also representable by a morphism 
    of rigid analytic spaces and the result  follows.
    \end{proof}

 \begin{theorem}\label{thmRep}
Let $(G, b, \mu)$ be a local Shimura datum and let $\CG$ be a  quasi-parahoric  stabilizer group scheme for $G$.  Suppose $p>2$ and assume that there exists a Hodge embedding $\iota: (G, \mu)\hookrightarrow (\GL_h, \mu_d)$
 and a $\BZ_p$-lattice $\Lambda\subset \BQ_p^h$ such that:
 
 \begin{itemize} 
  \item[a)]  The homomorphism $\iota$ extends to a closed immersion
   $
 \iota: \CG\hookrightarrow \GL(\Lambda).
$

 \item[b)] The Zariski closure $\overline {X}_\mu$ of $X_\mu\subset {\rm Gr}(d, h)_{E}$ in ${\rm Gr}(d, \Lambda)_{O_E}$
 is normal. 
 
 \item[c)] The image $\iota(G)$ contains the scalars.

 \end{itemize}
  
\noindent  Then  the following statements hold:
  
  \smallskip
  
  1)
 $\BM^{\rm loc}_{\CG,\mu} \simeq \overline {X}_\mu$. 
 
 \smallskip
 
 2) For any $x\in \CM^{\rm int}_{\CG, b, \mu}(\Spd(k))$ 
  there is $y\in \BM^{\rm loc}_{\CG,\mu}(k)$,   such that 
 \begin{equation}\label{eqFormaliso}
 \CM^{\rm int}_{\CG, b, \mu /x} \simeq ( \BM^{\rm loc}_{\CG,\mu /y})^\diam,
 \end{equation}
provided that $\iota: (\CG,\mu)\hook (\GL(\La),\mu_d)$ is a \emph{very good integral Hodge embedding} in the 
 sense of \cite{KPZ}, see below.  In the above, the orbit $\CG(k)\cdot y$ is equal to $\ell(x)$, cf. (\ref{Lmap}).

\end{theorem}

\begin{proof} By section \ref{ss:basepoint}, we may assume that $x=x_0$ is the base point. For simplicity, we write 
\[
\Ms= \CM^{\rm int}_{\CG, b, \mu /x_0}, \qquad \BM^{\rm loc}=\BM^{\rm loc}_{\CG,\mu}.
\]

 We first show  1), i.e., $\BM^{\rm loc} \simeq \overline {X}_\mu$.
 By \cite[Thm. 21.2.1]{Schber}, $\iota$ induces a closed immersion of $v$-sheaves over $\Spd(\BZ_p)$,
\[
\iota: {\rm Gr}_{\CG, \Spd(\BZ_p)}\to {\rm Gr}_{\GL_h, \Spd(\BZ_p)} .
\]
 This gives a closed immersion of $v$-sheaves over $\Spd(O_E)$,
\[
\BM^v\to {\rm Gr}(d, h)_{O_E}^\diam .
\]
 Indeed, $\BM^v=\BM^v_{\CG, \mu}$ is, by definition, the $v$-sheaf closure of $X_\mu^\diam$ in 
 \[
 {\rm Gr}(d, h)_{O_E}^\diam\into {\rm Gr}_{\GL_h}\times_{\Spd(\BZ_p)}\Spd( O_E).
 \]
 (See \cite[2.1]{AGLR} for a discussion of $v$-sheaf closures.)
Since  $\BM^v=(\BM^{\rm loc})^\diam$ and $\BM^{\rm loc}$ is normal, we deduce from  full-faithfulness (using \cite[Prop. 18.4.1]{Schber} and formal GAGA) that the closed immersion
$\BM^v\to {\rm Gr}(d, h)_{O_E}^\diam$ is represented by a morphism of schemes,
\[
\BM^{\rm loc}\to {\rm Gr}(d, h)_{O_E}.
\]
This morphism extends $X_\mu\hookrightarrow {\rm Gr}(d, h)_E$ on the generic fibers.
Now  $(\overline X_\mu)^\diam$
is topologically flat (\cite[Lem. 3.4.2]{PRglsv}) and so it is  the $v$-sheaf closure of the generic fiber $X_\mu^\diam$
in ${\rm Gr}(d, h)_{O_E}^\diam$. We obtain an isomorphism 
\[
\BM^v=(\BM^{\rm loc})^\diam\xrightarrow{\sim} (\overline X_\mu)^\diam\hookrightarrow {\rm Gr}(d, h)_{O_E}^\diam.
\]
Since $\BM^{\rm loc}$ is normal and we are also assuming that $\overline X_\mu$ is normal, the isomorphism 
$(\BM^{\rm loc})^\diam\xrightarrow{\sim} (\overline X_\mu)^\diam$ is obtained from an isomorphism 
$\BM^{\rm loc}\xrightarrow{\sim} \overline X_\mu$ of $O_E$-schemes, as claimed.

We now proceed to show  2). 
By \cite[Prop. 1.3.2]{KisinJAMS}, there is a finite set of tensors $\{s_a\}_a\subset \Lambda^\otimes$, such that $\CG$ is the scheme theoretic pointwise stabilizer of $s_a$,
\[
\CG=\{g \in \GL(\Lambda)\ |\ g\cdot s_a=s_a, \forall a\} .
\]
By the Tannakian formalism applied to $\CG\to \GL(\Lambda)$, each $\CG$-torsor over an affine $\BZ_p$-scheme $\Spec(A)$ gives a finite projective $A$-module $N$ of rank equal to ${\rm rank}_{\BZ_p}(\Lambda)$ with tensors 
$s_a(N)\in N^{\otimes}$.

Consider the universal $\CG$-shtuka over $\Ms$; its push-out by $\CG\to \GL(\Lambda)$ is the shtuka 
which is obtained from the Breuil-Kisin-Fargues module of the universal $p$-divisible group, as in the proof of \cite[Theorem 25.1.2]{Schber}. By specializing to the base point $x_0$ and using the Tannakian formalism we see that the $\CG$-shtuka over $\Spd(k)$ that corresponds to $x_0$ equips  the Dieudonn\'e module $\BD:=\BD(\sG_0)(\br \BZ_p)$ of $\sG_0$ with Frobenius invariant tensors $s_{a, 0}\in \BD^{\otimes}$.  

Write $R_G$ for the completion of $\BM^{\rm loc}\otimes_{O_E}O$ at the corresponding point $y_0$ with orbit $\ell(x_0)$ and denote by ${\frak m}_G$ the maximal ideal of $R_G$. The $O$-algebra $R_G$ is normal and is a quotient of   the formal completion $R_E\simeq R\otimes_WO$ of the local ring of ${\rm Gr}(d,\La)_{O}$ at $y_0$. Note that, since $\bar X_\mu\simeq \BM^{\rm loc}$, we see that $\BM^{\rm loc}\otimes_{O_E}O$
is identified with the reduced Zariski closure of a $G$-orbit $G\cdot y$ in ${\rm Gr}(d, h)_O$, where $y$ is an $F$-point that corresponds to a filtration 
induced by a $G$-valued cocharacter $\mu_y$ conjugate to $\mu^{-1}$. Here $F$ is a finite extension of $\breve\BQ_p$. 

Let us now briefly review certain constructions of \cite{KP}, \cite{KZhou}, \cite{KPZ}, and in particular the notion of a very good integral Hodge embedding; we will use the notations of these papers. We continue to assume $p>2$ and that (a), (b), (c) are satisfied. 

Set $M=\La\otimes_{\BZ_p}R_E$ and denote by $\hat I_{R_E}M\subset M_1\subset M$ the unique $\widehat{W}(R_E)$-submodule corresponding to the universal $R_E$-valued point of the Grassmannian. This gives a ``Dieudonn\'e pair" $(M, M_1)$. We will  denote by $(M_{R_G}, M_{R_G,1})$ the Dieudonn\'e pair of $\widehat{W}(R_G)$-modules which is obtained by base changing $(M, M_1)$ along $R_E\to R_G$. We set
\[
\widetilde M_{R_G,1}:={\rm Im}(\phi^*M_{R_G, 1}\to \phi^*M_{R_G}).
\] 
Then $\widetilde M_{R_G,1}$ is a finite free $\widehat{W}(R_G)$-module  and
\[
\widetilde M_{R_G,1}[1/p]=(\phi^*M_{R_G})[1/p].
\]

 By the argument of \cite[Cor. 3.2.11]{KP} (which extends to this situation using also the main result of \cite{An}, cf. Remark \ref{ANRem}), the tensors 
 \[
 \tilde s_a:=s_a\otimes 1=\phi^*(s_a\otimes 1)\in \La^{\otimes}\otimes_{\BZ_p} \widehat {W}(R_G)=(\phi^*M_{R_G})^{\otimes}\subset (\phi^*M_{R_G})^{\otimes}[1/p]=\widetilde M_{R_G,1}^{\otimes}[1/p]
 \]
lie in $\widetilde M_{R_G,1}^{\otimes}$ and the scheme
\[
\CT=\underline{\rm Isom}_{(\tilde s_a), (s_a)}(\widetilde M_{R_G,1}, \La\otimes_{\BZ_p}\widehat {W}(R_G))
\]
of isomorphisms that preserve the tensors is a (trivial) $\CG$-torsor over $\widehat {W}(R_G)$.
The scheme $\CT$ is independent of the choice of the set of tensors $(s_a)\subset \La^\otimes$ that cut out $\CG$.

Set ${\frak a}_G={\frak m}_G^2+\pi_ER_G\subset R_G$. There is a canonical isomorphism
\begin{equation}\label{canonicaliso}
c=c_{y_0}:  \widetilde M_{0,1}\otimes_{W(k)}\widehat W(R_G/{\frak a}_G)\xrightarrow{\sim} \widetilde M_{R_G, 1}\otimes_{\widehat W(R_G)}\widehat W(R_G/{\frak a}_G)
\end{equation}
(\cite[Lem. 3.1.9]{KP},  \cite[\S 5.2.1]{KPZ}).
Here, $(M_0, M_{0,1})$ is the Dieudonn\'e pair of $W(k)$-modules obtained from $(M_{R_G}, M_{R_G, 1})$ by the base change given by $y_0^*: R_G\to k$
and $\widetilde M_{0,1}={\rm Im}(\phi^*M_{0,1}\to \phi^*M_0)$. 

We say that the  tensors $(\tilde s_a)$  are preserved by $c$ if we have $ c(\tilde s_{a,0}\otimes 1)=\tilde s_a\otimes 1$, for all $a$. 
Then the isomorphism $c$ uniquely descends to an isomorphism of $\CG$-torsors
\[
c^\CG: \CT_0\otimes_{W(k)}\widehat W(R_G/{\frak a}_G)\xrightarrow{\sim} \CT\otimes_{\widehat W(R_\CG)}\widehat W(R_G/{\frak a}_G).
\]

We say that the integral Hodge embedding  $(\CG,\mu)\hook (\GL(\La),\mu_d)$ is \emph{very good at $y_0$}, if there are tensors $(s_a)\subset \La^{\otimes }$ cutting out $\CG$ in $\GL(\La)$ such that $(\tilde s_a)$ are  preserved by $c=c_{y_0}$, see \cite[\S 5.2]{KPZ}.
 This is equivalent to asking that $c$ descends to an isomorphism of $\CG$-torsors $c^\CG$, as above, and this property does not depend on the choice of $(s_a)$.
We say that $(\CG,\mu)\hook (\GL(\La), \mu_d)$ is \emph{very good}, if it is very good at all $y\in \BM^{\rm loc}_{\CG,\mu}(k)$. 

The following statement encapsulates a construction of \cite{KP}, as extended in \cite{KZhou}, \cite{KPZ}.

\begin{proposition}\label{adapted}  Let $(\CG,\mu)\hook (\GL(\La),\mu_d)$  be an integral Hodge embedding 
 which satisfies (a), (b), (c) of Theorem \ref{thmRep} and which is very good\footnote{this condition was erroneously omitted in \cite{KP}, and the first version of \cite{KZhou}; the fact that it is satisfied in all cases of interest is shown in \cite{KPZ}.} at a point $y_0\in \BM^{\rm loc}_{\CG,\mu}(k)$ corresponding to 
the base point $x_0\in \CM^{\rm int}_{\CG, b, \mu}(\Spd(k))$. Let $R_E$ be the formal completion  of the local ring of ${\rm Gr}(d,\La)_{O}$ at $y_0$.  There exists a $p$-divisible group $\sG^{\rm univ}$ over $R_E$ 
 which is a versal formal deformation of $\sG_0$ and which satisfies the following property:

Let $K/\br E$ be a finite field extension  and $\wt x^*: R_E\to O_K$
a local $O$-algebra homomorphism satisfying:

\begin{itemize}
\item[a)] The filtration on $\BD\otimes_{\breve\BZ_p} K$ which corresponds to 
the deformation $\sG_{\wt x}:={\wt x}^*(\sG^{\rm univ})$ of the $p$-divisible group $\sG_0$ to $O_K$
given by base change of $\sG^{\rm univ}$ via $\wt x^*$,
 is induced by a $G$-valued cocharacter which is $G$-conjugate to $\mu^{-1}$;

\item[b)] The tensors $s_{\alpha, 0}\in\BD^{\otimes}$ correspond to tensors $s_{\alpha, \et}\in T_p\sG^{\vee\otimes}_{\wt x}$
under the $p$-adic (\'etale-crystalline) comparison isomorphism.
\end{itemize}
Then the homomorphism $\wt x^*: R_E\to O_K$ factors through the quotient $R_E\to R_G$.
\end{proposition}

\begin{proof} Let us give a very quick overview to orient the reader. 
The construction of $\sG^{\rm univ}$ uses Zink's theory of displays. By the argument of \cite[\S 3.2.12]{KP} 
  we obtain  a ``Dieudonn\'e display triple" $(M_{R_G}, M_{R_G, 1}, \Psi_{R_G})$, as in loc. cit. (see Remark \ref{ANRem} below and the proof of \cite[Prop. 3.2.7]{KZhou} for the removal of Condition 3.2.2 in \cite{KP}). This triple gives, by \cite[Lem. 3.1.5]{KP}, a Dieudonn\'e display over $R_G$ and then one constructs also  a versal display $(M_{R}, M_{R, 1}, \Psi_{R})$
over $R_E$ which lifts $(M_{R_G}, M_{R_G, 1}, \Psi_{R_G})$. The display $(M_{R}, M_{R, 1}, \Psi_{R})$
gives the desired versal formal deformation $\sG^{\rm univ}$ of $\sG_0$ over $\Spf(R_E)$. 

We can then see that the $p$-divisible group $\sG^{\rm univ}$, as constructed as above, has the property stated in the proposition. For this first note that, by   \cite[Prop. 3.2.7]{KZhou}, a deformation $\sG_{\wt x}$ of $\sG_0$ which satisfies (a) and (b) is a ``$(\CG, \mu^{-1})$-adapted lifting" in the sense of \cite[Def. 3.2.4]{KZhou}. Then the claim follows from \cite[Prop. 3.3.4]{KZhou} (the same statement under  additional tameness hypotheses on $G$ also appears in \cite{KP}, see \cite[Prop. 3.3.13]{KP}).  
\end{proof}

The versal formal deformation $\sG^{\rm univ}$ over $R_E$ of the $p$-divisible group $\sG_0$ induces an identification of $\Spd(R_E)$ with  $ {\mathcal M}^{\rm int}_{\GL(\Lambda), \iota(b),\mu_d /\iota(x_0)}\times_{\Spd(W)}\Spd(O)$.
We will now use Proposition \ref{adapted} to check that the rigid analytic closed subspace
\[
\Spf(R_G)^{\rm rig}\hookrightarrow \Spf(R_E)^{\rm rig}=\Spf(R)^{\rm rig}_{\br E}
\]
agrees with the rigid analytic closed subspace $(\Ms)_\eta$, 
under the  identification induced by diagram \eqref{fibered}.

 Let $K/\br E$ be a finite extension.
We will compare  $(\Ms)_\eta(K)$ and $\Spf(R_G)^{\rm rig}(K)$ as subsets of $\Spf(R)^{\rm rig}_{\br E}(K)$.

\begin{proposition}\label{includeK}
For all finite field extensions $K/\br E$, there is an inclusion 
\[
(\Ms)_\eta(K)\subset \Spf(R_G)^{\rm rig}(K) .
\]
\end{proposition}

\begin{proof}
A $K$-point $x$ of $(\Ms)_\eta$ gives, after composing with $(\Ms)_\eta\hookrightarrow \Spf(R_E)^{\rm rig}$,
a formal scheme morphism $\wt x: \Spf(O_K)\to \Spf(R_E)$. The point $x$ also gives 
 a crystalline representation
 \begin{equation}\label{Gre}
\rho_x: \Gal(\bar K/K)\to \GL(T_p(\sG_{\wt x})^\vee)
 \end{equation}
 on the linear dual of the Tate module
of the $p$-divisible group $\sG_{\wt x}$ obtained  from $\sG^{\rm univ}$ as above, by pulling back by $\wt x$. 
Take $C=\widehat {\bar K}$ which supports a $\Gal(\bar K/K)$-action and consider the corresponding point
$\bar x: \Spa(C, O_C)\to \Spd(K)\to (\Ms)_\eta$ which gives a $(\CG, \mu)$-shtuka 
$(\sP_{\bar x}, \phi_{\sP_{\bar x}})$ with framing.
 By Proposition \ref{prop431} we obtain a functor
\[
\BP_{\bar x}: {\rm Rep}_{\BZ_p}(\CG)\to \text{ $\und\BZ_p$-Loc $(\Spa(C, O_C))\cong \BZ_p$-mod }
\]
such that $\BP_{\bar x}(\Lambda)=T_p(\sG_{\wt x})^\vee$. For $\gamma\in \Gal(\bar K/K)$, there
is an isomorphism ${\mathbb R}_x(\gamma): \BP_{\bar x}\simeq \BP_{\bar x\cdot \gamma}= \BP_{\bar x}$ giving the Galois action \eqref{Gre}, i.e., $\BR_{x}(\gamma)(\Lambda)=\rho_x(\gamma)$.

The $\CG$-invariant tensors $s_a\in \Lambda^\otimes$ give, by applying the functor $\BP$,
corresponding tensors $s_{a,\et}\in T_p(\sG_{\wt x})^{\vee\otimes}$;
these are invariant under the action of $\Gal(\bar K/K)$ through $\rho_x$.
As in \S \ref{classPoints}, we see that the Galois representation $\rho_x$ factors 
\[
\rho_x: \Gal(\bar K/K)\to \CG^o_{\beta}(\BZ_p)\subset \CG_{\beta}(\BZ_p)\subset G(\BQ_p),
\]
where $\bar\beta\in \Pi_\CG$ is such that $x$ lies in the  component ${\rm Sh}_{\CG_\beta(\BZ_p)}(G, b, \mu)$ of $(\CM^{\rm int}_{\CG,b,\mu})_\eta$, cf. Theorem \ref{quasiGeneric2}.
Note that 
\[
\CG_{\beta}\otimes_{\BZ_p}\br\BZ_p\simeq \CG\otimes_{\BZ_p}\br\BZ_p.
\]
Choose a uniformizer $\pi_K$ of $K$ and a collection of roots $\pi_K^{1/p^n}$ in $\bar K\subset C$. We can now see, as in \cite[\S 3.5]{PRglsv}, \cite[(3.3.3)]{KP}, that the Breuil-Kisin module $\fkM_{\wt x}$ of $\sG_{\wt x}$ obtained using these choices can be refined to a $\CG^o_{\beta}$-Breuil-Kisin module.
By definition, this is a $\CG^o_{\beta}$-torsor over $\Spec(W(k)\lps u\rps)$ with meromorphic Frobenius structure, see loc. cit.. Refined here is meant in the sense that the push-out by $\CG^o_{\beta}\otimes_{\BZ_p}W(k)\to \CG\otimes_{\BZ_p}W(k)\to \GL(\Lambda\otimes_{\BZ_p}W(k))$ recovers the $\GL(\Lambda\otimes_{\BZ_p}W(k))$-torsor over $\Spec(W(k)\lps u\rps)$ which corresponds to    $\fkM_{\wt x}$.  By pushing out the torsor by $\CG^o_{\beta}\otimes_{\BZ_p}W(k)\to \CG\otimes_{\BZ_p}W(k)$, we obtain a $\CG$-Breuil-Kisin module $\sP_{\rm BK}$ over $\Spec(W(k)\lps u\rps)$.
The tensors $s_a$ induce tensors $\wt s_a\in \fkM_{\wt x}^\otimes$ over $W(k)\lps u\rps$, comp. \cite[(3.3.3)]{KP}.
(Note that after base changing by $W(k)\lps u\rps\to W(O_C)$ given by $(\pi_K^{1/p^n})_n$, $\sP_{\rm BK}$ gives a $\CG$-Breuil-Kisin-Fargues module. 
By the proof of  \cite[Prop. 3.5.1]{PRglsv}, the restriction of this $\CG$-Breuil-Kisin-Fargues module from $\Spec(W(O_C))$ to $\CY_{[0,\infty)}(C, O_C)$ is isomorphic to the initial
$(\CG, \mu)$-shtuka $(\sP_{\bar x}, \phi_{\sP_{\bar x}})$ which corresponds to the point $x$.)  
Recall that the tensors $s_a$ also induce $s_{a, 0}\in \BD^\otimes$,
where, as above, $\BD=\BD(\sG_0)(\br\BZ_p)$ is the Dieudonn\'e module of the special fiber $\sG_0=\sG_{\wt x}\otimes_{O_K}k$. The existence of $\wt s_a$ above, together with the compatibility properties of the Breuil-Kisin functor (see for example, \cite[Thm. 3.3.2, Prop. 3.3.8]{KP}), implies that $s_{a,\et}$ and $s_{a, 0}$ correspond under the comparison
isomorphism between $p$-adic-\'etale and crystalline cohomology. It also implies that the Hodge filtration on $\BD\otimes_{\br\BZ_p}K$ which corresponds to the deformation $\sG_{\wt x}$
is induced by a $G$-cocharacter. The above compatibility of the $\CG$-Breuil-Kisin module $\sP_{\rm BK}$ 
with the initial $(\CG, \mu)$-shtuka $(\sP_{\bar x}, \phi_{\sP_{\bar x}})$ implies that this cocharacter is conjugate to $\mu^{-1}$.

The conditions a) and b) of Proposition \ref{adapted}  are now satisfied for $\wt x^*$ (the $p$-divisible
group $\sG_{\wt x}$   is   ``$(\CG, \mu^{-1})$-adapted", in the terminology of \cite{KZhou}).
It follows that the morphism $\wt x^*: R_E\to O_K$ inducing
$\sG_{\wt x}$ factors through $R_G\to O_K$. This gives a $K$-valued point of 
$\Spf(R_G)^{\rm rig}$ and hence $x$ belongs to $\Spf(R_G)^{\rm rig}(K)$.
\end{proof}

Now  $(\Ms)_\eta$ and $\Spf(R_G)^{\rm rig}$ are both smooth rigid analytic spaces over $\Sp(\br E)$ of the same dimension, both closed in $\Spf(R)^{\rm rig}_{\br E}$. Since $R_G$ is normal, $\Spf(R_G)^{\rm rig}$ is connected, cf. \cite[Lem. 7.3.5]{deJongCrys}. From Proposition \ref{includeK}  it now follows that, under our identifications,
\begin{equation}\label{genMsRG}
(\Ms)_\eta=\Spf(R_G)^{\rm rig}.
\end{equation}
Let us now complete the proof.
Both $\Spd(R_G)\hookrightarrow \Spd(R_E)$
and $\Ms\into \Spd(R_E)$ are closed immersions of $v$-sheaves.
By Proposition \ref{neutralIso} (b) and \cite[Prop. 3.4.9]{PRglsv}, $(\Ms)_\eta$ is ``topologically flat", i.e.
$|(\Ms)_\eta|$ is dense in $|\Ms|$. Similarly, $\Spd(R_G)$ is topologically flat by \cite[Lem. 3.4.2]{PRglsv}. 
Since by (\ref{genMsRG}), $|(\Ms)_\eta|=|\Spd(R_G)_\eta|$, we deduce that
$|\Ms|=|\Spd(R_G)|$. Hence, by \cite[Lem. 17.4.1]{Schber}, comp. also 
\cite[Prop. 12.15 (iii)]{Sch-Diam}, we have 
\[
\Ms=\Spd(R_G),
\]
under the identification of $\Spd(R_E)$ with  $ {\mathcal M}^{\rm int}_{\GL(\Lambda), \iota(b),\mu_d /\iota(x_0)}\times_{\Spd(W)}\Spd(O)$ induced by $\sG^{\rm univ}$.
 (Compare to the argument in the proof of Prop. \ref{adIso}, case (i).)
\end{proof}

 \begin{remark}\label{ANRem}
 In \cite[\S\S3.2, 3.3]{KP}, there is always the blanket  assumption that  $G$ splits over a tamely ramified extension of $\BQ_p$. This is because 
 essential use is made of ``purity" of $\CG^o$-torsors over $\Spec(W(k)[[u]])\setminus \{(u,p)\}$ which was shown in \cite[Prop. 1.4.3]{KP} under this  tameness hypothesis  on $G$. This purity is now proven in all cases by Ansch\"utz \cite{An}. With this additional ingredient, the proofs go through, see \cite{KZhou}, \cite{KPZ}. 
 Note that when $G$ is    essentially tamely ramified, then in \cite[\S 5]{PRglsv} there is a simpler proof of the purity statement in question. 
\end{remark}

\subsection{Reduction to the (very) good Hodge type case}\label{s:exHE}

 \begin{proposition}\label{goal}
  Let $(G, b, \mu)$ be a local Shimura datum of abelian type and $p> 2$. Let $\CG$ be a quasi-parahoric group scheme
  of $G$ and let ${\bf x}$ be a point in the building $\sB^e(G,\BQ_p)$ such that $\CG^\circ=\CG_{\bf x}^\circ\subset \CG\subset \CG_{\bf x}$. Write
  \begin{equation}\label{eqProduct}
  (G_\ad, \mu_\ad)\simeq \prod_{i=1}^m ({\rm Res}_{F_i/\BQ_p} H_i, \mu_i), 
  \end{equation}
  with all $H_i$ absolutely simple. Assume that all $\mu_i$ are non-trivial. 
  Then there exists a central lift $(G_1, b_1, \mu_1)$ of Hodge type for $(G, b, \mu)$ 
 as in Definition \ref{def:abtype} and a point ${\bf x}_1\in \sB^e(G_1,\BQ_p)$ with stabilizer group scheme $\CG_1=\CG_{{\bf x_1},1}$,  with the  following properties: 
 \begin{itemize}

 \item[1)]  The images of the points ${\bf x}_1$ and ${\bf x}$ in $\sB(G_\ad,\BQ_p)$ define the same parahoric subgroup scheme  of $G_\ad$.
 
  \item[2)] $E_1=E_\ad$.
  
  \item[3)] There is a integral Hodge embedding $(\CG_1,\mu_1)\hook (\GL(\La), \mu_d)$
 which satisfies (a), (b), (c) of Theorem \ref{thmRep} and which is very good. 
\end{itemize}
  \end{proposition}
  
  \begin{proof} 
  We  first observe that we may assume that $G_\ad$ is $\BQ_p$-simple. Indeed,  we can construct first $(G_1,  \mu_1 )$ and then $\CG_1$ and 
  the integral Hodge embedding, by taking the product over the factors in (\ref{eqProduct}). From now on let $G_\ad=\Res_{F/\BQ_p}(H)$, with $H$ absolutely simple over $F$.

  We index by $I$  the set $\{\phi_v\}_{v\in I}$ of embeddings $\phi_v: F\into \bar\BQ_p$.
  For each $v\in I$, we set $\sD_v$ for the Dynkin diagram of $H\otimes_{F, \phi_v}\bar\BQ_p$.
  Also write 
  \[
  \mu: {\BG_m}_{/\bar\BQ_p}\to ({\rm Res}_{F/\BQ_p}H)\otimes_{\BQ_p}{\bar\BQ_p}=\prod_{v\in I} H\otimes_{F, \phi_v}\bar\BQ_p,\quad 
  \mu=(\mu_v)_{v\in I}.
  \]
By the arguments in \cite[1.3.6, 1.3.8]{DeligneCorvallis}, the condition that $(G_\ad, \mu_\ad)$ is of abelian type implies 
that for each $v\in I$, $\mu_v$ is either trivial, or it corresponds to a node $s_v$ in the Dynkin diagram
$\sD_v$ which is among the ``encircled nodes" of the diagrams displayed in the table \cite[1.3.9]{DeligneCorvallis}.
In fact, $(G_\ad, \mu_\ad)$ is one of  types ${\bf A}$, ${\bf B}$, ${\bf C}$, ${\bf D}^{\BR}$, ${\bf D}^{\BH}$ as described
in \cite[2.3.8]{DeligneCorvallis}, see also \cite[\S 3, Annexe]{Serre}. For example, if $n>4$, ${\bf D}^{\BH}$ means that each (non-trivial) factor is $(D_n, \omega^\vee_n)$ or $(D_n, \omega^\vee_{n-1})$ and ${\bf D}^{\BR}$ means that each factor is $(D_n, \omega^\vee_1)$, cf. \cite[\S 5]{HLR}.  Note that, by \cite[\S 3, Cor. 2]{Serre}, since $\mu$ is not trivial, $H$ cannot be a trialitarian form $D_4^{(3)}$ or $D^{(6)}_4$. Recall that we assume $p>2$. Hence, it follows that $H$ splits over a tamely ramified extension of $F$.

Denote by $I_c$ the set of $v\in I$ for which $\mu_v$ is trivial and set $I_{nc}=I\setminus I_c$.
Set 
\[
\sD=\sqcup_{v\in I}\sD_v .
\]
 The Galois group 
$\Gamma=\Gal(\bar\BQ_p/\BQ_p)$ acts on $\sD$ compatibly with its action on the set of embeddings $I$.

By our assumption $\mu$ is not trivial, so $I_{nc}\neq\emptyset$.
 Let $S\subset \sD$  be a Galois stable subset such
that, for each $v\in I_{nc}$, $S\cap \sD_v=\{\underline s_v\}$, a node in $ \sD_v$ from the underlined nodes in the 
table \cite[1.3.9]{DeligneCorvallis} for $(\sD_v, s_v)$.\footnote{The table in \cite[1.3.9]{DeligneCorvallis} requires a small correction which is explained in the translation of the paper
on Milne's webpage: on the first line (case ${\bf A}$) all nodes should be underlined if $p=1$. But in our construction we do not use these omitted 
nodes anyway.}
A node $s\in S\cap\sD_v$ gives an irreducible representation 
\[
\rho_{s}: G_{{\rm sc}, \bar\BQ_p}\to  H_{\rm sc}\otimes_{F, \phi_v}\bar\BQ_p\to \GL(V(s))
\]
over $\bar\BQ_p$. Here,  $ G_{\rm sc}$ and $ H_{\rm sc}$  are the simply connected covers. 

If $(G_\ad,\mu_\ad)$ is of type ${\bf A}_n$ we can choose $\underline s_v$, for $v\in I_{nc}$, to be one of the two endpoint nodes; these 
correspond to the standard representation of $H_{\rm sc}\otimes_{F, \phi_v}\bar\BQ_p\simeq {\rm SL}_{n+1}$ and its dual. In fact, if $H\simeq {\rm PGL}_m(D)$, with $D$ a division algebra over $F$, then we can and will choose every $\underline s_v$, for $v\in I_{nc}$, to be the node that corresponds to the standard representation. Then, $\underline s_v$ is the node
corresponding to the standard representation for all $v$.
If $(G_\ad,\mu_\ad)$ is of type ${\bf D}^\BH_n$, then $\underline s_v$, for each $v$, is the node given by the simple 
endpoint of the type $D_n$ Dynkin diagram.

In all cases, consider
\[
 \bigoplus_{s\in S} V(s)^{\oplus n}
\]
which, for sufficiently divisible $n$, gives a representation $V$ of $ G_{\rm sc}$ defined over $\BQ_p$;
denote by $V_s$ the unique irreducible factor of $V_{\bar\BQ_p}$ isomorphic to $V(s)^{\oplus n}$;
the Galois group $\Gamma$ permutes the factors $V_s$ by an action compatible with its action on $S\subset \sD$.
There is a finite \'etale $\BQ_p$-algebra $K_S$ such that ${\rm Hom}_{\BQ_p}(K_S, \bar\BQ_p)\simeq S$ as $\Gamma$-sets
and so the decomposition
\[
V\otimes_{\BQ_p}\bar\BQ_p\simeq \bigoplus_{s\in S} V_s
\]
is induced by a corresponding $K_S$-module structure on $V$ which is such that
on $V_s$, $K_S$ acts via the map $K_S\to\bar\BQ_p$ that corresponds to $s$.
Note that since $S\to I$ is $\Gamma$-equivariant, $K_S$ is naturally an $F$-algebra.
It is a product $K_S=\prod_j K_j$ of field extensions $K_j$ of $F$ which are all at most tamely ramified over $F$.
The units $K_S^\times=\prod_j K^\times_j$ are the points of a torus $T'={\rm Res}_{F/\BQ_p} \prod_j T_j$
with $T_j/F$ splitting over the tame extension $K_j/F$. We have $K^\times_S\subset \GL(V)$, i.e. $T'\into \GL(V)$,
and this centralizes the map $G_{\rm sc}\to \GL(V)$.
Denote by $ G'_{\rm sc}$ the quotient of $ G_{\rm sc}$ which acts faithfully on $V$; then
\[
 G_{\rm sc}\to G'_{\rm sc}\into \GL(V).
\]
The group $ G'_{\rm sc}$ is the restriction of scalars $G'_{\rm sc}={\rm Res}_{F/\BQ_p} H'_{\rm sc}$, with $H'_{\rm sc}$ a quotient of $ H_{\rm sc}$.
In fact, we see that $ H'_{\rm sc}= H_{\rm sc}$ in all cases, except for groups of type ${\bf D}$. In the case of type ${\bf D}$,
the kernel of $H_{\rm sc}\to  H_{\rm sc}'$ is either trivial in type ${\bf D}_n^\BR$, or of order $2$ in type ${\bf D}^\BH_n$.
Set
\[
G_1:=G'_{\rm sc}\cdot T'\into \GL(V).
\]
We have $G_1={\rm Res}_{F/\BQ_p} H_1$, with  $H_1= H_{\rm sc}'\cdot (\prod_j T_j)$. As in \cite{DeligneCorvallis}, we see that $\mu: {\BG_m}_{/\bar\BQ_p}\to (G_{\ad})_{\bar\BQ_p}$ lifts to a 
fractional cocharacter $\mu'$ of $ G'_{{\rm sc}, \bar\BQ_p}$. The 
 composition of this fractional cocharacter with  $G'_{{\rm sc}, \bar\BQ_p}\into \GL(V_{\bar\BQ_p})\to \GL(V_s)$ is  trivial when $s$ maps to $v\in I_{c}$, in which case we set $V_s=V_s(0)$.  When $s$ maps to $v\in I_{nc}$, then this composition  has exactly two distinct weights $r_v$ and $r_v-1$, with $r_v$ as indicated in the table \cite[1.3.9]{DeligneCorvallis}, so that 
\[
V_s=V_s({r_v-1})\oplus V_s({r_v}) .
\]
We can now consider the direct sum decomposition of $V_{\bar\BQ_p}$ into two summands: one, $V_{\bar\BQ_p}(0)$, given as the sum of all $V_s(0)$ and all $V_s({r_v-1})$,
and the other, $V_{\bar\BQ_p}(1)$, given as the sum of all $V_s(r_v)$. This decomposition defines a cocharacter $\mu_1$ of $\GL(V)_{\bar\BQ_p}$ 
acting by weights $0$ and $1$ on these two summands and $\mu_1$ factors through $(G_1)_{\bar\BQ_p}=(G'_{\rm sc}\cdot T')_{\bar\BQ_p}$. Then $G_1\hook \GL(V)$
gives a Hodge embedding for $(G_1,\mu_1)$ and $(G_{\ad, 1},\mu_{\ad, 1})\simeq (G_\ad, \mu_\ad)$.

We will see that the above construction produces a desired $(G_1,\mu_1)$ but we also need to explain how to choose ${\bf x}_1$ and hence the corresponding stabilizer group scheme
$\CG_1$ for $G_1$. We let ${\bf x}_\ad$ be the point corresponding to ${\bf x}$ in the building $\sB(G,\BQ_p)=\sB^e(G_\ad, \BQ_p)$. Choose a ``nearby" point ${\bf x}_\ad\in 
\sB(G,\BQ_p)$ which is generic in its facet and is such that the parahoric group schemes of $G_\ad$ for ${\bf x}'_\ad$ and ${\bf x}_\ad$ coincide, i.e. have the 
same $\breve\BZ_p$-points.  Now lift ${\bf x}'_\ad$ to ${\bf x}_1$ under the canonical map
 $\sB^e(G_1, \BQ_p)\to \sB(G_\ad, \BQ_p)$. This defines the stabilizer group scheme $\CG_1=\CG_{{\bf x}_1,1}$.

We now verify that the pair $(G_1,\mu_1)$ and the point ${\bf x}_1\in \sB^e(G_1, \BQ_p)$ satisfy the desired conditions (1), (2) and (3).
Conditions (1) and (2) follow immediately from the construction, and it remains to explain Condition (3). This calls for the construction of a suitable integral
Hodge embedding $\iota: (\CG_1,\mu_1)\hook (\GL(\La),\mu_d)$; we will explain how this can be done by choosing the lattice $\La$ in a representation
obtained from $V$ above.

If $(G_\ad,\mu_\ad)$ is of type ${\bf A}$ and $H\simeq {\rm PGL}_m(D)$, with $D$ a division algebra over $F$, then for the above choices,
$V$ is isomorphic to a direct sum of copies of the representation of $\SL_m(D)$ given by the action on $D^m$, considered as a $\BQ_p$-vector 
space; this extends to a representation of $G_1=\Res_{F/\BQ_p}\GL_m(D)$.

If $(G_\ad,\mu_\ad)$ is of type 
${\bf D}_n^\BH$, then for the above choices, $G_1=\Res_{F/\BQ_p}H_1$, with $H_1$ the neutral component of
an orthogonal similitude group over $F$ and $V$ is isomorphic to the restriction of scalars of its standard representation.  
More precisely, as in \cite[\S 5.3.8]{PZ}, \cite{T}, we see that $H_1$ and $V$ are as in one of the following cases:

 a) There is an $F$-vector space $V'\simeq F^{2n}$ and a perfect symmetric $F$-bilinear $h': V'\times V'\to F$ such that $H_1=\GO^+(V', h')$
 where, as usual, for an $F$-algebra $R$,
 \[
 \GO(V', h')(R)=\{ \GL_{R}(V'\otimes_F R)\ |\ h'( gw, gw')=c(g)\psi(w, w'), \ c(g)\in R^\times\},
 \]
 and ${}^+$ signifies taking the neutral component. 
 Then the representation space $V$ is a direct sum of copies of $V'$ considered as an $\BQ_p$-vector space by restriction of scalars.

 b) There is a (left) $D$-module $W\simeq D^n$ for a division quaternion $F$-algebra $D$ and a non-degenerate anti-hermitian form $\phi: W\times W\to D$ 
 for the main involution on $D$, such that $H_1={\rm GU}^+(W, \phi)$, where ${\rm GU}(W, \phi)$ is a unitary similitude group defined as follows: Consider the alternating $F$-bilinear form $\psi: W\times W\to F$ given by 
 \[
 \psi(w_1, w_2)={\rm Tr}^0(\phi(w_1, w_2)).
 \]
where ${\rm Tr}^0: D\to F$ is  the reduced trace
(cf.    \cite[\S 5.3.8]{PZ}, \cite[Prop. A.53]{R-Z}, applied to $n=1$.) 
For an $F$-algebra $R$, 
\[
{\rm GU}(W, \phi)(R)=\{ \GL_{D\otimes_F R}(W)\ |\ \psi( gw, gw')=c(g)\psi(w, w'), \ c(g)\in R^\times\}.
\]
The representation space $V$ is a direct sum of copies of $W$ considered as an $\BQ_p$-vector space of dimension $4n$ by restriction of scalars. 

The existence of an integral Hodge embedding $\iota: (\CG_1,\mu_1)\hook (\GL(\La),\mu_d)$ satisfying (a), (b), and (c) of Theorem \ref{thmRep} and which is very good now follows
from \cite[\S 6]{KPZ}: see \cite[Thm. 6.1.1]{KPZ} (for ``non-exceptional cases'', see loc. cit. Remark 6.1.10) and \cite[Thm. 6.3.2]{KPZ} (for the remaining cases of type ${\mathbf A}$), and \cite[Thm. 6.2.3]{KPZ} (for type ${\mathbf D}^\BH_n$).  
In all cases, $\Lambda=\oplus_{i=1}^r\Lambda_i\subset V^{\oplus r}$ is a lattice in a direct sum of $r$ copies of the representation $V$ as given above, for some $r\geq 1$. The lattice $\Lambda$ is obtained by summing up lattices $\Lambda_i\subset V$ in a suitable lattice chain in $V$.
\end{proof}

\section{Proofs of Theorems \ref{thmRepgoal} and \ref{MainThm}}\label{s:proofs}
In this section we prove our main theorems. We treat the cases (A) and (B) separately. 

\subsection{Proof of Theorem \ref{thmRepgoal} in case (A)} 
Let $(G, b, \mu)$ be of abelian type and of type (A), and let $\CG$ be quasi-parahoric.   Write $(G_\ad,b_\ad, \mu_\ad)= 
\prod_i({\rm Res}_{F_i/\BQ_p}H_i, b_i, \mu_i)$, where $H_i$ is absolutely simple.  We split $(G_\ad,b_\ad, \mu_\ad)$ into the product of two factors: in the first factor we lump together all components   $({\rm Res}_{F_i/\BQ_p}H_i, b_i, \mu_i)$ where $\mu_i$ is trivial, and in the second factor we lump together all components  $({\rm Res}_{F_i/\BQ_p}H_i, b_i, \mu_i)$ where $\mu_i$ is non-trivial.  Let us first assume that the first factor is trivial.

 Write $\CG_{\bf x}^0\subset \CG\subset \CG_{\bf x}$, with $\CG_{\bf x}(\br\BZ_p)$ the stabilizer in $G(\br\BQ_p)$ of a point $\bf x$ in the 
extended building $\sB^e(G, \BQ_p)$ of $G(\BQ_p)$. Using Proposition \ref{goal} we construct $(G_1, b_1,\mu_1)$ of Hodge type
with $\CG_1$ a stabilizer group scheme of $G_1$ such that there is a (very good) integral Hodge embedding
\[
(\CG_1,\mu_1)\hookrightarrow (\GL(\Lambda),\mu_d)
\]
satisfying all the conditions of Theorem \ref{thmRep}. By the construction, we have a group scheme homomorphism
$\CG_1\to \CG'_\ad:=\CG_{\ad, {\bf x}'_\ad}$ extending $
 G_1\to G_\ad$. We similarly have  $\CG\to \CG_\ad:=\CG_{\ad, {\bf x}_\ad}$, giving $G\to G_\ad$. Note $\CG^\circ_{\ad}=\CG'^\circ_{\ad}$. 

 First note the natural isomorphisms
\begin{equation}\label{passLM}
 \BM^{\rm loc}_{\CG, \mu}\simeq  \BM^{\rm loc}_{\CG_\ad, \mu_\ad}
\times_{\Spec(O_\ad)}\Spec(O)\simeq  \BM^{\rm loc}_{\CG_1, \mu_1}
\times_{\Spec(O_1)}\Spec(O)
\end{equation}
obtained from \cite[Prop. 21.5.1]{Schber}, and $ \BM^{\rm loc}_{\CG^\circ_1, \mu_1}\simeq  \BM^{\rm loc}_{\CG_1, \mu_1}$, $ \BM^{\rm loc}_{\CG^\circ, \mu}\simeq  \BM^{\rm loc}_{\CG, \mu}$, 
from \cite[Prop. 21.4.3]{Schber}. These induce isomorphisms of corresponding formal completions.

 We now apply  Theorem \ref{thmRep} to 
 $(\CG_1, b_1, \mu_1)$ and obtain, for each  $x_1\in \CM^{\rm int}_{\CG_1, b_1, \mu_1}(\Spd(k))$, an isomorphism 
  \[
 \CM^{\rm int}_{\CG_1, b_1, \mu_1 /x_1} \simeq ( \BM^{\rm loc}_{\CG_1,\mu_1 /y_1})^\diam, 
 \]
 where $y_1$ is a corresponding point in  $\BM^{\rm loc}_{\CG_1,\mu_1}(k)\simeq \BM^{\rm loc}_{\CG^\circ_1,\mu_1}(k)$ whose $\CG_1(k)$-orbit $\ell(x_1)$ is well-defined and determined by $x_1$, see (\ref{Lmap}). It then follows from Proposition \ref{neutralIso} (b)
 that, for $x_1\in \CM^{\rm int}_{\CG^\circ_1, b_1, \mu_1}(\Spd(k))$, we also have  
 \[
 \CM^{\rm int}_{\CG^\circ_1, b_1, \mu_1 /x_1} \simeq ( \BM^{\rm loc}_{\CG^\circ_1,\mu_1 /y_1})^\diam.
 \]
 On the other hand, Proposition \ref{adIso} applied to the map $(\CG^\circ_1, b_1, \mu_1)\to (\CG'^\circ_\ad, b_\ad, \mu_\ad)$ induced by $\CG_1\to \CG'_\ad$,
  gives
  \[
 \CM^{\rm int}_{\CG^\circ_1, b_1, \mu_1 /x_1} \xrightarrow{\sim}   \CM^{\rm int}_{\CG'^\circ_\ad, b_\ad, \mu_\ad /x'_1}= \CM^{\rm int}_{\CG^\circ_\ad, b_\ad, \mu_\ad /x'_1},
 \]
 where $x'_1$ is the image of $x_1$. Here, the last equality follows from $\CG^\circ_{\ad}=\CG'^\circ_{\ad}$.
 Combining these we obtain  isomorphisms
 \begin{equation}\label{passad}
 \CM^{\rm int}_{\CG^\circ_\ad, b_\ad, \mu_\ad /x'_1}\simeq  (\BM^{\rm loc}_{\CG, \mu/y'_1})^\diam\simeq (\BM^{\rm loc}_{\CG^\circ_\ad, \mu_\ad /y_\ad})^\diam,
 \end{equation}
 for all points $x'_1$ of  $\CM^{\rm int}_{\CG^\circ_\ad, b_\ad, \mu_\ad}(\Spd(k))$ which are in the image
 of  $ \CM^{\rm int}_{\CG^\circ_1, b_1, \mu_1}(\Spd(k))$ under the natural map induced by $\CG^\circ_1\to \CG^\circ_\ad$. Here,  $y'_1$ and $y_\ad$ correspond to $y_1$ above under the isomorphisms given by (\ref{passLM}).
 By \S \ref{ss:adADLV}, this set of points of  $\CM^{\rm int}_{\CG^\circ_\ad, b_\ad, \mu_\ad}(\Spd(k))$ is the set of $\Spd(k)$-points of a union of ``components" 
 $\CM^{\rm int,\tau}_{\CG^\circ_\ad, b_\ad, \mu_\ad}$, for a certain set of $\tau$. Using Proposition \ref{actJ} applied to $G_\ad$  and the fact that the map $\ell$ of (\ref{Lmap}) is constant on each $J_b(\BQ_p)$-orbit, we see that
 the same conclusion, i.e. the isomorphism (\ref{passad}), follows for   points   in all components of $\CM^{\rm int}_{\CG^\circ_\ad, b_\ad, \mu_\ad}$ and, therefore, for all points in $\CM^{\rm int}_{\CG^\circ_\ad, b_\ad, \mu_\ad}(\Spd(k))$.
 We obtain that  Theorem \ref{thmRepgoal} holds for $(G_\ad, b_\ad,\mu_\ad)$ and all parahoric subgroups of $G_\ad$. Therefore, by the argument in the proof of Theorem \ref{varyGthm} (which uses Theorem \ref{quasiThm}), the result also holds for $(G_\ad, b_\ad,\mu_\ad)$ and all quasi-parahoric subgroups of $G_\ad$. In particular, it holds for $\CG_\ad$. Theorem \ref{thmRepgoal} for $(G,b,\mu)$, $\CG$ and $x\in \CM^{\rm int}_{\CG, b, \mu}(\Spd(k))$ now follows from the above
combined with Proposition \ref{adIso} applied to $(\CG, b, \mu)\to (\CG_\ad, b_\ad, \mu_\ad)$ and (\ref{passLM}).
 
 This concludes the proof of Theorem \ref{thmRepgoal} when the first factor of $(G_\ad,b_\ad, \mu_\ad)$ is trivial.

Now let us consider the general adjoint case. Using the compatibility of $\CM^{\rm int}_{\CG, b, \mu}$ and $\BM^v_{\CG, \mu}$ with products (cf. \S \ref{ss:inlocshim}), it suffices to consider each factor separately. The second factor has been treated above. For the first factor the assertion follows from section \ref{s:strivmu}. Now the passage from the adjoint case to the general case  follows by the same argument as above. 
\qed

\subsection{Proof of Theorem \ref{MainThm} in case (A)}  
Via the ad-isomorphism $G\to G_\ad$
it  follows using Theorem \ref{thm451}  that 
$\CM^{\rm int}_{\CG, b,\mu}$ is representable by a normal formal scheme locally formally of finite type over $\br O$, provided this representability holds for the adjoint group. As in the above proof, we may consider separately the case where $\mu_\ad$ is trivial and the case where all components of $\mu_\ad$ are non-trivial. The first case follows from section \ref{s:strivmu}. Let us consider the second case.  

As in the above proof, 
using Proposition \ref{goal} we construct $(G_1, b_1,\mu_1)$ of Hodge type
with $\CG_1$ a stabilizer group scheme such that there is a (very good) integral Hodge embedding
\[
(\CG_1,\mu_1)\hookrightarrow (\GL(\Lambda),\mu_d)
\]
satisfying all the conditions of Theorem \ref{thmRep}.
Then, using Theorem \ref{thmRep} we deduce that $\CM^{\rm int}_{\CG_1, b_1,\mu_1/x}$ is representable by the formal spectrum of a noetherian normal complete local ring, for any $x\in \CM^{\rm int}_{\CG_1, b_1,\mu_1}(\Spd(k))$. The   proof of \cite[Thm. 3.7.1]{PRglsv} now applies and we obtain that  
$\CM^{\rm int}_{\CG_1, b_1,\mu_1}$ is representable by a formal scheme which is normal and flat locally formally of finite type over $\br O$. 
(Note that \cite[Thm. 3.7.1]{PRglsv} is stated   for parahoric $\CG$. However, given the construction of the specialization map for quasi-parahorics as in \S\ref{def:spec} above, the argument extends in a straightforward fashion to our situation, in which $\CG_1$ is quasi-parahoric.
The argument in loc. cit. requires $\CG_1(\br\BZ_p)=\GL(\Lambda\otimes_{\BZ_p}\br\BZ_p)\cap G_1(\br\BQ_p)$ and this holds here since $\CG_1\hookrightarrow \GL(\Lambda)$ is a closed immersion.) Since the maps $G\to G_\ad$ and $G_1\to G_\ad$ are both ad-isomorphisms,
it now follows using Theorem \ref{thm451} twice,  together with Proposition \ref{actJ} applied to $G_\ad$, that 
$\CM^{\rm int, \tau}_{\CG, b,\mu}$, for each $\tau\in \Omega_G$, is representable by such a formal scheme.
It follows that $\CM^{\rm int}_{\CG, b,\mu}=\sqcup_\tau \CM^{\rm int, \tau}_{\CG, b,\mu}$
is representable by a formal scheme which is normal and flat locally formally of finite type over $\br O$.
\qed

\subsection{Proof of Theorems \ref{thmRepgoal}, \ref{MainThm}  in case (B)}

Recall that in case (B) we have $p=2$ and $G_\ad=\prod_{i=1}^m{\rm Res}_{F_i/\BQ_2} H_i$, 
with $H_i=B^\times_i/F_i^\times$, or $H_i={\rm PGSp}_{2n_i},$   or $\mu_i$ trivial. As in the proofs of Theorems \ref{MainThm}  and \ref{thmRepgoal} in the case (A),  we can easily reduce to 
the case that $m=1$ and assume $G_\ad= {\rm Res}_{F/\BQ_2} H$ with 
$H=B^\times/F^\times$, or $H={\rm PGSp}_{2n}$. In the first subcase, we take $G_1={\rm Res}_{F/\BQ_2} B^\times$. In the second subcase, we first set 
\[
\langle v, w\rangle={\rm Tr}_{F/\BQ_2}(v, w)
\]
where $(v, w)$ is the standard perfect alternating $F$-bilinear form on $F^{2n}$. Then we take
$G_1=J$, with the group $J$ defined by
\[
J(R)=\{g\in \GL_{2n}(F\otimes_{\BQ_p}R)\ |\ \langle gv, gw \rangle=c(g) \langle v, w\rangle, \ c(g)\in R^\times\}.
\]
In each subcase, we lift $\mu_\ad$ to a corresponding minuscule $\mu_1$.

The quasi-parahoric $\CG$ gives a corresponding  point $\bf x$ in the extended Bruhat-Tits building $\sB^e(G, \BQ_2)$, so that
$\CG^o=\CG^o_{\bf x}(\br\BZ_2)\subset \CG(\br\BZ_2)\subset \CG_{\bf x}(\br\BZ_2)$. 
Consider the stabilizer group scheme $\CG_1:=\CG_{1,{\bf x}_1}$ of $G_1$ which corresponds to a point ${\bf x}_1$ in $\sB^e(G_1,\BQ_2)$
that lifts the point ${\bf x}_\ad$ in the building $\sB(G_\ad, \BQ_2)$, as in the argument in the proof for the case (A) above.  The devissage results of \S\ref{s:parvs} (Theorem \ref{varyGthm} and its proof), allows us to replace $\CG_{1,{\bf x}_1}$ by the stabilizer group scheme of a point which is generic in the smallest facet that contains it, and still has the same  parahoric neutral component. Therefore, we can assume that ${\bf x}_1$ already has this ``genericity" property.  
   We can now see, using \cite[App. to Chapt. 3]{R-Z} and the standard explicit description of the buildings for these groups (\cite{BTcl}, \cite{BTcl2}), that there is a lattice  chain $(\CL)$, resp. a self-dual lattice  chain $(\CL)$, such that the stabilizer group scheme above is given as a scheme theoretic stabilizer of $(\CL)$ in $G_1$. 
The data $(G_1, b_1, \mu_1)$ together with the  lattice chain $(\CL)$ determine integral EL-, resp.  PEL-data $\CD$
as in \cite{R-Z}, see also \cite[Def. 24.3.3]{Schber}. 

Consider the corresponding   RZ formal scheme $\CM_\CD^{\rm naive}$ (as defined in \cite{R-Z}; the hypothesis $p\neq 2$ is not needed in this case). Then  $\CM_\CD^{\rm naive}$ is a formal scheme  
locally formally of finite type over $O_{\br E}$. Under our assumptions, $\CM^{\rm naive}_\CD$ has formal completions
at closed points which agree with those of the naive local model $\BM^{\rm naive}_\CD$ (\cite[Prop. 3.33]{R-Z}); recall that this follows by Grothendieck-Messing deformation theory which applies for $p=2$. Note that in general, $\BM^{\rm naive}_\CD$ and $\CM^{\rm naive}_\CD$ are not flat over 
$O_{\br E}$ (this accounts for the terminology ``naive".)

Set 
$
\CM_\CD:=\CM_\CD^{\rm flat}
$
 to be the closed
formal subscheme of $\CM^{\rm naive}_\CD$ given by chains of $2$-divisible groups which are
``$\BM^{\rm loc}_{\CG_1,\mu_1}$-admissible" in the sense of \cite[Lect. 24, 25]{Schber}. Note that under our assumptions,
$\BM^{\rm loc}_{\CG_1,\mu_1}$ is a closed subscheme of the ``naive" local model $\BM_{\CD}^{\rm naive}$
which is identified with the flat closure of its generic fiber; in particular, this flat closure has reduced special fiber. Indeed, in the case where $G_1={\rm Res}_{F/\BQ_2} B^\times$, this follows from \cite[Thm. 7.3]{PR-lm2} (based on G\"ortz \cite{G-lm1}). In the case where $G_1=J$, this follows from \cite[Thm. 12.4]{PR-lm2} (based on \cite{G-lm2,G-lm3}).
 (These two results also follow from the proof of the 
coherence conjecture by Zhu \cite{ZhuPR}.)  

By \cite[Cor. 25.1.3]{Schber}, we have
\[
(\CM_\CD)^\diam\simeq \CM^{\rm int}_{\CG_1, b_1, \mu_1}
\]
as $v$-sheaves over $\Spd(O_{\br E})$.  It follows that $\CM^{\rm int}_{\CG_1, b_1, \mu_1}$ is represented by the formal scheme  $\CM_\CD$. By its construction and the above discussion, we obtain that the formal scheme $\CM_\CD$ has formal completions
at closed points which agree with those of $\BM^\loc_{\CG_1,\mu_1}$; in particular, $\CM_\CD$ is normal and 
flat over $O_{\br E}$ since these properties hold for $\BM^\loc_{\CG_1,\mu_1}$.  Hence  the representability conjecture \ref{repconj} holds and  $\sM_{\CG_1, b_1, \mu_1}=\CM_\CD$. We can now pass from $(G_1, b_1, \mu_1)$ and $\CG_1$, to $(G, b, \mu)$
and $\CG$, by the same devissage as  in the proofs in case (A). This completes the proofs of 
Theorems \ref{thmRepgoal} and \ref{MainThm} in case (B).
\qed

\begin{remark}
Our strategy for the proofs of Theorems \ref{thmRepgoal} and \ref{MainThm} in case (B) could be applied
to more EL/PEL cases than the ones currently given,  provided that in the corresponding cases  the results of \cite[App. to Chapt. 3]{R-Z} can be
suitably modified  for $p=2$. An example is given by the unramified unitary group, cf. \cite[App. A]{RSZ-unit}.  On the other hand, it is also reasonable to expect
that the strategy 
of the proof in case (A) could be extended to $p=2$ (at least assuming  that $G$ is   essentially tamely ramified), if the constructions of \cite{KP}, \cite{KPZ} can be extended to cover 
$p=2$. Such an extension was given in \cite{MPKim} in the hyperspecial case, i.e. when $\CG$ is reductive over $\BZ_2$, and in \cite{Yang} in some parahoric cases.
\end{remark}

\end{document}